\documentclass[11pt, hidelinks]{article}

\usepackage[utf8]{inputenc}
\usepackage{lmodern}
\usepackage{geometry}
\usepackage{setspace}
\usepackage{amsmath, amssymb, amsthm, mathtools, bbm, bm, mathrsfs, relsize, centernot, cancel}

\usepackage{thmtools}
\usepackage{chngcntr}

\usepackage{graphicx}
\usepackage{subcaption}
\usepackage{caption}
\usepackage{tikz}
\usetikzlibrary{arrows.meta, positioning, decorations.pathreplacing, cd}
\usepackage{booktabs}
\usepackage{tabularx}
\usepackage{tablefootnote}

\usepackage{algpseudocode}
\usepackage{listings}

\usepackage{enumitem}
\usepackage{outlines}

\usepackage{xcolor}
\usepackage{xurl} 
\usepackage{hyperref} 

\newcommand{\cls}[1]{{\color{red} #1}} 

\usepackage{cleveref}

\newtheorem{thm}{Theorem}[section]
\newtheorem{propn}[thm]{Proposition}
\newtheorem{lemma}[thm]{Lemma}

\theoremstyle{remark}
\newtheorem{remark}[thm]{Remark}
\newtheorem{corlly}[thm]{Corollary}

\newtheorem{example}{Example}

\theoremstyle{definition}

\newtheorem{ass}[thm]{Assumption}

\allowdisplaybreaks

\hypersetup{
    colorlinks=true,
    linkcolor=red,
    citecolor=blue,
    urlcolor=blue,
    pdfborder={0 0 0}
}

\DeclareMathOperator{\sigmaf}{\mathscr{F}}

\DeclareMathOperator{\prob}{\mathbb{P}}
\DeclareMathOperator{\exptn}{\mathbb{E}}
\DeclareMathOperator*{\argmax}{arg\,max}

\newcommand{\bs}[1]{{\boldsymbol #1 }}
\newcommand{\asarrow}{\:\to_{\rm a.s.}\:}

\title{A stochastic optimization algorithm for revenue maximization in a service system with balking customers}
\author{S. A. Bodas, H. Honnappa, M. Mandjes, L. Ravner}
\date{\today}

\numberwithin{equation}{section}
\begin{document}
\maketitle

\begin{abstract}
\noindent This paper analyzes a service system modeled as a single-server queue, in which the service provider aims to dynamically maximize the expected revenue per unit of time. This is achieved by constructing a stochastic gradient descent algorithm that dynamically adjusts the price. A key feature of our modeling framework is that customers may choose to {\it balk} — that is, decide not to join — when facing high congestion. A notable strength of our approach is that the revenue-maximizing algorithm relies solely on information about {\it effective} arrivals, meaning that only the behavior of customers who choose not to balk is observable and used in decision-making. This results in an elaborate interplay between the pricing policy and the effective arrival process, yielding a non-standard state dependent queueing process. An important contribution of our work concerns a novel Infinitesimal Perturbation Analysis (IPA) procedure that is able to consistently estimate the stationary effective arrival rate. This is further leveraged to construct an iterative algorithm that converges, under mild regularity conditions, to the optimal price with provable asymptotic guarantees.

\medskip


\medskip

\noindent
SAB is with Korteweg-de Vries Institute for Mathematics, University of Amsterdam, Amsterdam, The Netherlands.
HH is with {Edwardson} School of Industrial Engineering, Purdue University, West Lafayette IN, United States. 
MM is with the Mathematical Institute, Leiden University,
The Netherlands; MM is also affiliated with (a)~Korteweg-de Vries Institute for Mathematics, University of Amsterdam, Amsterdam, The Netherlands, b)~E{\sc urandom}, Eindhoven University of Technology, Eindhoven, The Netherlands, (c)~Amsterdam Business School, Faculty of Economics and Business, University of Amsterdam, Amsterdam, The Netherlands.
LR is with Department of Statistics, University of Haifa,
Israel.

\medskip

\noindent SAB's and MM's  research has been funded by the European Union’s Horizon 2020 research and innovation programme under the Marie Sklodowska-Curie grant agreement no.\ 945045 \includegraphics[height=1em]{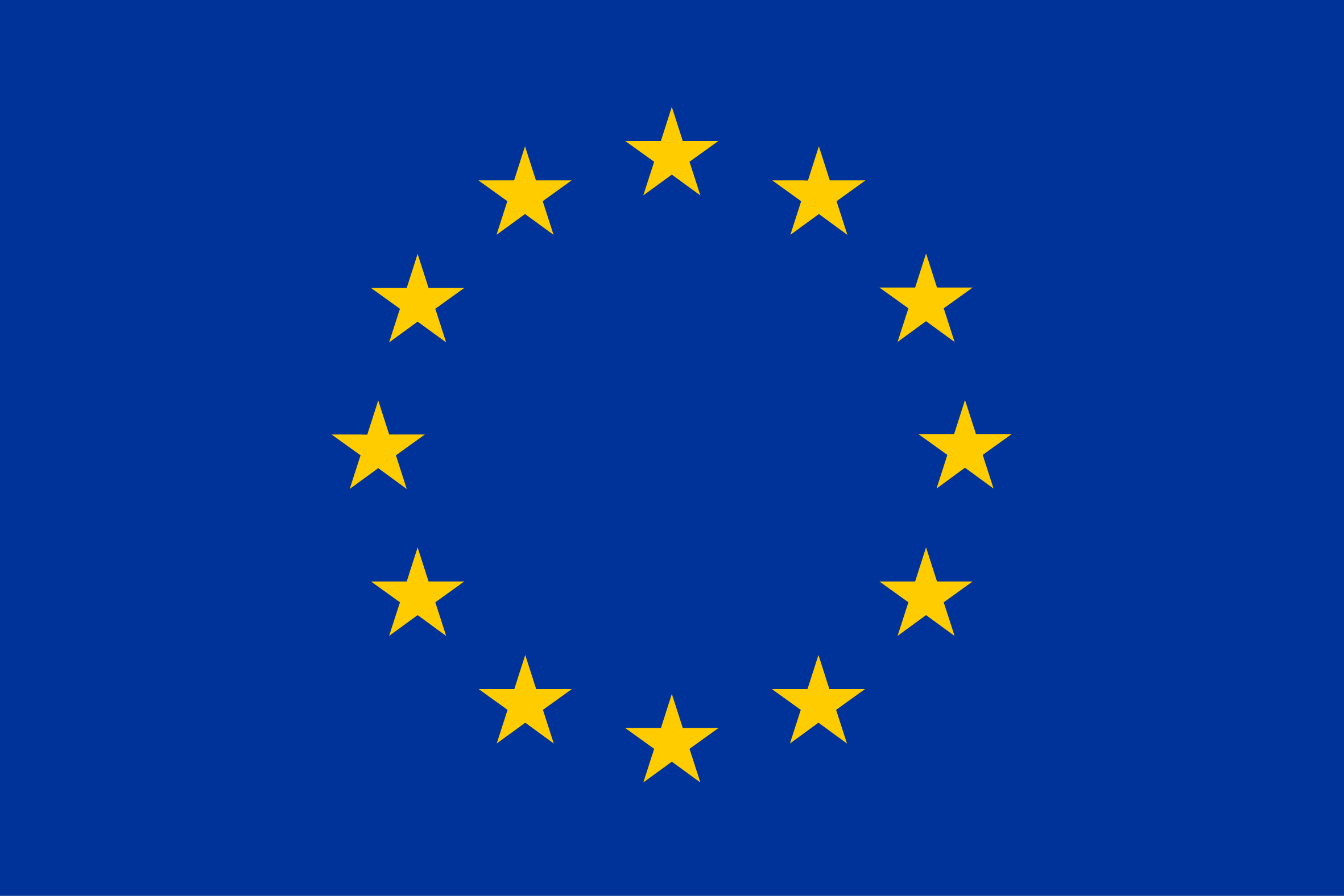}, and by the NWO Gravitation project {\sc Networks} under grant agreement no.\ 024.002.003.  LR's research was supported by the Israel Science Foundation (ISF), grant no.\ 1361/23. This work was partly conducted while HH was visiting Institute for Advanced Study, University of Amsterdam, the Netherlands. 
\end{abstract}

\newpage

\section{Introduction}

In this paper, we examine a service facility modeled as a single-server queue, where customers arrive in response to the price set by the provider. The service provider’s objective is to determine the optimal price, i.e., the price that maximizes the expected revenue per unit of time. Achieving this goal is particularly challenging, as it requires understanding the intricate relationship between pricing, customer demand, system congestion, and the resulting impact on revenue performance.

\medskip

{\it Offline and online revenue maximization~---} To put our work into perspective, we first provide a high-level categorization of the existing literature on revenue optimization. Broadly speaking, two dominant approaches can be identified, that one could refer to as the {\it offline} and {\it online} paradigms.
\begin{itemize}
    \item[$\circ$]
Offline revenue optimization is typically studied using a demand curve $d(p)$ that depends on the price $p$. The standard static optimization problem takes the form $\max_{p \geq 0} p\, d(p)$.
However, in the context of service systems, an additional layer of complexity arises: congestion can also reduce demand. 
Specifically, users may decide to balk if the (expected) delay exceeds their individual utility thresholds.
In stylized Markovian models, one can explicitly derive the revenue-maximizing price. A canonical example is Naor’s model \cite{Naor1969}, which features an observable queue and linear waiting costs. 
In such settings, both the {\it effective demand} and the underlying {\it effective arrival rate} $\lambda(p)$ depend on the price and the stationary queue-length distribution.
The problem of monopolistic pricing in congested systems has been studied extensively under various modeling assumptions; we refer the reader to \cite{HassinHaviv2003,ChenFrank2004,Hassin2016} for comprehensive reviews.
In studies where congestion is modeled through a queueing mechanism, the objective function often includes a deduction for queueing-related costs — see, for example, \cite{LeeWard2014, LeeWard2019}, in which a term proportional to the mean stationary queue length is subtracted. These analyses typically assume that the primitives of the underlying queueing model (in particular the service-time distributions) are known.
\item[$\circ$]
In the dynamic pricing literature, the price is iteratively adjusted until, with some level of confidence, it can be said that the revenue-maximizing price has been found.
This line of work typically assumes uncertainty in the demand curve, which must be learned, and in addition that the queueing primitives are not, or partly, known.
Such a setting naturally calls for online algorithms that learn and optimize simultaneously; we refer to, e.g., \cite{vdBoer2015} for a survey of this strand of the literature. 
An added layer of complexity arises when customers are delay-sensitive, for instance, when some decide to balk upon observing excessive congestion.
\end{itemize}

Dynamic revenue maximization problems come in various flavors. First, they may differ in the components included in the revenue function. 
While the core component reflects income from service, many formulations also include a penalty term for performance degradation — commonly modeled as a function of mean queueing delay.
Second, there are differences in the definition of the {\it decision space}. In some studies, the objective is solely to determine the optimal price, while others also optimize over additional parameters such as the service rate or the number of servers.
Third, there are several variants based on the type of information available to the service provider when making pricing decisions (e.g., the customers' arrival times, queueing delays, etc.).

Both offline and online approaches have distinct strengths and limitations. We begin by outlining key considerations related to the offline setting.
(1)~In this setting, it is typically assumed that the revenue can be expressed as a function of price. While evaluating this objective is tractable in simple queueing models, often relying on exponential inter-arrival and service times, it becomes considerably more challenging when incorporating more realistic features, such as general service-time distributions or realistic customer behaviors like balking upon encountering excessive workloads at arrival.
(2)~The offline approach also requires knowledge of the queueing model’s primitives and the demand curve. While certain customer parameters can often be estimated from historical data (e.g., the service-time distribution), the demand curve is more difficult to infer, as it requires observations across a broad range of pricing scenarios. An important but often overlooked issue is {\it model uncertainty}: the effects of estimation errors are typically ignored.
(3)~When revenue cannot be expressed explicitly, the offline approach frequently resorts to {\it heavy-traffic approximations}, which involve scaling up arrival rates and service capacity. This can make the optimal price sensitive to system size, entailing that in contexts where such scaling is inappropriate, the accuracy of these approximations degrades.

We now discuss the advantages and challenges of online pricing approaches.
(1)~The goal is to develop algorithms with low {\it regret}, defined as the cumulative revenue loss relative to the optimal fixed price in hindsight. A key strength of such algorithms is that they operate without requiring an explicit revenue function or prior parameter estimation.
Instead, prices are adjusted dynamically based on observed system behavior.
(2)~The central challenge lies in devising a price update mechanism. A natural strategy is to define {\it cycles} within which the price remains fixed, and based on the observed behavior adapt the price.
This setting can be framed as a multi-armed bandit problem, wherein each arm corresponds to a different price, and the objective is to identify the price that maximizes expected revenue. However, unlike classical bandits with independent rewards, queueing systems exhibit {\it memory effects}: customers carried over from previous cycles influence the current system state and, consequently, the observed revenue. Therefore, any effective algorithm must account for these transient dynamics when evaluating performance at different price levels.
(3)~Another challenge concerns how prices are updated. A common {\it gradient-based} method adjusts the price iteratively using an estimate of the objective’s gradient. The main difficulty is estimating this gradient accurately from data while ensuring fast convergence {of the price updating algorithm}.

\medskip

{\it Our setup~---} As mentioned, the class of revenue maximization problems we consider requires a sound modeling of the interaction between price, demand, and congestion. Existing approaches generally decompose the objective into a demand-based revenue component and a congestion penalty, with the relative influence of each determined by ad hoc weighting parameters. In contrast, our framework {\it integrates} congestion effects directly into a {\it single} demand-related revenue term, thereby eliminating the need to assign weights to intrinsically incomparable quantities such as monetary revenue and mean queueing delay. Moreover, our model embodies an inherent decentralized decision-making structure, whereby customers independently decide whether to join or balk after observing the congestion level, and the central planner exerts only indirect influence via the admission price.

Our base model is a single-server system with Poisson arrivals and a general service-time distribution. 
Upon arrival, customers are informed of the current price $p$ set by the administrator and the current workload level $v$, and decide whether to join or balk. Let $H_\theta(p, v)$ denote the probability that a customer joins {conditional on} observing workload $v$ and price $p$, where $\theta$ is a known ({possibly} multidimensional) parameter characterizing the customer population.
This setup is flexible in the sense that it can represent a heterogeneous population of strategic customers with some  parametric distribution of customer types that differ in their utility function, or it can simply reflect a behavioral (e.g., empirical) joining rule. The objective of the administrator is to maximize long-term revenue, i.e., they aim to find \( p^* \) that maximizes
\begin{equation}\label{eq:obj}
p\,\mathbb{E}[H_\theta(p,V)],
\end{equation}
where \( V\equiv V(p) \) is the {stationary workload} at price \( p \). Thus, the \textit{effective arrival rate} in this dynamic pricing problem is
\[
\lambda(p):=\mathbb{E}[H_\theta(p,V)].
\]
It is noted that for known queueing primitives, this quantity is generally intractable; even in the fully Markovian case with exponentially distributed service times, one has to resort to numerical evaluation.

This type of optimization problem naturally lends itself to a stochastic gradient descent (SGD) approach. However, implementing such an approach presents significant challenges: (1) the gradient of \( \lambda(p) \) must first be estimated, and (2) the {bias and variance} of the resulting estimator must be understood to establish meaningful asymptotic guarantees for the algorithm’s performance. Throughout this work, we assume that \( \theta \) is known, as it can be efficiently estimated in general settings with balking customers, as discussed in \cite{BodasMandjesRavner2024, InoueRavnerMandjes2023}.

\medskip

{\it Online learning for control of queueing systems ---}
Stochastic approximation (SA) algorithms, the class of iterative methods for finding the roots or extrema of functions when only noisy observations are available, constitute a flexible and robust approach for solving stochastic optimization problems \cite{Borkar2008}); a specific subclass is SGD, which relies on noisy gradient estimates. 
In the context of queueing systems there is a difficulty that stationary expectations and their gradients are often intractable, which makes direct implementation of standard SA algorithms infeasible. 
A useful method to overcome this difficulty is infinitesimal perturbation analysis (IPA); see \cite{Glasserman1990} for an extensive introduction to this method and \cite{LEcuyerGlynn1994} for an example of its use in a queueing context. 
The regenerative properties of such systems can often be used in order to obtain unbiased estimators of the objective function and to verify convergence to the true solution \cite{Fu1990}. 
In \cite{RavnerSnitkovsky2024} an SA algorithm based on busy periods is used to evaluate the Nash equilibrium for queues with balking customers, and more general strategic customer behavior. 
Reinforcement learning is another related approach for steady-state optimization of queueing networks \cite{DaiGluzman2022}, particularly well-suited to settings with a complex or high-dimensional policy space. In the context of dynamic pricing of service systems, the analysis is often limited to stylized, tractable models. For example, \cite{AA2013} presents a Bayesian method for learning the optimal admission price in a Markovian system with an unknown proportion of patient customers. The method relies on posterior updating of this proportion based on observed customer decisions.

Our work follows in the tradition of \cite{ChenLiuHong2024}, in that it also develops an SGD-based online learning dynamic pricing algorithm aimed at identifying the revenue-maximizing price  in a setting that includes congestion effects. In \cite{ChenLiuHong2024}, these congestion effects are incorporated into the objective function through a term proportional to the mean stationary queue length, while the arrival rate depends solely on the current price level.  It is assumed that the function $\lambda(p)$ is known, while the customers' service-time characteristics are not. 
This approach necessitates assuming that the moment-generating and cumulant-generating functions of service and inter-arrival times satisfy certain Cramér-type conditions.
The analysis relies on a specific construction for the inter-arrival and service time distributions to enable the use of conventional IPA. 
Specifically, their setup involves a sequence of independent and identically distributed random variables \( (U_n)_{n \in \mathbb{N}} \), where the \( n \)-th inter-arrival time is modeled as \( U_n / \lambda(p) \) given that the price is $p$; a similar construction is applied to the service times.
{Notably, this definition implies that the effective arrival process is not congestion-dependent. In contrast, our approach avoids imposing such a structure, in that we assume that the inter-arrival distribution depends on both the price and the workload.}

\medskip

{\it Contributions~---} 
We now outline the main contributions of our work, distinguishing between three aspects: (1)~the problem formulation; (2)~the results; and (3)~the underlying techniques.
\begin{itemize}
    \item[(1)] Regarding the problem formulation, we explicitly model the behavior of customers who choose not to join the system when faced with high congestion, resulting in lost revenue. Incorporating such balking customers is particularly challenging, as in our framework we observe only the {\it effective} arrivals — that is, data related to the customers who do not balk and enter the system. The dependence of each customer’s joining decision on the congestion level introduces several evident technical hurdles. A useful insight is that the arrival rate at a given price $p$ can be written as the reciprocal of the mean effective inter-arrival time at the same price. 

    \noindent
   Our objective function \eqref{eq:obj} consists of a single term that naturally captures the effect of congestion. This formulation is conceptually more transparent than the approach in \cite{ChenLiuHong2024}, which combines a revenue term with a separate congestion penalty and thus requires specifying the relative weight of each component. Moreover, while the congestion penalty in \cite{ChenLiuHong2024} reflects the performance degradation experienced by customers, it does not account for the reduced willingness of customers to join the system under high congestion.

    \item[(2)] We develop an {online} algorithm that provably converges to the optimal price $p^*$. The {adaptive queue control schemes in previous works, such as \cite{ChenLiuHong2024,Glasserman1990,LEcuyerGlynn1994}, typically rely on observations of customer delays}, where in the proofs it plays a central role that the queue process is geometrically ergodic \cite[Lemma 2]{ChenLiuHong2024},  {and assumes, as noted before, that the service and inter-arrival times satisfy Cramér-type conditions.} Importantly, our price update scheme avoids such assumptions. The process at the center of our analysis is the sequence of inter-arrival times, which is well-behaved (light tails and quick mixing) regardless of the workload process. We assume Poisson arrivals, which is a mild assumption in the context of a large pool of potential customers independently requesting service.

    As the algorithm adapts to the distribution of the underlying process, the resulting observations become non-stationary, introducing bias into the gradient estimator. Nevertheless, this estimator remains consistent, with its variance uniformly bounded with respect to the price. Furthermore, the bias decays at a rate of $\mathcal{O}\big(\eta_{k-1} ({T^*_{k-1}})^{-2}\big)$, where $\eta_k$ denotes the learning rate of the SGD algorithm and $T^*_k$ represents an increasing sampling window size at iteration $k \geq 1$. By setting $\eta_k = \eta k^{-\alpha}$, the algorithm achieves a regret bound of $\mathcal{O}\big(\sum_{k=1}^{L} T^*_k k^{-\alpha/2}\big)$, where $L$ is the total number of iterations.

    \item[(3)] To establish convergence to the optimal price \( p^* \), it is necessary at various stages to control the transient behavior of the underlying queueing system. 
    In each cycle of the price-adjustment algorithm, the price is kept fixed at a value $p$. The system then operates as an M/G/1 queue with balking customers, in which the probability that an arriving customer joins is given by $H_\theta(p, v)$, conditional on the observed congestion level $v$.
Analyzing such queues is notoriously difficult, even in terms of their stationary behavior. We succeed in developing sufficiently sharp bounds, {based on a natural coupling idea,} that enable proving the convergence to the revenue maximizing price.  
These enable us to identify explicit bounds on the bias and `variability' of the gradient of the revenue (i.e., the derivative of the revenue with respect to the price, in each step of the algorithm). 

 The IPA estimator required the development of a new recursive formulation of the sample-path gradients pertaining to the effective arrival process. It is noted that the standard analysis in this area covers processes that are considerably more elementary, typically of random-walk type. Providing performance guarantees for the specific form of the resulting gradient estimator required novel coupling arguments. As mentioned, the arrival rate is  the reciprocal of the mean effective inter-arrival time, which is a non-standard object in the context of IPA analysis.

     \item[(4)]
    \noindent 
    Our approach paves the way for the development of more robust online control algorithms for queueing systems with congestion-dependent joining behavior. Unlike in \cite{ChenLiuHong2024}, we do not optimize over service capacity. However, a byproduct of our analysis is the derivation of workload gradients, which makes our  {algorithm} extensible to settings that include holding costs. 
    Notably, these workload gradients differ structurally from those obtained through classical IPA schemes, due to the dependency between inter-arrival times and the workload. Furthermore, our approach naturally extends to multi-server systems, enabling optimization over the number of servers as well. 
\end{itemize}

We have organized our paper as follows: Section \ref{section:model_and_objectives} introduces the model, formulates the problem that is being studied, and in short highlights the stochastic gradient descent algorithm designed to solve it. 
Section \ref{sec:queue} goes into intricate mathematical details and establishes crucial results about the model that are key to the goal of this paper.
Section \ref{sec:grad_est} is devoted to the construction of the gradient estimator used in the stochastic gradient descent algorithm.
By this point, the reader is exposed to the complete algorithm with all the details. 
Section \ref{section:convergence_analysis} then focuses on the analysis of the proposed gradient estimator. 
Upper bounds on the bias and the variability of the gradient estimator are established.
To evaluate the performance of the proposed algorithm, Section~\ref{section:regret_analysis} analyzes a widely used  metric known as \textit{regret}.
Section \ref{section:numerical_experiments} discusses a series of simulation experiments, confirming our estimation procedure's efficacy.
{Appendix \ref{appendix:proofs} contains the proofs to some results in this paper, and Appendix \ref{appendix:formulae:alternative_joining_probability} consists of additional formulae for the interested reader, which were used to obtain some of the plots in Section~\ref{section:numerical_experiments}.}

\section{Model and objectives} \label{section:model_and_objectives}

In this section we discuss the model under study, introduce our objective function, and specify our learning algorithm.

\subsection{System dynamics}

We model the service system via a FCFS (first-come-first-served) single-server queue.
Potential customers arrive according to a Poisson process with 
known
rate $\Lambda$.
Let the customers' service requirements form a sequence of independent and identically distributed (i.i.d.) random variables, distributed as the generic random variable $S$ with cumulative distribution function (CDF) $G(\cdot)$, where these service times are assumed independent of the arrival process.
To join the system, each customer must pay the admission price $p \in \mathcal{P} \equiv \big[\underline{p}, \overline{p}\big]$ quoted by the service operator (where $0 \leq \underline{p} \leq \overline{p}$).
The system operator is free to adapt this price over time.
After noting the admission price $p$ and their prospective waiting time (before service) in the system $V$, the customer joins the system with probability $H(p, V)$, where $H: \mathcal{P} \times \mathbb{R}_{+} \mapsto [0,1]$, and otherwise they balk.
These join/balk decisions are sampled independently across customers, and independently of both the arrival process and the service requirements.
Customers that join the system remain until service completion. Over a period during which the price $p$ is fixed, the model described above can be denoted — using standard Kendall notation —  as an M/G/1\,+\,$H(p, V)$ queue, to emphasize that it is an M/G/1 queue with customers who are sensitive to both price and congestion.

\begin{example}\label{example: Psi and H}
    For $\theta_1, \theta_2 > 0$, suppose $H(p, V) = \exp\left(-\theta_1 p - \theta_2 V\right)$. Then an incoming customer who must pay admission price $p$, and whose prospective waiting time is $V$ joins the system with probability $\exp\left(-\theta_1p-\theta_2V\right)$ and balks otherwise.
    \hfill $\blacklozenge$
\end{example}

For ease of discussion, throughout our exposition we maintain the probabilistic interpretation of $H$. Nonetheless, there is an alternative (utility-theoretic) interpretation of $H$ as well. Specifically, suppose customers are equipped with
\begin{enumerate}
    \item disutility function $\phi: \mathcal{P} \times \mathbb{R}_{+} \mapsto [0, \infty)$ which evaluates their disutility in joining service,
    \vspace{-0.3cm}
    \item disutility threshold $Y$ sampled from a distribution with a given CDF.
\end{enumerate}
When an incoming customer faces an admission price $p$ and a prospective waiting time $V$, they evaluate their disutility of joining the service, given by $\phi(p, V)$, and decide to enter the system only if $Y \geq \phi(p, V)$; otherwise, they balk. In Example~\ref{example: Psi and H}, we considered the case $\phi(p, V) = \theta_1 p + \theta_2 V$ and   ${\mathbb P}(Y\leq y)=1 - e^{-x}$, where the latter corresponds to the disutility threshold $Y$ being exponentially distributed with mean~1.

We impose some natural assumptions on the joining probability function $H(p, V)$. Specifically, we assume that it decays to zero as either the admission price $p$ or the prospective waiting time $V$ tends to infinity, and that it satisfies mild regularity conditions.

\begin{ass}\label{assumption:H}
    The joining probability function $H: \mathcal{P} \times \mathbb{R}_{+} \rightarrow [0,1]$ is such that
    \begin{enumerate}
        \item[(a)] $\lim_{p \rightarrow \infty} H(p, V) = 0$.
        \item[(b)] $\lim_{V \rightarrow \infty} H(p, V) = 0$.
        \item[(c)] $H(\cdot,V): \mathcal{P} \rightarrow [0,1]$ is continuously differentiable.
        \item[(d)] $H(p, \cdot): \mathbb{R}_{+} \rightarrow [0,1]$ is differentiable.
    \end{enumerate}
\end{ass}

We now proceed by further detailing the system dynamics and introducing additional notation. Since customers may balk, we define the notion of {\it effective arrivals}, referring to those customers who actually choose to join the system.
For $k \geq 1$, let $\widetilde{A}_k$ denote the $k$-th effective arrival time, and let $A_k$ be the $k$-th effective interarrival time, so that $\widetilde{A}_k = \widetilde{A}_{k-1} + A_k$, with $\widetilde{A}_0 := 0$. Importantly, the effective arrival process is {\it not} Poisson due the balking mechanism.

For any $k \geq 1$, we let $\underline{W}_{k}$ ($\overline{W}_k$) denote the {\it workload} in the system just before (after) the $k$-th effective arrival, with $\overline{W}_0 := 0$.
Let $S_k$ be the service requirement corresponding to the $k$-th effective arrival, so that $\overline{W}_k = \underline{W}_k + S_k$. A graphical illustration of the workload process is given in Figure~\ref{fig:workload_dynamics}.

The probability distribution of $A_k$ is characterized by the admission price $p$, and $\overline{W}_{k-1}$.
Let $F_p(.;w)$ be the cumulative distribution function of $A_k$, in a period in which the admission price is $p$ and $\overline{W}_{k-1}=w$. Denoting by $N(\cdot)$ the counting process of the effective arrivals, we have, in self-evident notation, 
\begin{align}
    F_p\left(\ell; w\right) = \prob\left(A_k \leq \ell \ \big\vert \:p, \overline{W}_{k-1} = w\right) &= 1 - \prob\left(A_k > \ell \ \big\vert\: p, \overline{W}_{k-1} = w\right) \notag
    \\
    &= 1 - \prob\left(N\left(\widetilde{A}_{k-1}, \widetilde{A}_{k-1}+\ell\right] = 0 \ \big\vert\: p, \overline{W}_{k-1} = w \right). \notag
\end{align}
Observe that $N(\cdot)$, conditioned on $p$ and $\overline W_{k-1}$, is a {\it thinned Poisson process}. Specifically, 

\begin{align} \label{eqn:general:F_p(l;w)}
    F_p(\ell;w) &=  1 - \text{exp}\bigl(-\Lambda J(\ell;p,w)\bigr), \quad \mbox{}\:\:
    J(\ell;p,w) := \int_{0}^{\ell} H\left(p, (w-t)^{+}\right) \,\mathrm{d}t.
\end{align}

\setcounter{example}{0}
\begin{example}\textit{(Continued) }
    In this case, an explicit formula for $F_p(\ell;w)$ is given by
    \begin{align} \label{eqn:F_p(l;w)}
        F_p(\ell;w) = 
        \begin{cases}
            1 - \exp\biggl(-\Lambda e^{-\theta_1 p -\theta_2 w} \Big\{\frac{e^{\theta_2 \ell} - 1}{\theta_2}\Big\} \biggr), \hspace{1.7cm}  \ell < w.
            \\
            1 - \exp\biggl(-\Lambda e^{-\theta_1 p} \Big\{ \frac{1 - e^{-\theta_2 w}}{\theta_2} + (\ell-w) \Big\}\biggr), \hspace{0.5cm} \ell \geq w.
        \end{cases}
    \end{align}
    A specific situation of interest is $\ell = w$; then we have
    \begin{align} \label{eqn:F_p(w;w)}
        F_p(w;w) = 1 - \exp\Biggl(-\Lambda e^{-\theta_1 p} \Bigg\{\frac{1 - e^{-\theta_2 w}}{\theta_2}\Bigg\} \Biggr).
    \end{align}
    We use this expression later in our analysis. 
    \hfill $\blacklozenge$
\end{example}

\begin{figure}
    \centering
    \includegraphics[width = 15 cm, height = 10cm]{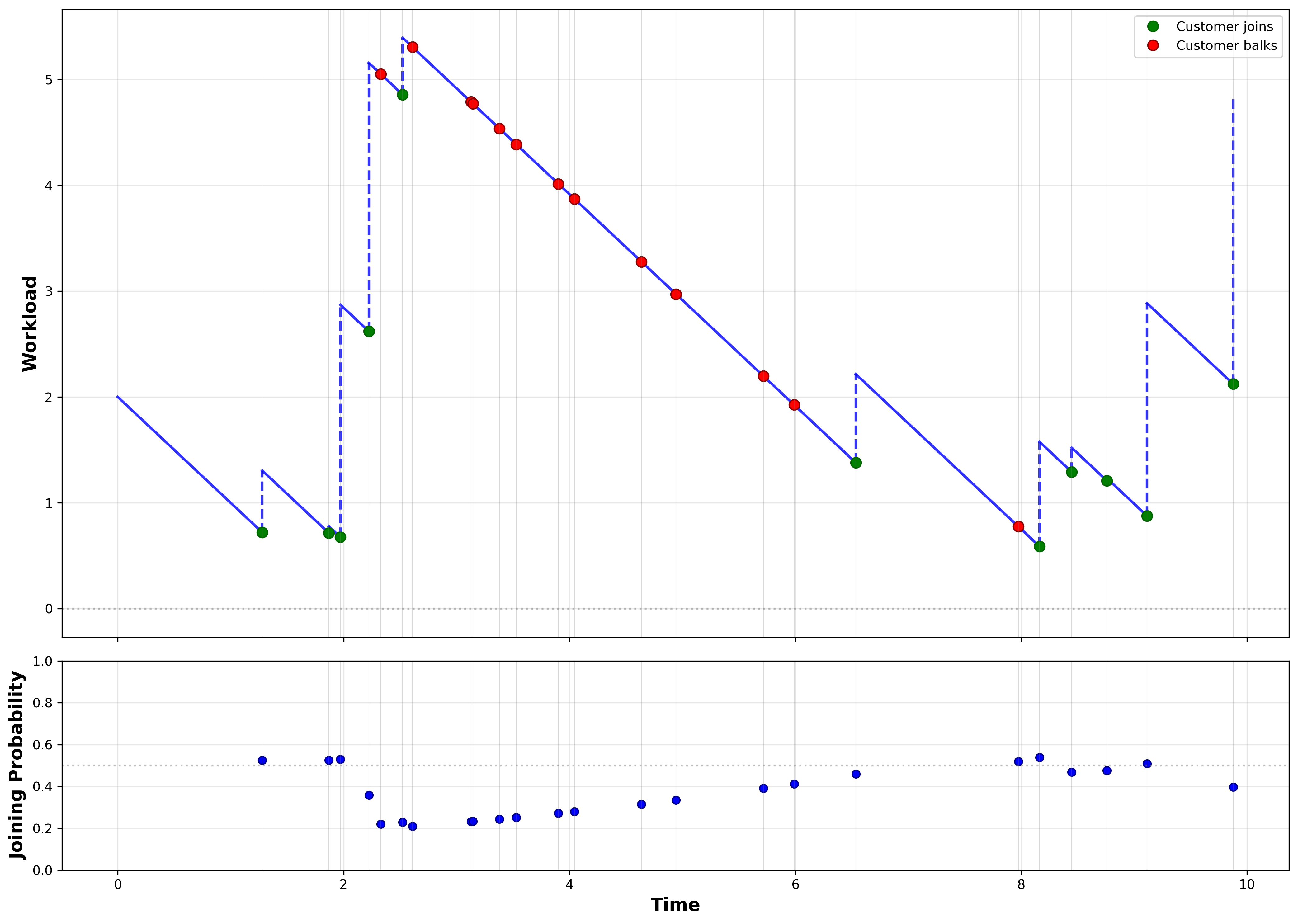}
    \caption{Workload dynamics of an M/G/1 + H queue with $\Lambda = 2$, $p = 5$, $H(p, V) = \exp(-\theta_1p -\theta_2V)$. The green dots represent those customers that join, while red dots represent those that balk.}
    \label{fig:workload_dynamics}
\end{figure}

\begin{figure}[t!]
    \centering
    \begin{subfigure}{0.48\textwidth}
        \centering
        \includegraphics[width=\linewidth]{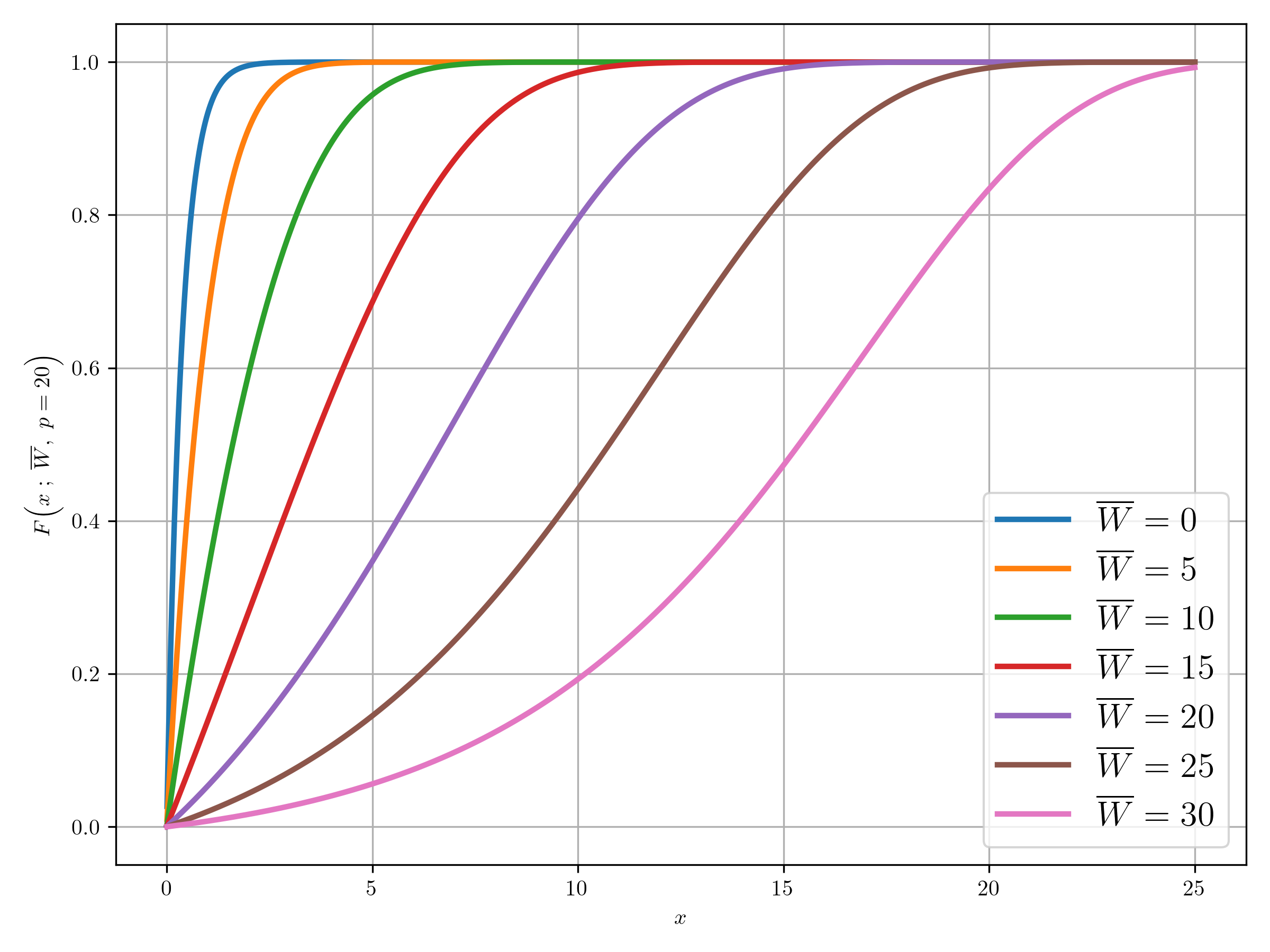}
        \caption{Plots of \( F_p(x;\overline{W}) \) for fixed price \( p = 20 \) and varying initial workloads \( \overline{W} \).}
        \label{fig:CDF:a}
    \end{subfigure}
    \hfill
    \begin{subfigure}{0.48\textwidth}
        \centering
        \includegraphics[width=\linewidth]{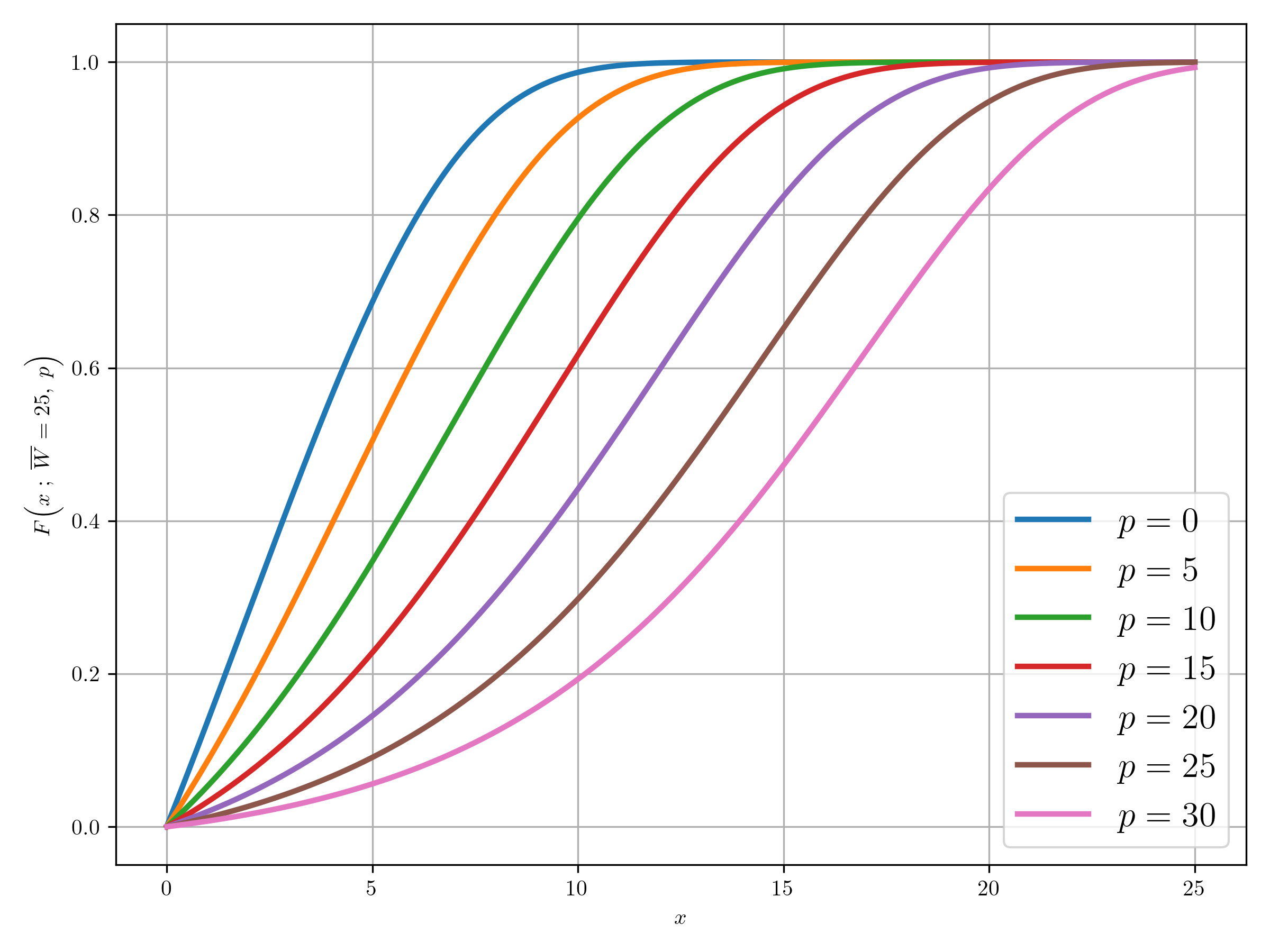}
        \caption{Plots of \( F_p(x;\overline{W}) \) for fixed initial workload \( \overline{W} = 25 \) and varying prices \( p \).}
        \label{fig:CDF:b}
    \end{subfigure}
    \caption{CDF of an effective interarrival time in the setup from Example \ref{example: Psi and H} for different combinations of $p, \overline{W}$.} \label{fig:CDF:A}
\end{figure}

As mentioned, the effective arrival process is not Poisson: the interarrival times are neither exponentially distributed nor independent. 
{Figure \ref{fig:CDF:A} plots the CDF of an effective interarrival time in Example~\ref{example: Psi and H} for $\Lambda = 20$, $(\theta_1, \theta_2) = (0.1, 0.2)$.}
In particular, Figure \ref{fig:CDF:a} shows plots for a fixed admission price $p = 20$, and different initial workloads $\overline{W} \in \{0, 5, 10, 15, 20, 25, 30 \}$ whereas Figure \ref{fig:CDF:b} shows plots for a fixed initial workload $\overline{W} = 25$, and different admission prices $p \in \{0, 5, 10, 15, 20, 25, 30\}$.

Finally we define three random variables corresponding to the steady-state of an M/G/1 + H$(p, V)$ queue:
\begin{enumerate}
    \item[(i)] $W_{\infty}(p)$ denotes the steady-state workload,
    \vspace{-0.3cm}
    \item[(ii)] $\overline{W}_{\infty}(p)$ denotes the steady-state workload just after an effective arrival, and
    \vspace{-0.3cm}
    \item[(iii)] $A_{\infty}(p)$ denotes the steady-state effective interarrival time.
\end{enumerate}

\subsection{Objective function}
 
The primary goal of this paper is to find an efficient algorithm to identify the {\it revenue maximizing price} $p^*$, defined as
\begin{align*}
    p^*=\argmax_{p \in \mathcal{P}}{\Psi(p)}.
\end{align*}
We are now ready to formally define the
steady-state revenue generated per unit time $\Psi(p)$, given that the admission price is (permanently)~$p$.
The following proposition characterizes $\Psi(p)$ in terms of the steady-state queueing performance. {Its proof is deferred to Appendix \ref{appendix:proof:propn:formula:Psi}}.
 
\begin{propn}\label{propn:formula:Psi} For any $p \in \mathcal{P}$,
    \begin{align}
        \Psi(p) &= \Lambda \exptn\big[H(p,W_{\infty}(p))\big] \, p \label{eqn:formula:Psi:1}
        \\
        &= \frac{p}{\exptn\left[A_{\infty}(p)\right]}. \label{eqn:formula:Psi:2}
    \end{align}
\end{propn}

Observe that Proposition \ref{propn:formula:Psi} highlights the fundamental trade-off inherent in the problem of admission pricing in a congested system. Indeed, while a higher price increases the revenue earned per customer,
\begin{enumerate}
    \item it reduces the probability of joining, as in \eqref{eqn:formula:Psi:1}; 
    \item  it also lengthens the expected time between effective arrivals, as in \eqref{eqn:formula:Psi:2}.
\end{enumerate}
The remainder of the paper is devoted to the discussion and analysis of an algorithm for learning $p^*$ that relies solely on Equation \eqref{eqn:formula:Psi:2}.
As such, understanding the structure of the effective arrival process is central to our analysis.
In particular, it allows for the estimation of both the objective function and its gradient — quantities that form the foundation of the learning algorithm presented in Subsection~\ref{sec:alg}.

We next assume strong concavity on $\Psi(\cdot)$ and a Lipschitz continuity on $\nabla\Psi(\cdot)$.
These assumptions are instrumental in establishing the convergence rate of the price iterates generated by our algorithm (Theorem~\ref{thm:price_convergence_rate}), and, in turn, in deriving a regret bound (Theorem~\ref{thm:regret:upper_bound}).

\begin{ass}\label{assumption:Psi}
    There exist $K_1 > K_0 > 0$ such that for all $p \in \mathcal{P}$, \[K_0\, (p^*-p)\le \nabla\Psi(p)\le  K_1\, (p^*-p).\]
\end{ass}

\subsection{Learning algorithm}\label{sec:alg}

In this subsection we define our {SGD-based online} algorithm to learn the optimal price $p^*$.
{Let $\{ T_k \}_{k \ge 1}$ be a sequence of random variables representing {\it window sizes}, to be defined in Section \ref{sec:grad_est}.}
For $k \geq 1$, define $\overline{T}_0:= 0$ and $\overline{T}_k = T_1 + \cdots + T_k$.
Let iteration $k$ of the algorithm correspond to time interval $\big[\overline{T}_{k-1}, \overline{T}_k\big)$.
Then, for $k \geq 1$, in iteration $k$ of the algorithm the operator sets the admission price $p_{k-1}$. In the algorithm $p_0 \in \mathcal{P}$ can be chosen freely, and $p_k$ is given by the recursion
\begin{align}\label{recur}
    p_k := \pi_{\mathcal{P}}\left[p_{k-1} + \eta_k \widehat{\nabla\Psi}(p_{k-1})\right];
\end{align}
here 
\begin{itemize}
    \item[$\circ$] $\widehat{\nabla\Psi}(p_{k-1})$ is an estimator of the gradient {at price $p_{k-1}$}; 
    \item[$\circ$] $\{\eta_k\}_{k \geq 1}$ is a suitable step-size sequence, 
    \item[$\circ$] $\pi_{\mathcal{P}}(\cdot): \mathbb{R} \mapsto \mathcal{P}$ is the standard projection operator (i.e., $\pi_{\mathcal{P}}(x):=\max(\underline{p},\min(\overline{p},x))$).
\end{itemize}

Our strategy for constructing an estimator $\widehat{\nabla\Psi}(p_{k-1})$ follows from Equation \eqref{eqn:formula:Psi:2}:
\begin{align} \label{eqn:der Psi}
    \nabla \Psi(p_{k-1}) = \frac{1}{\exptn\left[A_{\infty}(p_{k-1})\right]} - p_{k-1} \frac{\nabla_p \exptn\left[A_{\infty}(p_{k-1})\right]}{\exptn\left[A_{\infty}(p_{k-1})\right]^2} = \frac{1}{\exptn\left[A_{\infty}(p_{k-1})\right]} - p_{k-1} \frac{\exptn\left[\nabla_p A_{\infty}(p_{k-1})\right]}{\exptn\left[A_{\infty}(p_{k-1})\right]^2}.
\end{align}
The change of the order between the derivative and expectation is justified in~Proposition \ref{propn:der[exptn A] = exptn[der A]}, and is the basis for IPA \cite{Glasserman1990}.  
Based on the observations in the time interval $\big[\overline{T}_{k-1}, \overline{T}_k\big)$, the gradient estimator is then
\begin{align}\label{eqn:psi_hat}
   \widehat{\nabla\Psi}(p_{k-1}):= \frac{1}{\widehat{A_{\infty}}(p_{k-1})} - p_{k-1} \frac{\widehat{\nabla A_{\infty}}(p_{k-1})}{\widehat{A_{\infty}}(p_{k-1})^2},
\end{align}
where the estimators $\widehat{A_{\infty}}(p_{k-1})$ and $\widehat{\nabla A_{\infty}}(p_{k-1})$ are based on event counts and waiting time data. The specific construction of these estimators requires additional notation and is detailed in Section~\ref{sec:grad_est}.









\section{The effective arrival and workload processes}
\label{sec:queue}

In this section we discuss a series of queueing-theoretic results that will be used in our analysis. The first part (Section~\ref{sec:queue_grad}) constructs a recursive representation of the gradients of the interarrival and workload process at arrival instants. This is then used to establish uniform bounds on the steady-state gradients, and finally to prove the characterization of the objective function stated in Proposition~\ref{propn:formula:Psi}. Furthermore, we establish that differentiation and expectation can be interchanged in the expression for the steady-state revenue, thereby validating \eqref{eqn:der Psi}. The second part (Section~\ref{sec:queue_coupling}) introduces a coupling argument to derive a collection of useful properties, which are instrumental in the subsequent performance analysis of our algorithm.

\medskip

The first result provides a sufficient condition for stability. 

\begin{propn}\label{propn: stability condition}
    An {\em M/G/1 + $H(p,V)$} queue with joining probability $H: \mathcal{P} \times \mathbb{R}_+ \mapsto [0,1]$ satisfying Assumption \ref{assumption:H} is stable for all arrival rates $\Lambda > 0$, admission prices $p \in \mathcal{P}$, and service time distributions $G(.)$ with finite first moment.
\end{propn}
\begin{proof}
    The proof is straightforward using \cite[Propositions 7.12, 7.14]{BenaimHurth2022} by defining the Lyapunov function $\Phi(x) = x$; here $x$ denotes the system state, in our case the workload. For more details, we refer the reader to \cite[Theorem 4.1]{BodasMandjesRavner2024}.
\end{proof}

\begin{remark}\label{remark:toy example satisfies stability condition}
    Example \ref{example: Psi and H} satisfies the requirements of Proposition \ref{propn: stability condition}, so that the associated queueing system is stable. \hfill $\spadesuit$
\end{remark}

\subsection{Sample path gradient analysis}\label{sec:queue_grad}

Let $\zeta_k \sim \text{Uniform}[0,1]$ be the random seed used to generate the $k$-th effective interarrival time $A_k$, i.e., $A_k \overset{\mathrm{d}}{=} F^{-1}_p\big(\zeta_k; \overline{W}_{k-1}\big)$.
Recalling \eqref{eqn:general:F_p(l;w)}, we have
\begin{align} \label{eqn:general:inverse:F_p(zeta;w)}
    F^{-1}_p(\zeta;w) = J^{-1}\left(-\frac{1}{\Lambda} \log(1-\zeta) ; p,w\right).
\end{align}
\setcounter{example}{0}
\begin{example}\textit{(Continued) }
    In this case, we have the explicit expression
    \begin{align} \label{eqn:inverse:F_p(zeta;w)}
        F^{-1}_p(\zeta;w) =
        \begin{cases}
            \begin{array}{@{}l@{\quad}l}
                \dfrac{1}{\theta_2} \log\left(1 - \dfrac{\theta_2}{\Lambda} e^{\theta_1 p + \theta_2 w} \log\big(1 - \zeta\big) \right), & \zeta < F_p(w;w) \\[1.2ex]
                w - \left(\dfrac{1-e^{-\theta_2 w}}{\theta_2}\right) - \dfrac{1}{\Lambda} e^{\theta_1 p} \log\big(1 - \zeta\big), & \zeta \geq F_p(w;w),
            \end{array}
        \end{cases}
    \end{align} 
    with $F_p(w;w)$ given in \eqref{eqn:F_p(w;w)}. \hfill $\blacklozenge$
\end{example}

\begin{figure}[t!]
    \centering
    \begin{subfigure}{0.48\textwidth}
        \centering
        \includegraphics[width=\linewidth]{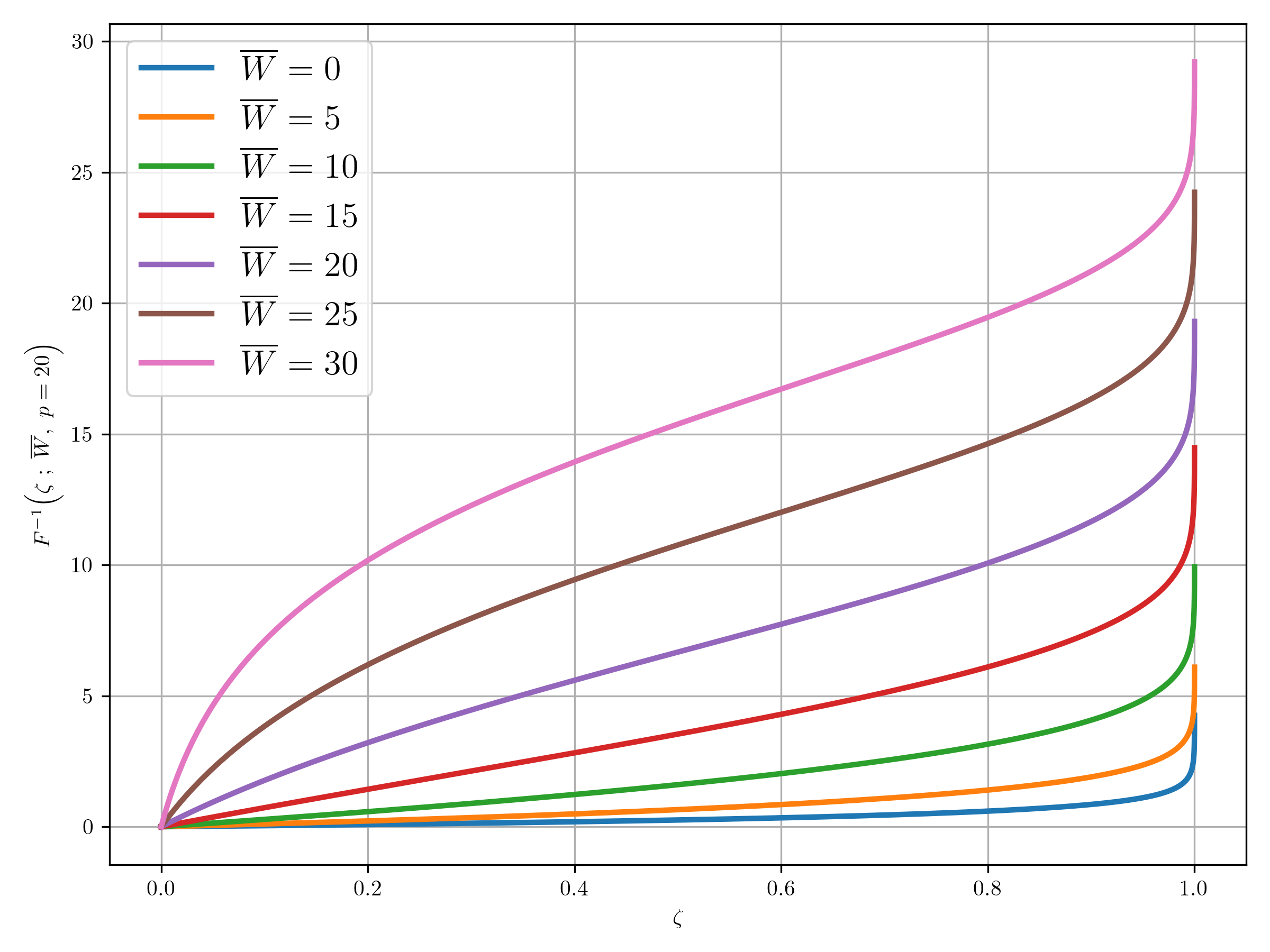}
        \caption{Plots of \( F^{-1}_p(\zeta;\overline{W}) \) for fixed price \( p = 20 \) and varying initial workloads \( \overline{W} \).}
        \label{fig:inverse CDF:a}
    \end{subfigure}
    \hfill
    \begin{subfigure}{0.48\textwidth}
        \centering
        \includegraphics[width=\linewidth]{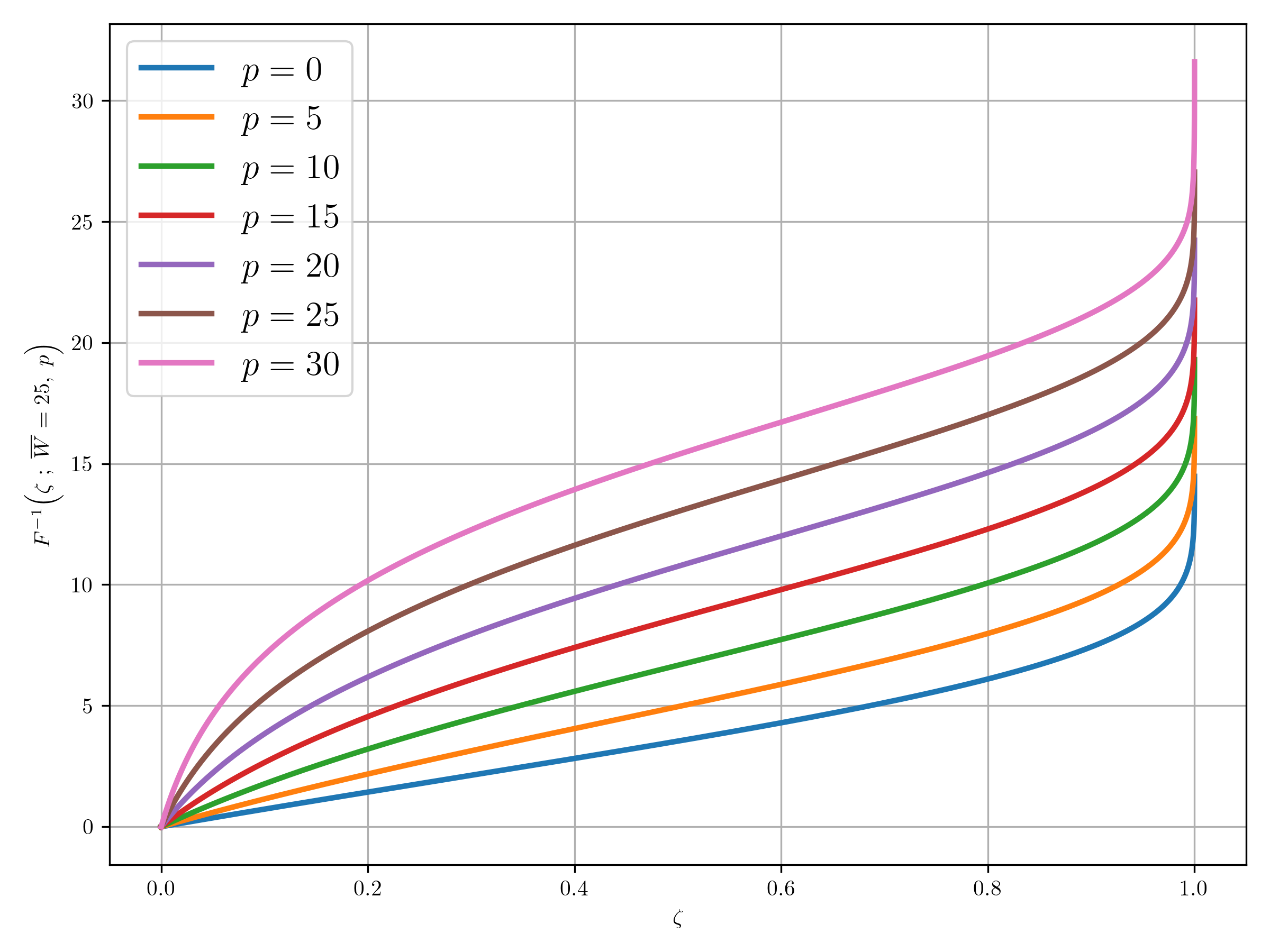}
        \caption{Plots of \( F^{-1}_p(x;\overline{W}) \) for fixed initial workload \( \overline{W} = 20 \) and varying prices \( p \).}
        \label{fig:inverse CDF:b}
    \end{subfigure}
    \caption{Inverse of the CDF of $A$ in the setup from Example \ref{example: Psi and H} for different combinations of $p, \overline{W}$}
    \label{fig:inverse:F}
\end{figure}

    Figure \ref{fig:inverse:F} shows $F^{-1}_p\big(\zeta; \overline{W}\big)$ as a function of the seed $\zeta \in (0,1)$; recall that $\Lambda = 20$ and $H(p, V) = \exp\left(-0.1p-0.2V\right)$ 
    (special case of Example~\ref{example: Psi and H}). 
    Figure \ref{fig:inverse CDF:a} shows curves for a fixed admission price $p = 20$ and different initial workloads $\overline{W} \in \{0, 5, 10, 15, 20, 25, 30 \}$, whereas Figure \ref{fig:inverse CDF:b} shows curves for a fixed initial workload $\overline{W} = 25$ and different admission prices $p \in \{0, 5, 10, 15, 20, 25, 30\}$. The figures confirm the intuition that the effective interarrival time increases when the admission price or the initial workload increases.

The price-gradients associated with the effective interarrival times, are given by  
\begin{align} \label{eqn:general:price der_inverse_F_p(zeta;w)}
    \nabla_p F^{-1}_p\left(\zeta;w\right) = \nabla_p J^{-1}\left(- \frac{1}{\Lambda} \log(1-\zeta); p, w\right),
\end{align}
play a central role in the algorithm for learning the optimal price $p^*$.

\setcounter{example}{0}
\begin{example}\textit{(Continued) }
    An explicit formula for $\nabla_p F^{-1}_p\left(\zeta; w\right)$ is given by
    \begin{align} \label{eqn:price der_inverse_F_p(zeta;w)}
        \nabla_{p} F^{-1}_p(\zeta;w) =
        \begin{cases}
            \begin{array}{@{}l@{\quad}l}
                \left(\theta_1 + \theta_2 \nabla_p w\right) 
                \left(\theta_2 - \dfrac{\Lambda}{\log(1-\zeta)} e^{-\theta_1 p - \theta_2 w} \right)^{-1}, 
                & \zeta < F_p(w;w) \\[1.2ex]
                \left(1 - e^{-\theta_2 w}\right) \nabla_p w 
                - \dfrac{\theta_1}{\Lambda} e^{\theta_1 p} \log(1-\zeta), 
                & \zeta \geq F_p(w;w)
            \end{array}
        \end{cases}
    \end{align}   with $F_p(w;w)$ as in \eqref{eqn:F_p(w;w)}.\hfill $\blacklozenge$
\end{example}

We now introduce two assumptions concerning the {sensitivity} of the interarrival time {to changes in} the initial workload, which will be crucial to our analysis. Broadly speaking, these assumptions require that the sample paths generated by the recursive equations above exhibit uniformly bounded gradients with respect to the price. We subsequently verify that these conditions are satisfied in the setting presented in Example~\ref{example: Psi and H}.

\begin{ass} \label{assumption: workload der: bounds}
    There exists a random variable $\Xi[\zeta]$ such that we have (i)~$\gamma_m := \exptn\big[(\Xi[\zeta])^m\big] < 1$ for all $m \geq 1$, and (ii)~almost surely, 
\begin{align*}
0 <  1 - \nabla_w F^{-1}_p\big(\zeta; W\big) < \Xi[\zeta].
\end{align*}
\end{ass}

\setcounter{example}{0}
\begin{example}\textit{(Continued)}
    In the case of Example \ref{example: Psi and H}, {Assumption \ref{assumption: workload der: bounds} is fulfilled. In particular,}
    \begin{align*}
    \nabla_w F^{-1}_p\left(\zeta; W\right) &=
    \begin{cases}
        \displaystyle\left(1 - \frac{\Lambda}{\theta_2} \, e^{-\theta_1 p - \theta_2 W} \, \frac{1}{\log(1 - \zeta)}\right)^{-1}, & \text{if } \zeta < F_p(W; W), \\
        \displaystyle 1 - e^{-\theta_2 W}, & \text{if } \zeta \geq F_p(W; W),
    \end{cases} \\[1.5ex]
    \Xi[\zeta] &=
    \begin{cases}
        \displaystyle 1 - \left(1 - \frac{\Lambda}{\theta_2} \, e^{-\theta_1 p} \, \frac{1}{\log(1 - \zeta)}\right)^{-1}, & \text{if } \zeta < F_p(W; W), \\
        \displaystyle 1, & \text{if } \zeta \geq F_p(W; W).
    \end{cases}
\end{align*}
        The proof for $\exptn\left[(\Xi[\zeta])^m\right] < 1$ is deferred to {Appendix \ref{appendix:proof:example:workload der:bounds}}. 
\hfill $\blacklozenge$
        \end{example}

\begin{ass}\label{assumption:gradient_bounds}
    Consider a sample path of an {M/G/1 + $H(p,V)$} queue starting with some initial workload $\overline{W}_0  < \infty$. 
    Recall that $\big\{\overline{W}_n\big\}_{n \geq 1}$ are the workloads just after effective arrivals. 
    Then, there exists $G_{\max} > 0$ such that
    \begin{align*}
        \ \nabla_p \overline{W}_n \in \left[-G_{\max}, 0\right], \ \nabla_p F^{-1}_p\big(\zeta_n;\overline{W}_n\big) \in \left[0, G_{\max}\right].
    \end{align*}
\end{ass}
     
The following lemma proves that Assumption \ref{assumption:gradient_bounds} is met in the setting of Example \ref{example: Psi and H}. Its proof is given in {Appendix \ref{appendix:proof:lemma:example:gradient_bounds}}. Figure \ref{fig:visualization:workload_gradient} visualizes this lemma.

\begin{lemma}\label{lemma:example:gradient_bounds}
    When $H(p,V) = \exp\left(-\theta_1 p -\theta_2 V\right)$, we can pick $G_{\max} = {\theta_1}/{\theta_2}$.
\end{lemma}

Returning to the case of a general joining probability, we choose $\overline{W}_0 \sim \overline{W}_{\infty}(p)$. Then, according to the setup of Assumption \ref{assumption:gradient_bounds}, 
    \begin{align*}
     \overline{W}_i \overset{\mathrm{d}}{=} \overline{W}_{\infty}(p),\quad A_i = F^{-1}_p\left(\zeta_i; \overline{W}_{i-1}\right) \overset{\mathrm{d}}{=} A_{\infty}(p),
    \end{align*} 
    with $\nabla_p \overline{W}_i \in \left[-G_{\max}, 0\right]$, $\nabla_p A_i \in \left[0, G_{\max}\right]$. Since this holds for any sample path, we conclude the validity of the following corollary. 
\begin{corlly}\label{corr:bounds_on_gradients} 
If Assumptions~\ref{assumption: workload der: bounds} and \ref{assumption:gradient_bounds} hold, then
\begin{align*}
    \nabla_p \exptn\left[\overline{W}_{\infty}(p)\right] \in \left[-G_{\max}, 0\right]\ , \nabla_p \exptn\left[A_{\infty}(p)\right] \in \left[0, G_{\max}\right].
\end{align*} 
\end{corlly}

\begin{figure}[htbp]
    \centering
    \includegraphics[width=0.8\textwidth]{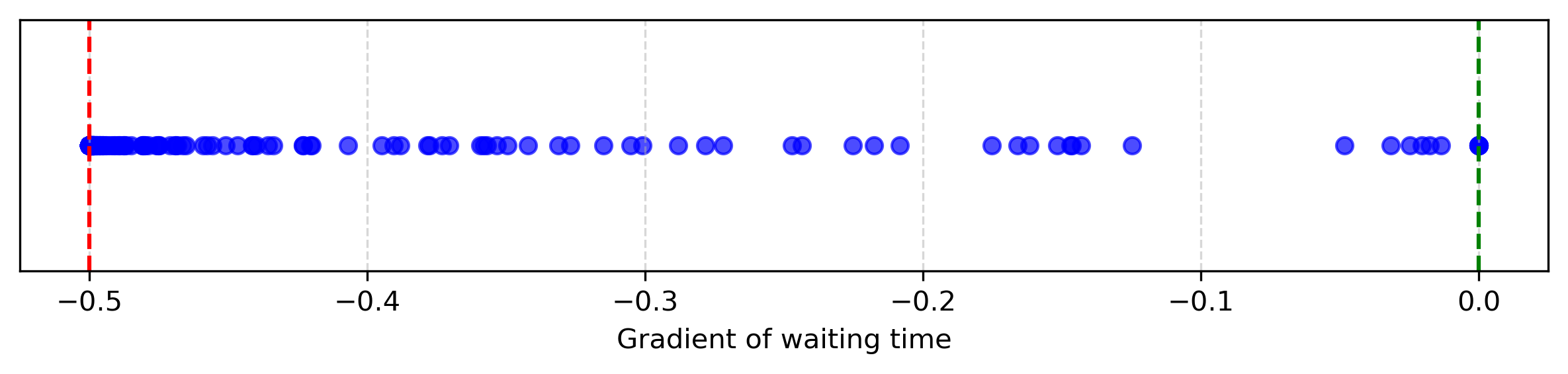}
    \caption{Visualization of Lemma \ref{lemma:example:gradient_bounds}. We simulate an M/G/1\,+\,$H(p, V)$ queue with price $p = 20$, and model parameters $\Lambda = 20$, $\theta_1 = 0.1$, and $\theta_2 = 0.2$.
For each effective arrival $i$, we record the value of the price derivative of the waiting time, denoted $\nabla_p \overline{W}_i$.
The plot shows that $\nabla_p \overline{W}_i$ consistently lies in the interval $\left[ -\frac{\theta_1}{\theta_2}, 0 \right] = [-0.5, 0]$ for all $i$.}
    \label{fig:visualization:workload_gradient}
\end{figure}

\begin{lemma} \label{lemma:DCT:prerequisite:1}
    If $p \in \mathcal{P}$, $\zeta \sim \text{\em Uniform}[0,1]$, and $\overline{W}$ is such that $\exptn\big[\overline{W}\big] < \infty$, then $\exptn \big[F^{-1}_p\big(\zeta; \overline{W}\big)\big] < \infty$.
\end{lemma}
\begin{proof}
    Observe the following two facts:
    \begin{enumerate}
        \item[(i)] Because the inverse of a probability distribution function is monotonically increasing, we have $\big\{\zeta \leq F_p\big(\overline{W};\overline{W}\big) \big\} = \big\{F^{-1}_p\big(\zeta;\overline{W}\big) \leq \overline{W}\big\}$. These are scenarios where the effective interarrival time is not larger than the initial workload $\overline{W}$.
        
        \item[(ii)] Similarly, $\big\{\zeta > F_p\big(\overline{W};\overline{W}\big) \big\} = \big\{F^{-1}_p\big(\zeta;\overline{W}\big) > \overline{W}\big\}$. These are scenarios where the effective interarrival time is larger than the initial workload. Therefore, the workload in the system hits 0 after time $\overline{W}$. By the memoryless property associated with the Poisson process, the remaining time $\big( F^{-1}_p\big(\overline{W}, \overline{W}\big) - \overline{W}\big)$ follows an $\text{Exp}\big(\Lambda H(p, 0)\big)$ distribution {(i.e., is exponentially distributed with mean $1/\Lambda H(p, 0)$)}.
    \end{enumerate}
    Let $\widetilde{\zeta} \sim \text{Uniform}[0,1]$ be independent of $\big(\zeta, \overline{W}\big)$. Then (i) and (ii) can mathematically be written as
    \begin{align*}
        & F^{-1}_p\big(\zeta;\overline{W}\big) \mathbbm{1}\big\{\zeta \leq F_p\big(\overline{W};\overline{W}\big) \big\} \leq_{\text{a.s.}} \overline{W},
        \\
        \overline{W} \leq_{\text{a.s.}} \ &F^{-1}_p\big(\zeta;\overline{W}\big) \mathbbm{1} \big\{\zeta > F_p\big(\overline{W};\overline{W}\big) \big\} \overset{\mathrm{d}}{=} \overline{W} + F^{-1}_p\big(\widetilde\zeta; 0\big).
    \end{align*}
    Upon combining the above, we conclude that
    \begin{align*}
        \exptn\big[F^{-1}_p\big(\zeta, \overline{W}\big)\big] \leq \exptn\big[\overline{W}\big] + \exptn\big[F^{-1}_p\big(\tilde\zeta; 0\big)\big] = \exptn\big[\overline{W}\big] + \frac{1}{\Lambda H(p,0)} < \infty.
    \end{align*}
    This proves the claim. 
\end{proof}

\begin{propn}\label{propn:der[exptn A] = exptn[der A]}
    For any $p \in \mathcal{P}$ we have
    $\nabla_p \exptn\left[A_{\infty}(p)\right] = \exptn\left[\nabla_p A_{\infty}(p)\right]$.
\end{propn}
\begin{proof}
    The interchange of derivative and expectation is justified by using the dominated convergence theorem (see e.g.\ \cite{Billingsley1995}). With $\zeta \sim \text{Uniform}[0,1]$, recall that $A_{\infty}(p) = F^{-1}_p\big(\zeta; \overline{W}_{\infty}(p)\big)$, where by Proposition \ref{propn: stability condition} we have $\exptn\big[\overline{W}_\infty(p)\big] < \infty$. Now,
    \begin{enumerate}
        \item By Lemma \ref{lemma:DCT:prerequisite:1}, $\exptn\left[F^{-1}_p\left(\zeta; \overline{W}_{\infty}(p)\right)\right] < \infty$.
        \item By Assumption \ref{assumption:Psi}, $F^{-1}_p\left(\zeta; \overline{W}_{\infty}(p)\right)$ is continuously differentiable in $p$.
        \item By Assumption \ref{assumption:gradient_bounds}, $\nabla_p F^{-1}_p\left(\zeta; \overline{W}_{\infty}(p)\right) \leq_{\text{a.s.}} G_{\max}$.
    \end{enumerate}
    Observe that this means that the hypotheses of the dominated convergence theorem are met, and the result therefore follows.
\end{proof}

We emphasize that Proposition \ref{propn:der[exptn A] = exptn[der A]} is crucial for the construction of an estimator for the gradient $\nabla\Psi(p)$, as needed in the recursive algorithm \eqref{recur}. Suppose one aims to obtain an estimator for the gradient of the expectation of a random variable with respect to a given parameter --- something that is often highly non-trivial to compute.
Infinitesimal Perturbation Analysis (IPA) \cite{Glasserman1990} constructs a pathwise derivative whose expectation coincides with the derivative of the stationary expectation under suitable regularity conditions. In our setting, these conditions are verified in Proposition \ref{propn:der[exptn A] = exptn[der A]}.
This allows us to build the estimator by essentially `differentiating the sample path', thus yielding a pathwise gradient estimate.

\subsection{Coupling of systems with different initial workloads}\label{sec:queue_coupling}

We continue by providing some additional queueing theoretic results which will be central to the proofs in this paper.
We say that two M/G/1 + $H(p, V)$ queues are \textit{coupled} if there exists a probability space $\big(\Omega, \sigmaf, \prob\big)$ with 
\begin{itemize}
    \item[$\circ$]
random variables $\overline{W}^{1}_0, \overline{W}^{2}_0$, representing the workloads at time $0$ in system 1 and system~2, respectively, 
 \item[$\circ$] the sequence $\big\{ \zeta_n \big\}_{n \geq 1}$, representing the {\it common} interarrival time seeds, and 
  \item[$\circ$] the sequence $\big\{S_n\big\}_{n \geq 1}$, representing the {\it common} service requirements,
\end{itemize}and if the workloads just after effective arrivals $\big\{\overline{W}^{1}_n\big\}_{n \geq 1}$ and $\big\{\overline{W}^{2}_n\big\}_{n \geq 1}$ in the two systems satisfy the Lindley recursion, for $i=1,2$ and $n \geq 1$,
    \begin{align*}
        \overline{W}^{i}_n = \Big(\overline{W}^{i}_{n-1} + S_n - F^{-1}_p\big(\zeta_n; \overline{W}^{i}_{n-1}\big) \Big)^{+}.
    \end{align*}
   The following proposition provides bounds on various quantities that capture the difference between the two coupled queues, typically expressed in terms of the distinct initial workloads $\overline{W}^{1}_0$ and $\overline{W}^{2}_0$.  Let $\left\{A^{i}_k\right\}_{k \geq 1}$ for i = 1,2 be the effective interarrival times and let $\big\{\widetilde{A}^{i}_k\big\}_{k \geq 1}$ for $i=1,2$ be effective arrival times in the two queues. Also, let $N_i(\cdot)$, for $i=1,2$, be the counting process of the effective arrivals in queue $i$. The proof of the proposition is given in Appendix~\ref{appendix:proof:propn:model:theoretical_results}.

\begin{propn} \label{propn:model:theoretical_results}
    Suppose that two {\em M/G/1\,+}\,$H(p,V)$ queues are coupled with finite random initial workloads $\overline{W}^1_0 < \infty$, and $\overline{W}^2_0 < \infty$. For $m \geq 1$, let $\gamma_m$ be as defined in Assumption \ref{assumption: workload der: bounds}. Then,
    \begin{enumerate}
        \item[(a)] For any $n \geq 1$, and $m \geq 0$,
        \begin{align*}
            \exptn \big\vert \overline{W}^{1}_{n} - \overline{W}^{2}_{n} \big\vert^m \leq \gamma_m^{n}  \exptn\big\vert \overline{W}^{1}_0 - \overline{W}^{2}_0 \big\vert^m.
        \end{align*}
    
        \item[(b)]  For any $n \geq 1$,
        \begin{align*}
            \exptn\left\vert A^{1}_n - A^{2}_n\right\vert &\leq \gamma_1^n \exptn\left\vert \overline{W}^{1}_0 - \overline{W}^{2}_0 \right\vert,
            \\
            \exptn \big\vert \widetilde{A}^{1}_n - \widetilde{A}^{2}_n \big\vert &\leq \left( \frac{1-\gamma_1^{n}}{1-\gamma_1}\right) \exptn\left\vert \overline{W}^{1}_0 - \overline{W}^{2}_0 \right\vert.
        \end{align*}

        \item[(c)]  For any $T^* > 0$,
        \begin{align*}
            \exptn \big\vert N_1\left(0,T^*\right] - N_2\left(0, T^*\right]\big\vert \leq \frac{\Lambda H(\underline{p}, 0)}{1-\gamma_1} \exptn\left\vert \overline{W}^{1}_0 - \overline{W}^{2}_0 \right\vert.
        \end{align*}

        \item[(d)] Let $g(\cdot) = x^m + a_1 x^{m-1} + \cdots + a_m$ be a degree-$m$ polynomial. Let $T$ be a random variable with support $[T^*, \infty)$. Then,
        \begin{align*}
            \exptn\left[\frac{1}{g\big(N_1 \left(0, T\right]\big)}\right] \le  \mathcal{O}\left(\frac{1}{{T^*}^m}\right).
        \end{align*} 
    \end{enumerate}
\end{propn}

The coupling featuring in Proposition \ref{propn:model:theoretical_results} is illustrated by Figure \ref{fig:coupled-queues}.

\begin{figure}
    \centering
    \includegraphics[width=0.8\textwidth]{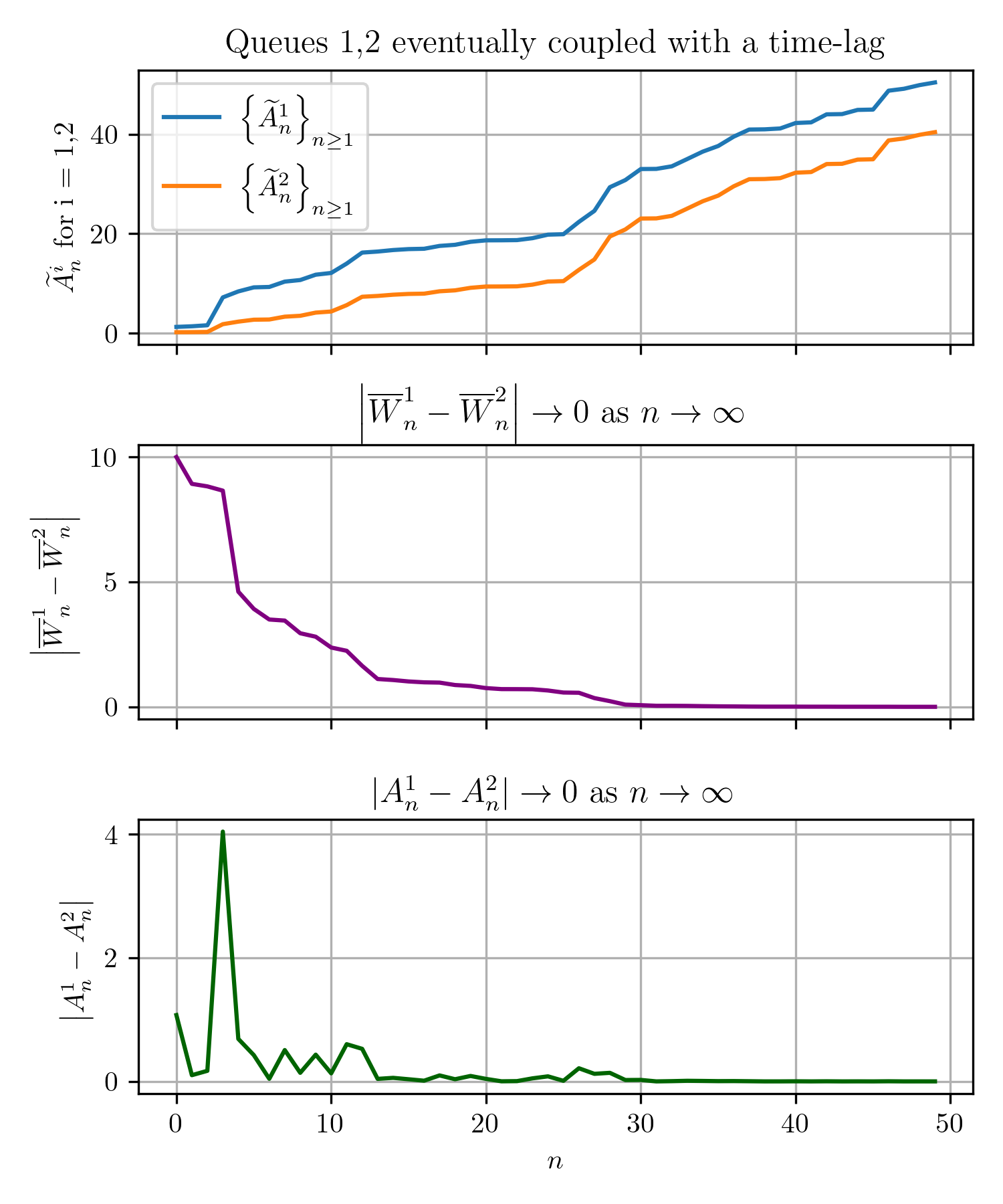}
    \caption{Visualizations of Proposition \ref{propn:model:theoretical_results}. The top plot illustrates the observation from the proof of part (b) that the two queues with distinct initial workloads are eventually coupled with a time-lag. We observe that $\vert \widetilde{A}^1_n - \widetilde{A}^2_n \vert$ does not change from some point on. The second and third plots illustrate parts (a) and (b).}
    \label{fig:coupled-queues}
\end{figure}

\section{Gradient {estimator}}\label{sec:grad_est}

Recalling the gradient estimator  \eqref{eqn:psi_hat}, the main objectives of this section are (i)~to show how its components are computed and (ii)~to establish almost \cls{sure} convergence of the estimator when the sampling window increases.

\medskip

In Section \ref{sec:alg} we introduced the random variables $\{T_k\}_{k \geq 1}$ and $\{\overline{T}_k\}_{k \geq 1}$ corresponding to the window-sizes. In this section, we discuss their construction. Let $\left\{ T^*_k \right\}_{k \geq 1} \in \mathbb{R}_+$ be a strictly increasing deterministic sequence, chosen by the service operator. This sequence is given as input to the algorithm and can be seen as a design choice.
For $k \geq 1$, let $N_k$ be the random variable which denotes the number of effective arrivals during the $k$-th iteration, i.e., during the time interval $\big[\overline{T}_{k-1}, \overline{T}_k\big)$, and let $\widetilde{N}_k = N_1 + \cdots + N_k$. To avoid degenerate estimators, we would like at least one arrival in each sampling window. However, since the timing of the first arrival after time $T_k^*$ is not known, the duration of the sampling window $\big[\overline{T}_{k-1}, \overline{T}_{k-1}+T_k\big)$ is random. Specifically, the window closes at the time of the first arrival after the deterministically chosen $T_k^*$:
\begin{enumerate}
    \vspace{-0.2cm}
    \item[$\circ$] $T_k \geq_{\text{a.s.}} T^*_k$, and
    \vspace{-0.2cm}
    \item[$\circ$] there is an effective arrival at time $\overline{T}_k$, i.e., $\overline{T}_k = A_{\widetilde{N}_k}$.
    \vspace{-0.2cm}
\end{enumerate}
\begin{remark}\label{rem:T_choice}
    Note that the construction of the sampling windows allows for some flexibility. We primarily control the window duration; however, because the number of arrivals is random, there is a non-zero probability of observing zero arrivals, which would result in an unbounded arrival rate estimator. To address this, we extend the window until the first arrival occurs after the specified threshold $T_k^*$. Alternatively, one could fix the number of arrivals in each window as a deterministic increasing sequence. While this leads to a similar analysis, it makes the window duration random, which has distinct implications for the convergence and regret analysis in the following sections. We choose to fix window sizes as this seems more natural in a continuous time setting, and it is also more convenient in terms of the mathematical analysis of the algorithm performance.
\end{remark}

We detail later how one could pick values $\left\{T^*_k\right\}_{k \geq 1}$ so that in the process of learning the optimal price $p^{*}$, the regret of enforcing suboptimal prices remains well controlled. 

A natural estimator for $\exptn\left[A_{\infty}(p_{k-1})\right]$ is the sample average of effective interarrival times observed in the $k$-th iteration, i.e.,
\begin{align} \label{eqn:estimator_exptn_A}
    \widehat{A_{\infty}}(p_{k-1}) := \frac{T_k}{N_k} = \frac{1}{N_k} \sum_{i=\widetilde{N}_{k-1}+1}^{\widetilde{N}_k} A_i.
\end{align}

It is inherently challenging to define an estimator for $\nabla_p \exptn\left[A_{\infty}(p_{k-1})\right]$. On the other hand, a natural estimator for $\exptn\left[\nabla_p A_{\infty}(p_{k-1})\right]$ is the sample average of price-derivatives of observed effective interarrival times. We can therefore leverage Proposition~\ref{propn:der[exptn A] = exptn[der A]} to justify using the latter as a proxy for estimating the former. Recall that for $i \in \big\{\widetilde{N}_{k-1}+1, \cdots, \widetilde{N}_k\big\}, \ A_i = F^{-1}_p\big(\zeta_i; \overline{W}_{i-1} \big)$. We therefore define {the IPA-type estimator}
\begin{align} \label{eqn:estimator_nabla_exptn_A}
    \widehat{\nabla A_{\infty}}(p_{k-1}) := \frac{1}{N_k} \sum_{i=\widetilde{N}_{k-1}+1}^{\widetilde{N}_k} \nabla_p F^{-1}_p\left(\zeta_i; \overline{W}_{i-1} \right).
\end{align}
Evaluating $\nabla_p F^{-1}_p\big(\zeta_i; \overline{W}_{i-1} \big)$ through Equation \eqref{eqn:general:price der_inverse_F_p(zeta;w)} requires knowledge of $\nabla_p \overline{W}_{i-1}$, the computation of which we now explain.
First, we set $\nabla_p \overline{W}_{\widetilde{N}_{k-1}} := 0$.
Then for any $i \in \big\{\widetilde{N}_{k-1}+1, \cdots, \widetilde{N}_{k}\big\}$, $\nabla_p\overline{W}_{i}$ can be computed through the recursion
\begin{align*}
    \nabla_p \overline{W}_{i} =
    \begin{cases}
        \nabla_p \overline{W}_{i-1} - \nabla_p F^{-1}_p\left(\zeta_i; \overline{W}_{i-1}\right) \hspace{1 cm} &\text{if} \ \underline{W}_i > 0,
        \\
        0  &\text{if } \underline{W}_i = 0.
    \end{cases}
\end{align*}

Combining these statistics, we obtain the estimator $\widehat{\nabla \Psi}(p)$, defined by
\begin{align*}
\widehat{\nabla\Psi}(p):= \frac{1}{\widehat{A_{\infty}}(p)} - p \frac{\widehat{\nabla A_{\infty}}(p)}{\widehat{A_{\infty}}(p)^2},
\end{align*}
as given in \eqref{eqn:psi_hat}.
Under the assumptions made above we have the following result regarding consistency of the gradient estimator. 

\begin{propn}\label{propn:estimator_is_consistent}
    {Consider a window of length $T$ on which price $p$ is enforced. Then, under Assumption \ref{assumption:H}, $\widehat{\nabla \Psi}(p) \overset{{\rm a.s.}}{\longrightarrow} \nabla\Psi(p)$ as $T \rightarrow \infty$.}
\end{propn}

The proof relies on applying a law of large numbers for regenerative processes, and is deferred to {Appendix \ref{appendix:proof:propn:estimator_is_consistent}}.

\section{Convergence analysis of the learning algorithm} \label{section:convergence_analysis}

In this section we focus on the analysis of the proposed gradient estimator. 
Specifically, upper bounds on the bias and the variability of the gradient estimator are established.

Recall that, for $k \geq 1$, our price update equation is given by \eqref{recur}. 
Define a filtration $\left\{\mathcal{F}_n \right\}_{n \geq 0}$ as follows:
\begin{align*}
    \mathcal{F}_{0} = \emptyset,
    \quad
    \mathcal{F}_{n} &= \sigma\Big(\widehat{\nabla\Psi} (p_0), \cdots, \widehat{\nabla\Psi}(p_{n-1})\Big)\, \ \text{for } n \geq 1.
\end{align*}

Let $\beta_k$ denote the bias of the gradient estimator in the $k$-th iteration of the algorithm, and let $\nu_k$ denote its variability, i.e.,
\begin{align}
    \beta_k &:= \exptn\left[\widehat{\nabla\Psi}(p_{k-1}) \vert \ \mathcal{F}_{k-1}\right] - \nabla\Psi(p_{k-1}), \label{eqn:bias:beta}
    \\
    \nu_k &:= \exptn\left[\widehat{\nabla\Psi}(p_{k-1})^2\right]. \label{eqn:variability:nu}
\end{align}
The following two subsections study how this bias and variability behave as $k$ goes to $\infty$. In particular, we establish the following proposition. The proofs of the two upper bounds are  detailed in the following two subsections.

\begin{propn}\label{propn:bias_var}
There exists constants $C_1,C_2 > 0$ such that, for all $k\geq 1$,
\begin{align}\label{eqn:var_bound}
    \nu_k &\leq C_1,\\
\label{eqn:bias_bound}
    |\beta_k| &\leq C_2  \frac{\eta_{k-1}}{{T^*_{k-1}}^2}.
\end{align}
\end{propn}

\begin{corlly}
    If the window sizes $\left\{ T^*_k \right\}_{k \geq 1} \in \mathbb{R}_+$ are chosen so that $\sum_{k=1}^\infty {\eta_{k-1}}/({{T^*_{k-1}})^2}<\infty$, then $p_k\asarrow p^*$ as $k\to\infty$.
\end{corlly}

\subsection{Variability analysis}\label{subsection:variability analysis}
The goal of this subsection is to prove the uniform upper bound \eqref{eqn:var_bound}.
From Equation \eqref{eqn:psi_hat}, in combination with Corollary \ref{corr:bounds_on_gradients}, it follows that
\begin{align*}
    \widehat{\nabla\Psi}(p_{k-1})^2 \leq \frac{1}{\widehat{A_{\infty}}(p_{k-1})^2} + \overline{p}^2\frac{\widehat{\nabla A_{\infty}}(p_{k-1})^2}{\widehat{A_{\infty}}(p_{k-1})^4} \leq \frac{1}{\widehat{A_{\infty}}(p_{k-1})^2} + \overline{p}^2 G_{\max}^2 \frac{1}{\widehat{A_{\infty}}(p_{k-1})^4}.
\end{align*}
Therefore,
\begin{align}\label{eqn:variability:temp_1}
    \exptn\left[\widehat{\nabla \Psi}(p_{k-1})^2\right] &= \exptn\left[\exptn\left[\widehat{\nabla \Psi}(p_{k-1})^2 \ \big\vert N_k\right] \right] \notag
    \\
    &\leq \exptn\left[\exptn\left[\frac{1}{\widehat{A_{\infty}}(p_{k-1})^2} \ \Big\vert N_k\right]\right] + \overline{p}^2 G_{\max}^2 \exptn\left[\exptn\left[\frac{1}{\widehat{A_{\infty}}(p_{k-1})^4} \ \Big\vert N_k\right] \right].
\end{align}
By definition,
\begin{align*}
    \widehat{A_{\infty}}(p_{k-1}) = \frac{T_k}{N_k}\asarrow \exptn\left[A_{\infty}(p_{k-1})\right].  
\end{align*}
{It is clear that $A_{\infty}(p_{k-1}) \geq_{\text{st}} \text{Exp}\big(\Lambda H\big(\underline{p}, 0\big)\big) \geq_{\text{st}} \text{Exp}\big(\Lambda\big)$, which implies that $\exptn\big[A_{\infty}(p_{k-1})\big] > 0$.}
Hence, by the continuous mapping theorem, for any $m\in{\mathbb N}$,
\begin{align*}
    &\frac{1}{\widehat{A_{\infty}}(p_{k-1})^m}\asarrow \frac{1}{\exptn\left[A_{\infty}(p_{k-1})\right]^m} \leq \frac{1}{\exptn\big[A_{\infty}(\underline{p})\big]^m}.
\end{align*}
It follows  that $\nu_k = \mathcal{O}(1)$ as $k\to\infty$, thus verifying  \eqref{eqn:var_bound}.

\subsection{Bias analysis}\label{subsection:bias analysis}
In the analysis of the bias, the starting point is
\begin{align}\label{eqn:bias:decomposition}
    \big\vert \beta_k \big\vert &= \Big\vert \exptn\left[\widehat{\nabla\Psi}(p_{k-1}) \vert \ \mathcal{F}_{k-1}\right] - \nabla\Psi(p_{k-1}) \Big\vert \notag
    \\
    &= \Bigg\vert \exptn\left[\frac{1}{\widehat{A_{\infty}}(p_{k-1})} - p_{k-1}\frac{\widehat{\nabla A_{\infty}}(p_{k-1})}{\widehat{A_{\infty}}(p_{k-1})^2} \ \Big\vert \mathcal{F}_{k-1} \right] - \left(\frac{1}{\exptn\big[A_{\infty}(p_{k-1})\big]} - p_{k-1}\frac{\nabla \exptn\big[A_{\infty}(p_{k-1})\big]}{\exptn\big[A_{\infty}(p_{k-1})\big]^2} \right)\Bigg\vert \notag
    \\
    &\leq \omega^{\rm (I)}_k +\overline{p}\,\omega^{\rm (II)}_k,
\end{align}
where
\begin{align*}
    \omega_k^{\rm (I)}&:={\Bigg\vert \exptn\left[\frac{1}{\widehat{A_{\infty}}(p_{k-1})} \ \Big\vert \mathcal{F}_{k-1} \right] - \frac{1}{\exptn\big[A_{\infty}(p_{k-1})\big]} \Bigg\vert},\\ \omega_k^{\rm (II)}&:=  {\Bigg\vert \exptn\left[\frac{\widehat{\nabla A_{\infty}}(p_{k-1})}{\widehat{A_{\infty}}(p_{k-1})^2} \ \Big\vert \mathcal{F}_{k-1} \right] - \frac{\nabla \exptn\big[A_{\infty}(p_{k-1})\big]}{\exptn\big[A_{\infty}(p_{k-1})\big]^2}\Bigg\vert}.
\end{align*}
We aim to determine an upper bound on the rate at which $\beta_k$ vanishes as $k \to \infty$, which we do by analyzing the decay behavior of both $\omega_k^{\text{(I)}}$ and $\omega_k^{\text{(II)}}$. To this end we rely on the following two lemmas, the proofs of which are given in Appendix~\ref{appendix:proof:bias}.

\begin{lemma} As $k\to\infty$, \label{L1}
  \begin{align*}
    \omega_k^{\rm (I)} =  \mathcal{O}\left(\eta_{k-1} \frac{1}{{T^*_{k-1}}^2}\right).
\end{align*}  
\end{lemma}

\begin{lemma} As $k\to\infty$, \label{L2}
  \begin{align*}
    \omega_k^{\rm (II)} =  \mathcal{O}\left(\eta_{k-1} \frac{1}{{T^*_{k-1}}^2}\right).
\end{align*}  
\end{lemma}

Using Lemmas \ref{L1}--\ref{L2}, in combination with Equation \eqref{eqn:bias:decomposition}, we arrive at the following result, which further implies \eqref{eqn:bias_bound}.

\begin{lemma}
    As $k\to\infty, $
\begin{align}\label{eqn:bias:final_rate}
    |\beta_k| = \mathcal{O}\left(\eta_{k-1} \frac{1}{{T^*_{k-1}}^2}\right).
\end{align}
\end{lemma}

\section{Regret Analysis} \label{section:regret_analysis}

This section is the culmination of the results obtained in the previous sections, in that we present our two main theorems. The first theorem quantifies the convergence of the price iterates $p_k$ to the optimal price $p^*$. The second theorem assesses the {\it regret} associated with the algorithm given in Subsection~\ref{sec:alg} (relying on the estimates proposed in Section~\ref{sec:grad_est}).
While learning the optimal admission price $p^*$ by following this algorithm, the operator sets sub-optimal admission prices and thereby does not earn optimal revenue. This deficit between the theoretical optimal revenue and the actual revenue generated is called regret. We denote the cumulative regret after $L$ iterations of the algorithm by $R(L)$.

Let $N(\cdot)$ be the counting process corresponding to effective arrivals and let $W(\cdot)$ be the workload process in a M/G/1\, +\, $H(p,V)$ queue.
Let $\mathscr{F} = \left(\sigmaf_{k}\right)_{k \geq 0}$ be a filtration such that $\sigmaf_{0} = \emptyset$, and $\sigmaf_{k}$ consists of all queueing data from the first $k$ iterations of the algorithm, i.e., from time $0$ till time $\overline{T}_k$.
Then the regret can be written as
\begin{align*}
    R(L) = \sum_{k=1}^{L} \Psi(p^*) \exptn\left[T_k\right] - \exptn\Big[p_{k-1}\exptn\big[N\big(\overline{T}_{k-1}, \overline{T}_k\big] \big\vert \sigmaf_{k-1}\big]\Big].
\end{align*}
{Recalling that $N\big(\overline{T}_{k-1}, \overline{T}_k\big] = N\big(\overline{T}_{k-1}, \overline{T}_{k-1} + T^*_k\big] + 1$, and that $p_{k-1}$ is $\sigmaf_{k-1}$ measurable}, we can decompose the above as
\begin{align*}
    R(L) &= \sum_{k=1}^{L} \Psi(p^*) T^*_k - \exptn\Big[p_{k-1}\exptn\big[N\big(\overline{T}_{k-1}, \overline{T}_{k-1} + T^*_k\big] \big\vert \sigmaf_{k-1}\big]\Big] + \sum_{k=1}^{L} \Psi(p^*) \exptn\left[T_k - T^*_k\right] - \exptn\left[p_{k-1}\right],
\end{align*}
where $\exptn\left[T_k - T^*_k\right] = o\left(T^*_k\right)$ and $\exptn\left[p_{k-1}\right] = o\left(\exptn\Big[p_{k-1}\exptn\big[N\big(\overline{T}_{k-1}, \overline{T}_{k-1} + T^*_k\big] \big\vert \sigmaf_{k-1}\big]\Big]\right)$. Therefore,
\begin{align*}
    R(L) = \mathcal{O} \left(\sum_{k=1}^{L} \Psi(p^*) T^*_k - \exptn\Big[p_{k-1}\exptn\big[N\big(\overline{T}_{k-1}, \overline{T}_{k-1} + T^*_k\big] \big\vert \sigmaf_{k-1}\big]\Big]\right).
\end{align*}
Let $R_1(L)$ correspond to the regret due to picking a suboptimal price, and $R_2(L)$ correspond to the regret due to non-stationarity:
\[R(L)=R_1(L)+R_2(L)\ = \sum_{k=1}^{L} R_{1,k}+\sum_{k=1}^{L} R_{2,k},\]
where
\begin{align*}
    R_{1,k} &:= {\Psi(p^{*}) T^*_k - \exptn\Big[\Psi(p_{k-1})\Big] T^*_k},
    \\
    R_{2,k}&:=  {\exptn\Big[\Psi(p_{k-1})\Big] T^*_k - \exptn\Big[p_{k-1}\exptn\big[N\big(\overline{T}_{k-1}, \overline{T}_{k-1} + T^*_k\big] \big\vert \sigmaf_{k-1}\big]\Big]}.
\end{align*}

In order to derive an upper bound on $R(L)$, we analyze $R_1(L)$ and $R_2(L)$ in the following two subsections.

\subsection{Analysis for $R_1(L)$}

{The following lemma establishes an upper bound on $R_1(L)$.}
Define
    \begin{align*}
        \sigma := \max\left\{\frac{1}{\exptn\big[A_{\infty}(\underline{p})\big]}, \left\vert \frac{1}{\exptn[A_{\infty}(\overline{p})]} - \overline{p} \ G_{\max}\frac{1}{\exptn[A_{\infty}(
        \underline{p})]^2} \right\vert \right\}.
    \end{align*}

\begin{lemma} \label{lemma:Regret:Term 1}
   For any $L$,
    \begin{align*}
        R_1(L) \leq \sigma \sum_{k=1}^{L} T^*_k \exptn\big\vert p^{*}-p_{k-1} \big\vert.
    \end{align*}
\end{lemma}
\begin{proof}
     By the mean-value theorem for random variables \cite{MVT_random_variables}, for a random variable $\xi_{k-1}$ taking values in $(0,1)$, and for $\widetilde{p}_{k-1} := \xi_{k-1}p^* + \big(1-\xi_{k-1}\big)p_{k-1}$,
    \begin{align*}
        R_{1,k} = T^*_k \exptn\big[\Psi(p^*) - \Psi(p_{k-1})\big] &= T^*_k \exptn\Big[(p^* - p_{k-1}) \exptn\big[\nabla_p \Psi\big(\widetilde{p}_{k-1}\big) \ \big\vert p_{k-1}\big]\Big] 
        \\
        &\leq T^*_k \exptn\Big[\big\vert p^* - p_{k-1}\big\vert \exptn\Big[\big\vert \nabla_p \Psi\big(\widetilde{p}_{k-1}\big)\big\vert \ \big\vert p_{k-1}\Big] \Big].
    \end{align*}
    With reference to Corollary \ref{corr:bounds_on_gradients},
    \begin{align*}
        \frac{1}{\exptn[A_{\infty}(\overline{p})]} - \overline{p} \ G_{\max}\frac{1}{\exptn[A_{\infty}(
        \underline{p})]^2}  \leq \nabla_p \Psi(p) &= \frac{1}{\exptn[A_{\infty}(p)]} - p\frac{\nabla\exptn\big[A_{\infty}(p)\big]}{\exptn[A_{\infty}(p)]^2} \leq \frac{1}{\exptn\big[A_{\infty}(\underline{p})\big]}.
    \end{align*}
    Therefore,
    \begin{align*}
        R_{1,k} &\leq \max\left\{\frac{1}{\exptn\big[A_{\infty}(\underline{p})\big]}, \left\vert \frac{1}{\exptn[A_{\infty}(\overline{p})]} - \overline{p} \ G_{\max}\frac{1}{\exptn[A_{\infty}(
        \underline{p})]^2} \right\vert \right\} T^*_k \exptn\big\vert p^* - p_{k-1}\big\vert= \sigma \ T^*_k \exptn\big\vert p^* - p_{k-1}\big\vert.
    \end{align*}
    From this, the result follows.
\end{proof}

\subsection{Analysis for $R_2(L)$}

{In this subsection we state and prove an upper bound on $R_2(L)$.}

\begin{lemma} \label{lemma:Regret:Term 2}
   For any $L$,
    \begin{align*}
        R_2(L) \leq \overline{p} \ \frac{\Lambda H\big(\underline{p}, 0\big)}{1-\gamma_1} \sum_{k=1}^{L}  \exptn\big\vert W_{\infty}(p_{k-1}) - W\big(\overline{T}_{k-1}\big) \big\vert.
    \end{align*}
\end{lemma}

\begin{proof}
    Fix $k \geq 1 $ and let $\mathcal F_{k-1}$ denote the filtration up to time $k-1$. 
    By construction, $p_{k-1}$ and $W(\overline T_{k-1})$ are $\mathcal F_{k-1}$ measurable.
    By definition of $\Psi(\cdot)$,
    \begin{align*}
        \Psi(p_{k-1})T_k^* = p_{k-1}\,\mathbb E\!\left[
        N_1\big(\overline T_{k-1},\overline T_{k-1}+T_k^*\big)
        \,\big|\, p_{k-1}, W_{\infty}(p_{k-1})
        \right],
    \end{align*}
    where $N_1(\cdot)$ denotes the counting process of effective arrivals for an M/G/1$+H(p,V)$ queue with admission price $p_{k-1}$ and initial workload
    $W_{\infty}(p_{k-1})$. 
    Similarly,
    \begin{align*}
        \mathbb E\!\left[p_{k-1}
        N\big(\overline T_{k-1},\overline T_{k-1}+T_k^*\big)
        \,\big|\,\mathcal F_{k-1} \right] =
        p_{k-1}\,\mathbb E\!\left[
        N_2\big(\overline T_{k-1},\overline T_{k-1}+T_k^*\big) \,\big|\,\mathcal F_{k-1} \right],
    \end{align*}
    where $N_2(\cdot)$ is defined as $N_1(\cdot)$ but now with initial workload $W(\overline T_{k-1})$.
    Therefore,
    \begin{align*}
        R_{2,k} = p_{k-1} \,\bigg\vert \mathbb E\!\left[
        N_1\big(\overline T_{k-1},\overline T_{k-1}+T_k^*\big)
        \,\big|\, p_{k-1}, W_{\infty}(p_{k-1})
        \right] - \,\mathbb E\!\left[
        N_2\big(\overline T_{k-1},\overline T_{k-1}+T_k^*\big) \,\big|\,\mathcal F_{k-1} \right]\bigg\vert
    \end{align*}
    Recalling the coupling construction in  Proposition~\ref{propn:model:theoretical_results}(c) with $\overline{W}^1_0  = W_{\infty}(p_{k-1}), \overline{W}^2_0 = W\big(\overline{T}_{k-1}\big)$, by applying the triangle inequality and using $p_{k-1} \leq \overline{p}$, we conclude
    \begin{align*}
        R_{2,k} \leq \overline{p} \, \frac{\Lambda H(\underline{p},0)}{1-\gamma_1} \mathbb E\!\left[ \Big| W_{\infty}(p_{k-1})-W(\overline T_{k-1}) \Big|
        \right].
    \end{align*}
    Summing over $k=1,\dots,L$ completes the proof.
\end{proof}

\subsection{Analysis for $R(L)$}

{
In the previous two subsections, we derived upper bounds on $R_1(L)$ and $R_2(L)$. This now enables us to assess, in Theorem~\ref{thm:regret:upper_bound}, the growth rate of $R(L)$. Before proceeding, however, we state in Theorem~\ref{thm:price_convergence_rate} a useful result that characterizes the rate of convergence of the price iterates $p_k$. The proof of this theorem draws on elements from the line of reasoning used in its counterpart in \cite{ChenLiuHong2024}.
}
\begin{ass}\label{assk}
    For $k \geq 1$, define
    \begin{align*}
        B_k = \exptn\left[\exptn\left[\widehat{\nabla\Psi}(p_{k-1}) - \nabla\Psi(p_{k-1}) \big\vert \mathcal{F}_{k-1}\right]^2\right]^{\frac{1}{2}}.
    \end{align*}
    Let $\frac{1}{2} < \alpha \leq 1$, $\eta > 0$, and let $\eta_k := \eta k^{-\alpha}$ where $\eta$, $\left\{T^*_k\right\}_{k \geq 1}$ are chosen such that
    \begin{enumerate}
        \item[(a)] $\left(1 + \frac{1}{k}\right)^{\alpha} \leq 1 + \frac{K_0}{2}\eta_k$,
        \item[(b)] $B_k \leq \frac{K_0}{8} k^{-\alpha}$.
    \end{enumerate}
\end{ass}
Let $k_0 := \max\big\{k: 2\eta_k K_0 > 1\big\}$, and define $C_4 := \max\big\{ k_0^{\alpha}\big(\overline{p}-\underline{p}\big)^2, 8 \eta \,C_1/K_0 \big\}$.

\begin{thm}\label{thm:price_convergence_rate}
   Under Assumption \ref{assk}, for any $k \geq 1$,
    \begin{align*}
        \exptn\left[\left(p_{k-1}-p^*\right)^2\right] \leq C_4 k^{-\alpha}.
    \end{align*}
\end{thm}
\begin{proof}
    Deferred to Appendix \ref{appendix:proof:th:price_convergence_rate}.
\end{proof}
An immediate consequence of Theorem~\ref{thm:price_convergence_rate}, obtained using Jensen's inequality, is that \[\exptn\big[\big\vert p_{k-1}-p^*\big\vert\big] \leq C_4 k^{-\alpha/2}.\]
This gives us the rate of convergence of iterates $\{p_k\}_{k \geq 1}$ to $p^*$.
This rate is in terms of the parameter $\alpha$ appearing in Assumption \ref{assk}.
By Equation \eqref{eqn:bias_bound}, choosing a sequence $\big\{T^*_k\big\}_{k \geq 1}$ which grows fast tightens the rate at which the bias decays.
This allows a potentially larger value of $\alpha$ to satisfy Assumption \ref{assk}(b), and thereby to obtain a faster convergence rate to $p^*$.
What this means is that using larger window sizes in the learning algorithm reduces the bias in the gradient estimators $\widehat{\nabla\Psi}(p)$, making them more accurate and enabling faster convergence to $p^*$. Since the service operator’s objective is to minimize regret by learning the optimal price as quickly as possible, a natural trade-off arises: larger window sizes reduce the number of iterations required but increase the time per iteration. We explore this trade-off through simulation experiments in Section~\ref{subsection:numerical_experiments:window_size}.

We now have the full machinery in place to establish an upper bound on the total regret $R(L)$ in the regime where the number of iterations $L$ becomes large.

\begin{thm}\label{thm:regret:upper_bound}
    Under Assumption \ref{assk}, as $L \rightarrow \infty$,
    \begin{align*}
        R(L) = \mathcal{O}\left(\sum_{k=1}^{L} T^*_k \ k^{-\alpha/2}\right).
    \end{align*}
\end{thm}
\begin{proof}
    Using Theorem \ref{thm:price_convergence_rate} and Jensen's inequality,
    \begin{align*}
        \exptn\left[\big\vert p_{k-1}-p^* \big\vert\right] \leq \exptn\left[(p_{k-1}-p^*)^2\right]^{\frac{1}{2}} = \mathcal{O}\left(k^{-\alpha/2}\right),
    \end{align*}
    which implies, in {conjunction with Lemma \ref{lemma:Regret:Term 1}} that, as $L\to\infty$,
    \begin{align*}
        R_1(L) = \mathcal{O}\left(\sum_{k=1}^{L} T^*_k \ k^{-\alpha/2}\right).
    \end{align*}
   Equations \eqref{eqn:rho:defn}, \eqref{eqn:delta:stochastic upper bound}, and \eqref{eqn:rho:stochastic upper bound}, together with Lemma \ref{lemma:Regret:Term 2}, lead to
    \begin{align*}
        R_2(L) = \mathcal{O}\left(\sum_{k=1}^{L} \eta_{k-1} \frac{1}{{T^*_{k-1}}^2}\right) = \mathcal{O}\left(\sum_{k=1}^{L} k^{-\alpha}\right)
    \end{align*}
    as $L\to\infty$. 
    Upon combining the above two displays,
    \begin{align*}
        R(L) = \mathcal{O}\left(\sum_{k=1}^{L} T^*_k \ k^{-\alpha/2}\right) + \mathcal{O}\left(\sum_{k=1}^{L} k^{-\alpha}\right) = \mathcal{O}\left(\sum_{k=1}^{L} T^*_k \ k^{-\alpha/2}\right)
    \end{align*}
    as $L\to\infty$.
\end{proof}

\section{Numerical experiments} \label{section:numerical_experiments}

In the preceding sections, we provided a comprehensive mathematical analysis of our {SGD-based online learning algorithm}.
Section~\ref{subsection:numerical_experiments:impact_of_service_time_distribution} empirically investigates how the service-time distribution influences the optimal price $p^*$.
In Section~\ref{subsection:numerical_experiments:convergence_rate}, we study the algorithm's performance through a series of illustrative examples that demonstrate its convergence behavior. Subsequently, in Section~\ref{subsection:numerical_experiments:window_size}, we examine how different window size choices affect the convergence rate toward the optimal price, providing insight into their influence on algorithmic efficiency. We consider two examples of joining functions: the exponential form from Example~\ref{example: Psi and H}, and the following {power-law} form.

\setcounter{example}{1}
\begin{example}\label{example: poly_H}
    For $\theta_1, \theta_2 > 0$, suppose that $H(p, V) = \left(1+\theta_1 p^2 + \theta_2 V^2\right)^{-1}$ is the joining probability of a customer when the admission price is $p$, and the prospective waiting time is $V$.
    \hfill $\blacklozenge$
\end{example}

\begin{remark}
    This section includes several examples based on Gamma-distributed service times.
    We emphasize that our notation $\mathrm{Gamma}\bigl(\alpha, \beta\bigr)$ denotes a Gamma distribution with shape parameter $\alpha$ and rate parameter $\beta$.
    Its mean is $\alpha/\beta$, and its variance is $\alpha/\beta^2$. {We recall that ${\rm Exp}(\beta)$ denotes an exponentially distributed random variable with mean $1/\beta$ (and hence variance $1/\beta^2$).} \hfill $\spadesuit$
\end{remark}

\subsection{{Impact of service-time distribution}} \label{subsection:numerical_experiments:impact_of_service_time_distribution}

In this subsection, our goal is to {illustrate} how the service-time distribution $G(\cdot)$ impacts the optimal price $p^*$ for the system.
To this end, Figure~\ref{fig:numerical_experiments:impact_of_service_time_distribution} presents three sets of plots of the objective function $\Psi(p)$, each corresponding to a different family of service-time distributions.

To the best of our knowledge, no closed-form expression for $\Psi(p)$ is available, not even for the simplest setting of Example~\ref{example: Psi and H}.
Therefore, to obtain a reliable approximation of $\Psi(p)$ and thereby, the optimal price $p^*$, we simulate each example over a fine grid of prices for $N_{\text{eff}} = 5 \times 10^5$ effective arrivals each. 
We then compute the empirical average revenue per unit time $\widehat{\Psi}(p) := p/\widehat{A_{\infty}}(p)$, where $\widehat{A_{\infty}}(p)$ is the sample average of the $N_{\text{eff}}$ effective interarrival times.
Finally, $p^*$ is approximated from the grid as the price which maximizes $\widehat{\Psi}(p)$.
Throughout these experiments, all other model parameters are held fixed. 
In particular, we choose $\Lambda = 20$ and $H(p, V) = \exp\bigl( -0.1 p - 0.2 V \bigr)$.

Figure~\ref{fig:numerical_experiments:Gamma:varying_mean} examines a single family of service-time distributions while varying only their mean. 
Specifically, we consider $G(\cdot) \sim \mathrm{Gamma}\bigl(\mu^2, \mu\bigr), \ \mu \in \{1,2,3,4\}$. 
Under this parameterization, the service-time distribution has mean $\mu$ and variance $1$, so that the four systems differ only in their mean.
The figure also reports the optimal price and associated maximum revenue for each case. The optimal price increases with the mean service time, which is intuitively natural: larger mean service times correspond to larger jobs. At low prices, each arrival generates limited revenue while encouraging high joining probabilities, thereby increasing congestion and lengthening effective interarrival times. A better balance between price and congestion is thus achieved at a moderately higher admission price.
We encourage the reader to interpret these observations in light of Equation~\eqref{eqn:formula:Psi:1}, which makes the dependencies of the objective function on the admission price and the effective interarrival time explicit.

Figure \ref{fig:numerical_experiments:Exp}, based on $\text{Exp}(1/\mu)$ service times with mean $\mu$ and variance $\mu^2$, leads to the same conclusion: the optimal price increases with the mean service time.

Finally, Figure \ref{fig:numerical_experiments:Gamma:varying_variance} examines the effect of service-time variability while keeping the mean fixed.
We consider $\mathrm{Gamma}\bigl(1/\sigma^{2},1/\sigma^{2}\bigr)$ service times, all with mean 1 but with variance $\sigma^{2}$.
In this setting, we observe a notable trend: the optimal price $p^{*}$ decreases as the variance of $G(\cdot)$ increases.

Another striking observation from all three figures is that, for any fixed price $p$, the function $\Psi(p)$ decreases as either the mean or the variance of the service-time distribution increases. 
This behavior is readily explained using \eqref{eqn:formula:Psi:1}: increases in the mean or variance of the service time induce an increase in the mean interarrival time, which in turn lowers $\Psi(p)$.

We conclude this subsection by summarizing the empirical patterns observed in our numerical experiments.
Within the range of parameter settings tested, and keeping all other parameters fixed, we find that:
\begin{enumerate}
    \item $p^*$ increases whenever the mean service time increases.
    \item $p^*$ increases whenever the variance of the service distribution decreases.
    \item $\Psi(p)$, and in particular $\Psi(p^*)$, increases whenever the mean service time increases. 
    \item $\Psi(p)$, and in particular $\Psi(p^*)$, decreases whenever the variance of the service time increases. 
\end{enumerate}

\begin{figure}[t!]
    \centering
    \begin{subfigure}{0.48\textwidth}
        \centering
        \includegraphics[height=5cm]{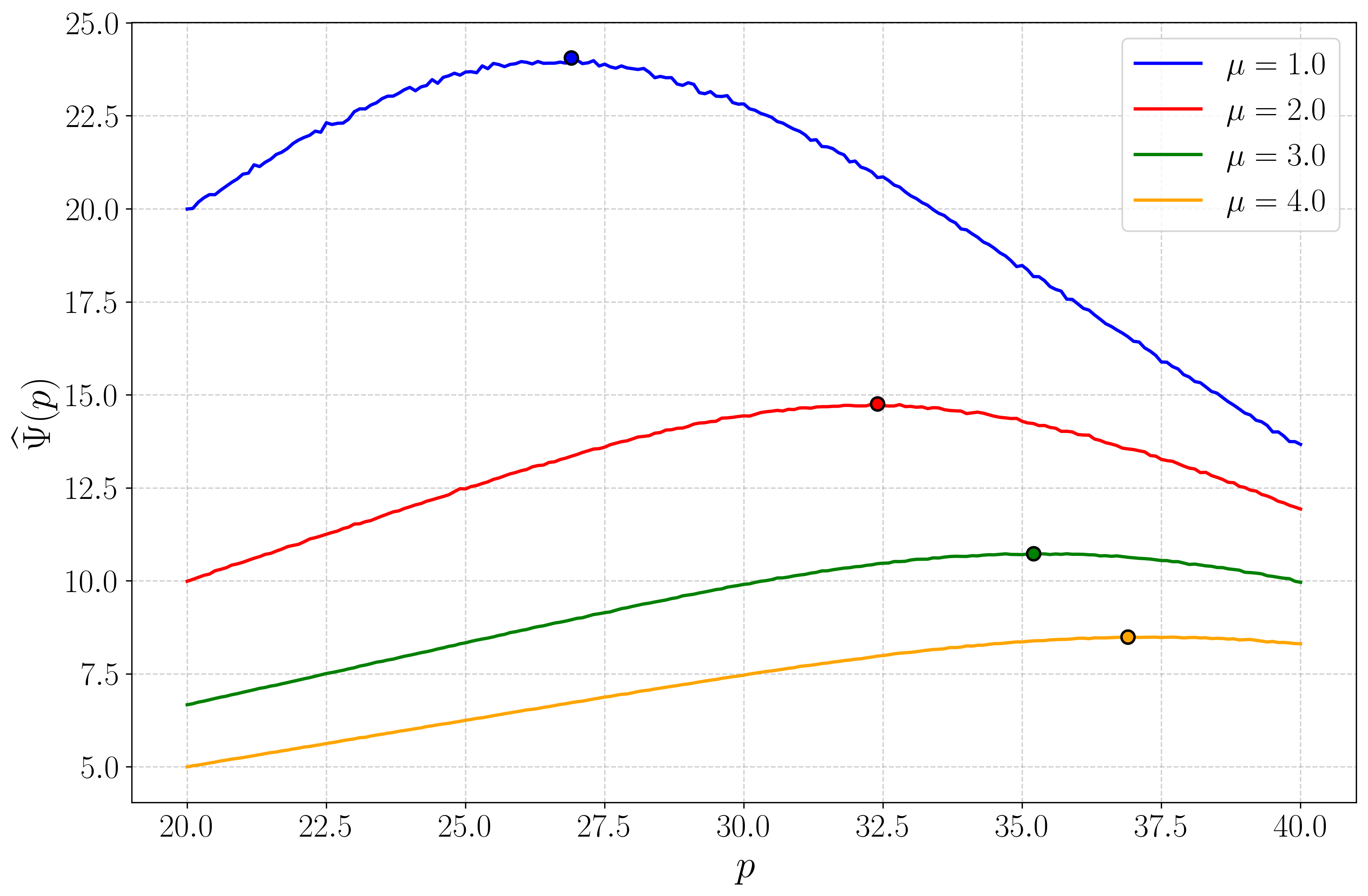}
        \caption{$\widehat{\Psi}(p)$ versus $p$ for $G(\cdot) \sim \text{Gamma}\big(\mu^2, \mu\big)$ for various choices of $\mu$.}
        \label{fig:numerical_experiments:Gamma:varying_mean}
    \end{subfigure}
    \hspace{0.01\textwidth}  
    \begin{subfigure}{0.48\textwidth}
        \centering
        \includegraphics[height=5cm]{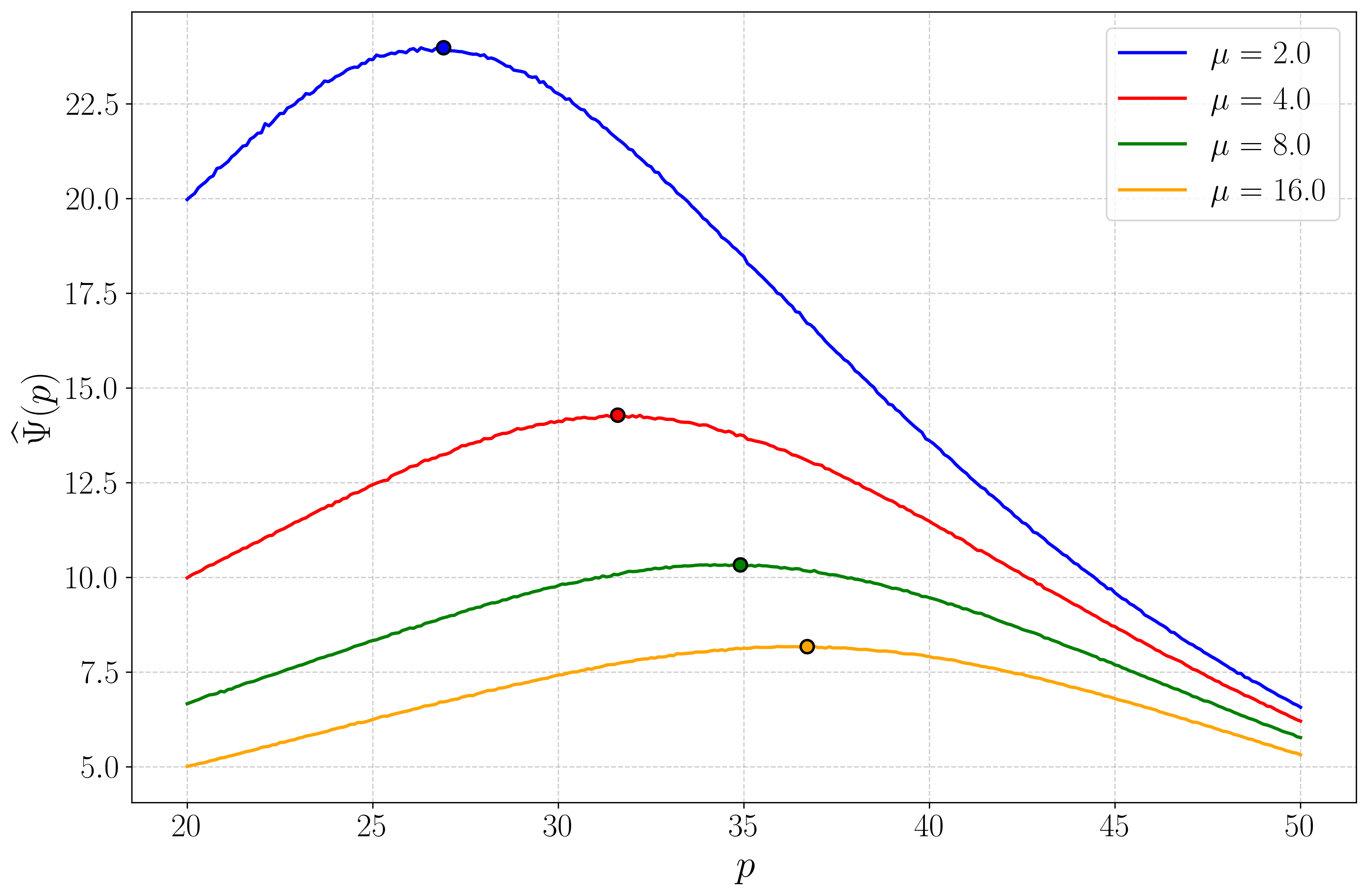}
        \caption{$\widehat{\Psi}(p)$ versus $p$ for $G(\cdot) \sim \text{Exp}\bigl(\frac{1}{\mu}\bigr)$ for various choices of $\mu$.}
        \label{fig:numerical_experiments:Exp}
    \end{subfigure}

    \vspace{0.35cm}

    \begin{subfigure}{0.48\textwidth}
        \centering
        \includegraphics[height=5cm]{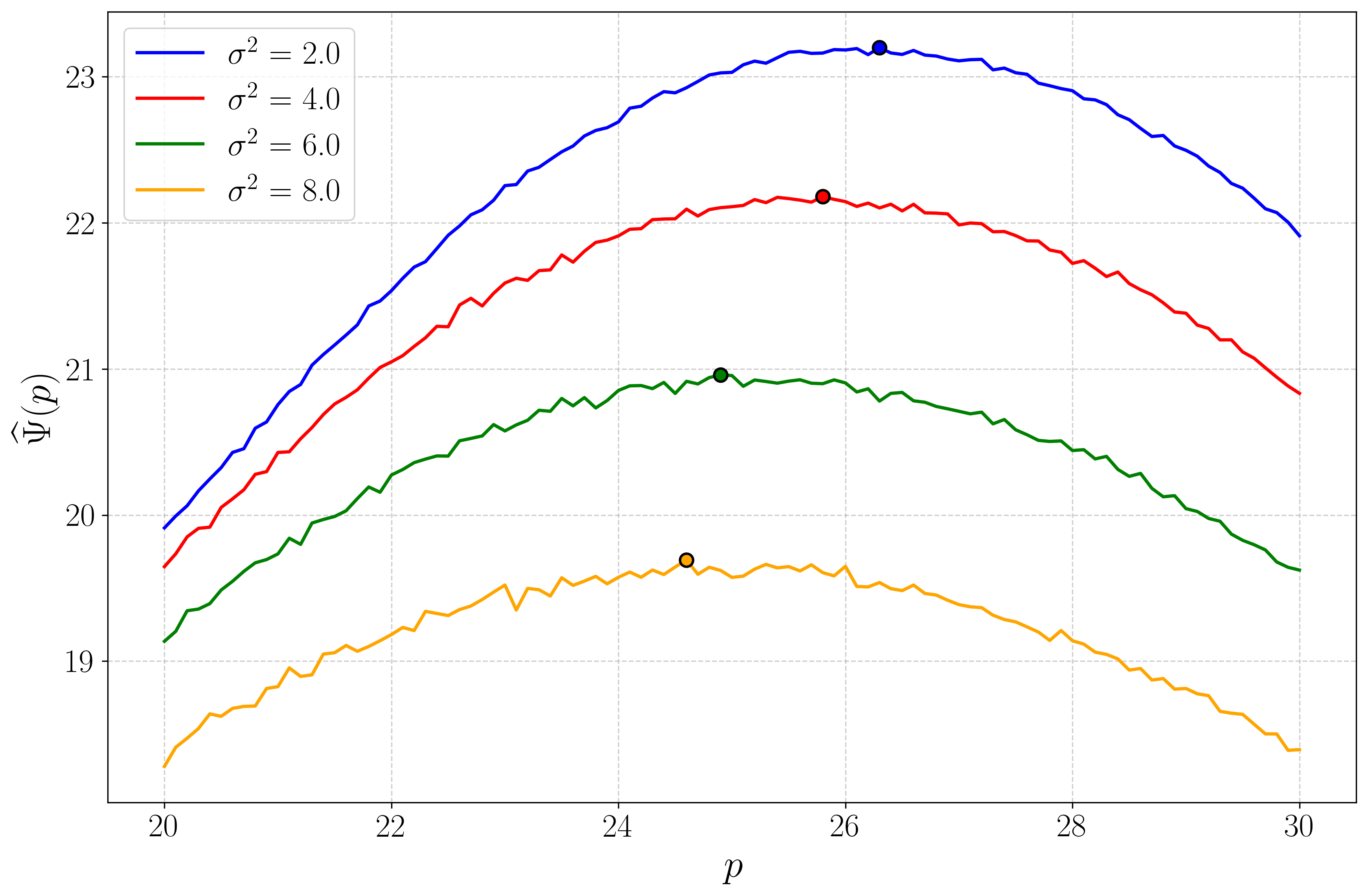}
        \caption{$\widehat{\Psi}(p)$ versus $p$ for $G(\cdot) \sim \text{Gamma}\big(\frac{1}{\sigma^2}, \frac{1}{\sigma^2}\big)$ for various choices of $n$.}
        \label{fig:numerical_experiments:Gamma:varying_variance}
    \end{subfigure}
    
    \caption{Impact of service-time distribution $G(\cdot)$ on the optimal price $p^*$}
    \label{fig:numerical_experiments:impact_of_service_time_distribution}
\end{figure}

\subsection{Convergence of price iterates} \label{subsection:numerical_experiments:convergence_rate}

In this subsection, we illustrate that the stochastic gradient descent algorithm presented in Section~\ref{sec:alg} reliably converges to the optimal price. 
Our goal in this subsection is to convince the reader that convergence is robust to choices of service-time distribution and the joining probability function.
Therefore, we choose 4 examples which have been created by a mix-match of 2 service-time distributions, namely Exponential and Gamma, and 2 joining probability functions corresponding to Examples~\ref{example: Psi and H} and \ref{example: poly_H}, namely $H(p, V) = \exp\bigl(-0.1 p - 0.2 V\bigr)$ and $H(p,V) = \bigl(1 + 0.1 p^2 + 0.2 V^2\bigr)^{-1}$.
These examples serve to demonstrate the practical effectiveness and stability of the proposed method in achieving optimality. 

Expressions analogous to those given in \eqref{eqn:F_p(l;w)}, \eqref{eqn:inverse:F_p(zeta;w)}, \eqref{eqn:price der_inverse_F_p(zeta;w)}, but then for the joining probability function $H(p, V) = \bigl(1 + \theta_1 p^2 + \theta_2 V^2\bigr)^{-1}$, are provided in Appendix \ref{appendix:formulae:alternative_joining_probability}.

{\small
\begin{table}[h!]
\centering
\renewcommand{\arraystretch}{1.2}
\begin{tabularx}{\linewidth}{|l|X|X|X|X|}
\hline
\textbf{Feature} 
    & \textbf{Example 1} 
    & \textbf{Example 2} 
    & \textbf{Example 3} 
    & \textbf{Example 4} \\
\hline
$\Lambda$ 
    & $20$ 
    & $20$ 
    & $20$ 
    & $20$ \\
$G(\cdot)$
    & $\mathrm{Exp}(2)$
    & $\Gamma(0.5,1/3)$ 
    & $\mathrm{Exp}(2/3)$ 
    & $\Gamma(0.5,1/3)$ \\ 
$H(p,V)$
    & $\bigl(1 + 0.1 p^2 + 0.2 V^2\bigr)^{-1}$
    & $\exp\!\bigl(-0.1 p - 0.2 V\bigr)$ 
    & $\exp\!\bigl(-0.1 p - 0.2 V\bigr)$ 
    & $\bigl(1 + 0.1 p^2 + 0.2 V^2\bigr)^{-1}$ \\ 
Grid 
    & $\{0,0.1,\ldots,50\}$ 
    & $\{20,20.1,\ldots,40\}$ 
    & $\{20,20.1,\ldots,40\}$ 
    & $\{10,10.1,\ldots,25\}$ \\ 
$N_{\text{eff}}$
    & $5 \times 10^5$
    & $5 \times 10^5$
    & $5 \times 10^5$
    & $5 \times 10^5$ \\
$p^*$
    & $\approx 9.3$
    & $\approx 29.0$
    & $\approx 29.5$
    & $\approx 16.5$ \\
$\Psi(p^*)$
    & $\approx 16.8$
    & $\approx 17.2$
    & $\approx 17.8$
    & $\approx 9.4$ \\
$T^*_k$
    & $50\log(k+1)$
    & $50\log(k+1)$
    & $50\log(k+1)$
    & $50\log(k+1)$ \\
$\eta_k$
    & $20/k^{0.75}$
    & $20/k^{0.75}$
    & $20/k^{0.75}$
    & $20/k^{0.75}$ \\
\hline
\end{tabularx}
\caption{Model specifications for Examples~1--4.}
\label{table:examples}
\end{table}
}

{Since, to the best of our knowledge, no closed-form expression for $p^*$ is available, we approximate it through long simulation runs, following the procedure described in Section~\ref{subsection:numerical_experiments:convergence_rate}.}

Table \ref{table:examples} summarizes the model parameters used for the examples.
The feature ``Grid" in this table refers to the price grid over which the system is simulated to approximate $\Psi(p)$, and $N_{\text{eff}}$ denotes the number of effective arrivals used to obtain this approximation.
$T^*_k$ the window size and $\eta_k$ the learning rate are the hyperparameters of the algorithm, which are kept the same for all examples.

Figures \ref{fig:SGD:example_1:known_parameters}--\ref{fig:SGD:example_4:known_parameters} illustrate the behavior of our learning algorithm.
The subfigures on the left display the empirical average revenue per unit time for the price grid given in Table \ref{table:examples}.
The subfigures on the right show the algorithm’s performance. 
The horizontal dashed red lines indicate the optimal price $p^*$ for the corresponding instance which the algorithm is expected to converge to. 
As the plots demonstrate, the stochastic gradient descent iterates for the price converge to the empirically determined optimal value, confirming that the algorithm reliably identifies the optimal pricing level.

\begin{remark}
    Certain choices of the service time distribution $G(\cdot)$ and the joining probability $H\big(p,V\big)$ give rise to plots for $\Psi(p)$ which plateau around $p^*$. 
    In such cases, the algorithm naturally has more difficulty converging precisely to $p^*$.
    Example 4 illustrates this phenomenon. Figure \ref{fig:example_4:psi_plot_zoomed} shows $\widehat{\Psi}(p)$ in a neighborhood of $p^*$, where the objective function is visibly flat. Consequently, as seen in Figure \ref{fig:example_4:price_iterates}, the algorithm does not converge fully to $p^*$ even after 500 iterations, in contrast to the other examples in which convergence is much faster.  However, it is important to note that, from a practical standpoint, this behavior is not problematic: the service operator’s primary goal is to maximize revenue. In that respect, the algorithm performs well, quickly moving into the vicinity of the optimal price, where the resulting revenue is essentially optimal. \hfill $\spadesuit$
\end{remark}

\begin{figure}[t!]
    \centering

    \begin{subfigure}{0.48\textwidth}
        \centering
        \includegraphics[height=4.5cm]{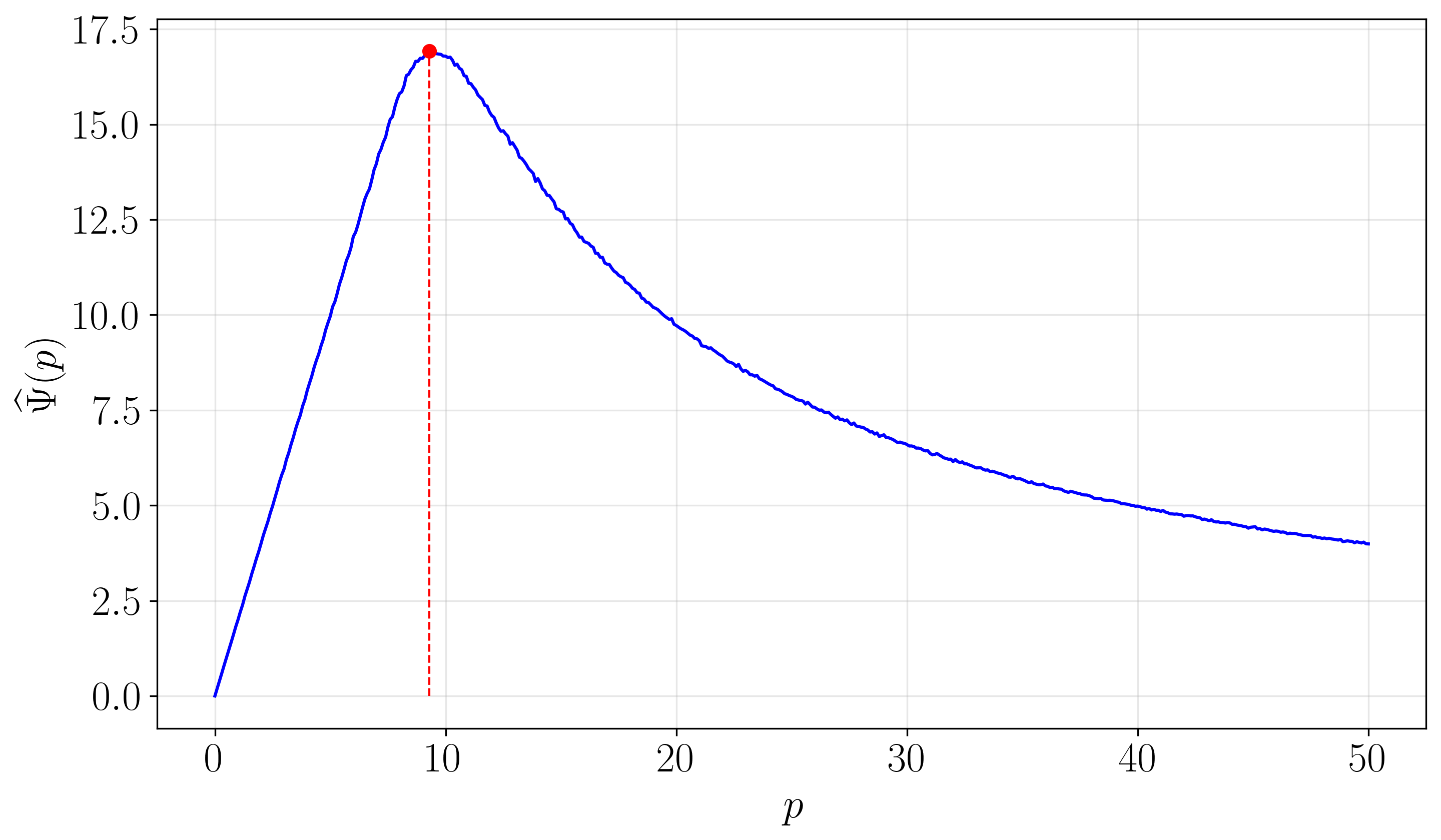}
        \caption{$\widehat{\Psi}(p)$ versus $p$.}
        \label{fig:example_3:psi_plot}
    \end{subfigure}
    \hspace{0.01\textwidth}
    \begin{subfigure}{0.48\textwidth}
        \centering
        \includegraphics[height=4.5cm]{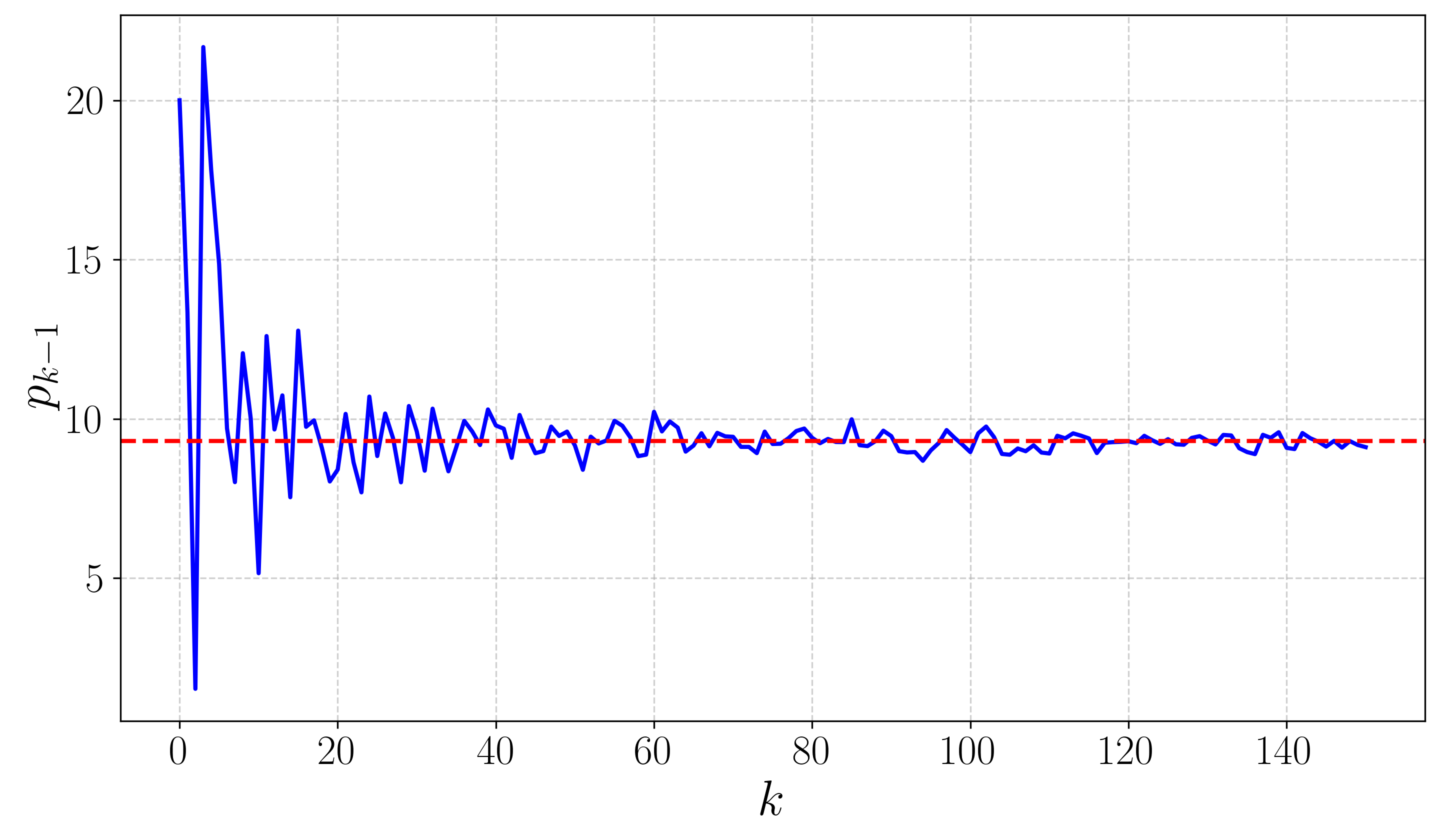}
        \caption{$p_{k-1}$ vs.\ $k$ for $L = 150$ iterations, $p_0 = 20$.}
        \label{fig:example_3:price_iterates}
    \end{subfigure}

    \caption{SGD-based learning algorithm in action for Example 1.}
    \label{fig:SGD:example_3:known_parameters}
\end{figure}

\begin{figure}[t!]
    \centering

    \begin{subfigure}{0.48\textwidth}
        \centering
        \includegraphics[height=4.5cm]{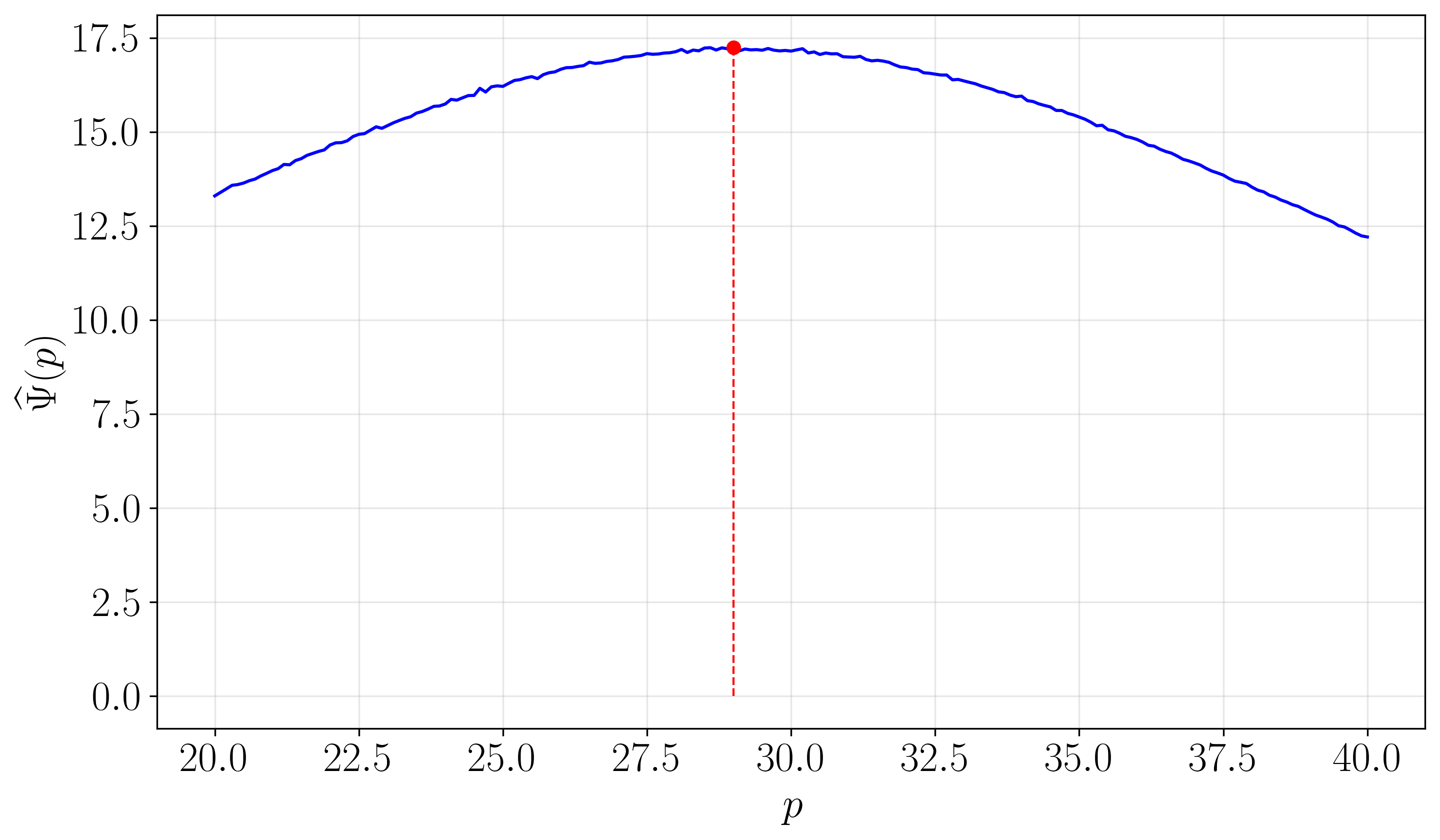}
        \caption{$\widehat{\Psi}(p)$ versus $p$.}
        \label{fig:example_2:psi_plot}
    \end{subfigure}
    \hspace{0.01\textwidth} 
    \begin{subfigure}{0.48\textwidth}
        \centering
        \includegraphics[height=4.5cm]{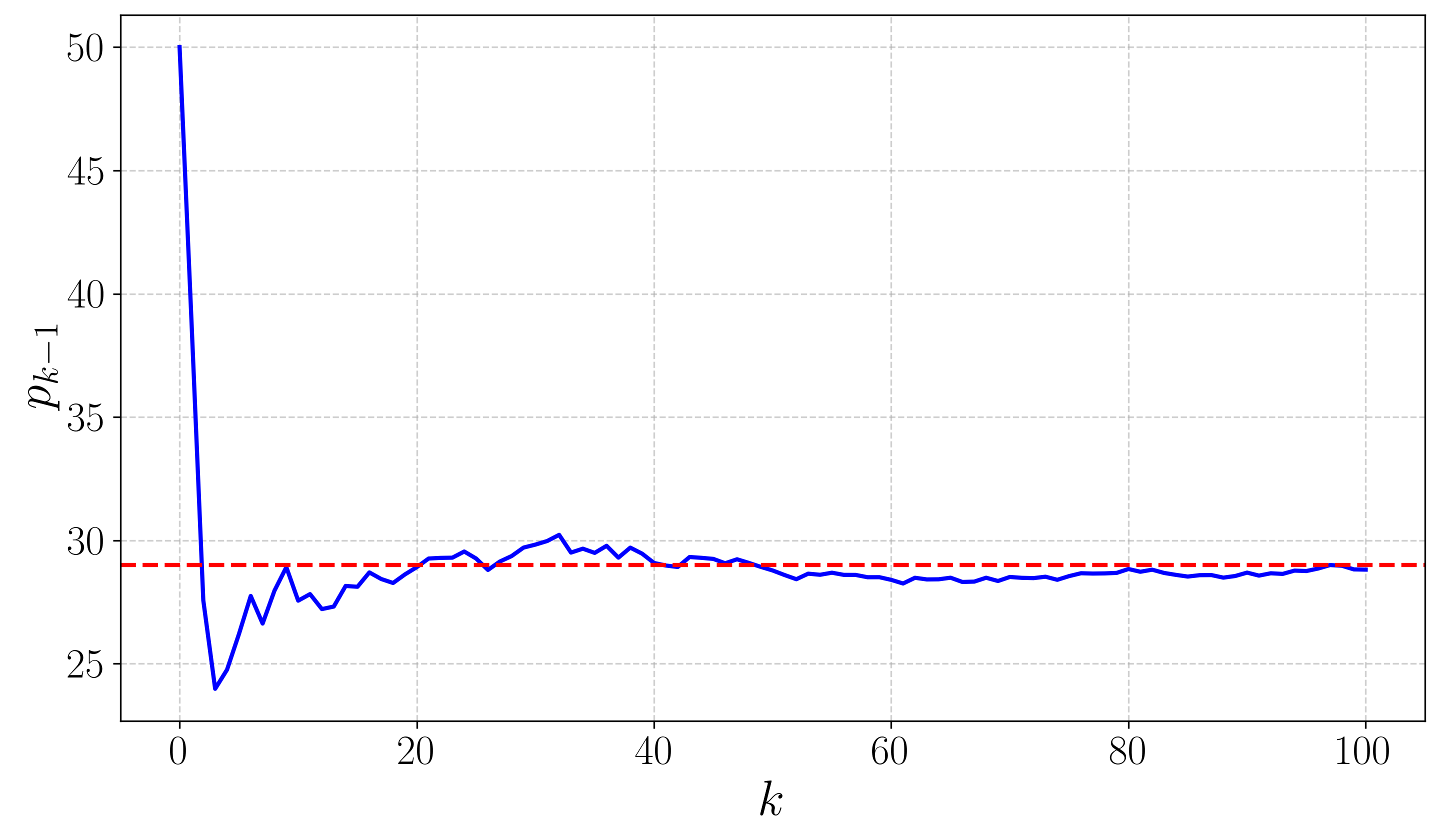}
        \caption{$p_{k-1}$ vs.\ $k$ for $L = 100$ iterations, $p_0 = 50$.}
        \label{fig:example_2:price_iterates}
    \end{subfigure}

    \caption{SGD-based learning algorithm in action for Example 2.}
    \label{fig:SGD:example_2:known_parameters}
\end{figure}

\begin{figure}[t!]
    \centering
    \begin{subfigure}{0.48\textwidth}
        \centering
        \includegraphics[height=4.5cm]{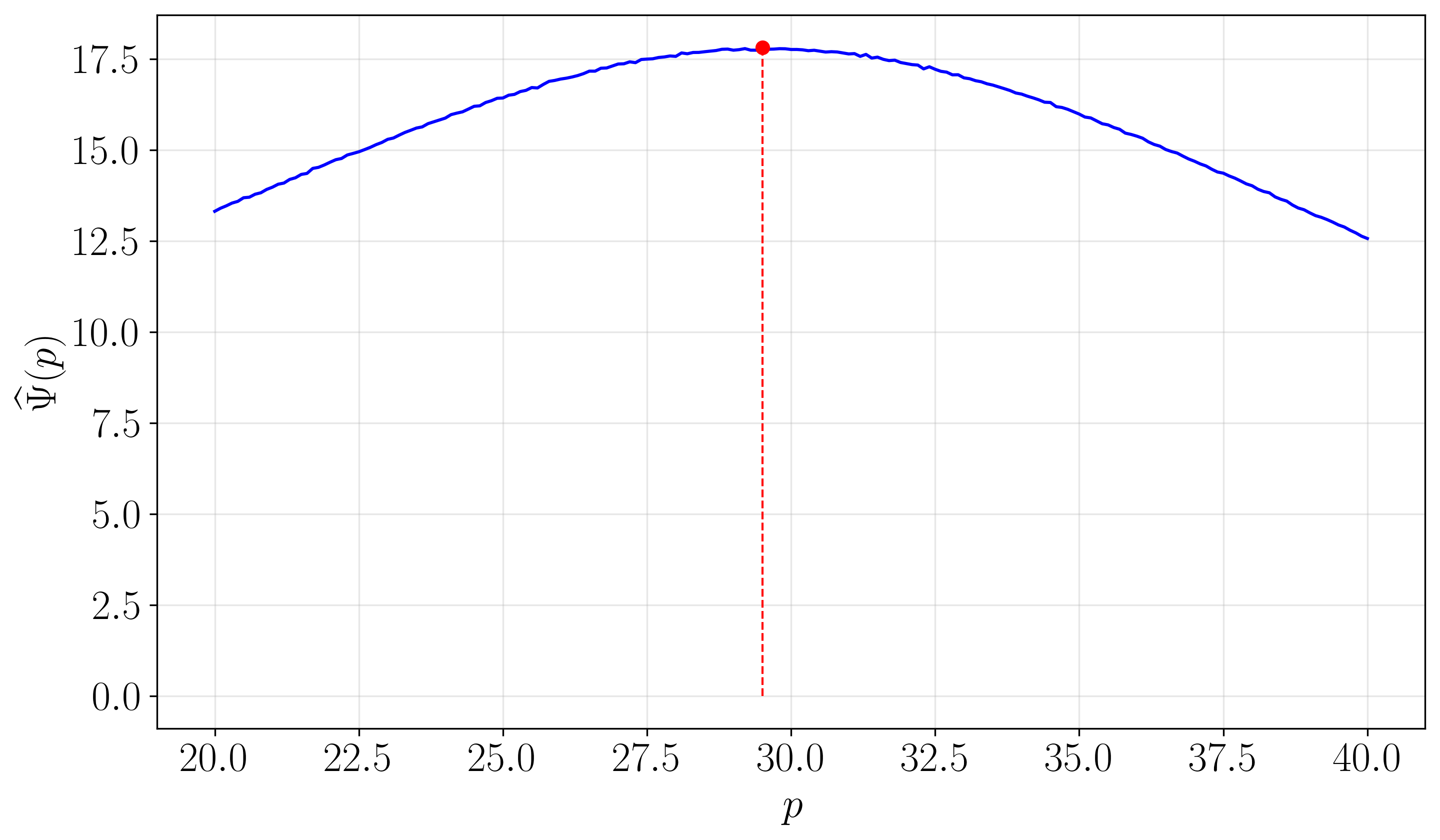}
        \caption{$\widehat{\Psi}(p)$ versus $p$.}
        \label{fig:example_1:psi_plot}
    \end{subfigure}
    \hspace{0.01\textwidth}  
    \begin{subfigure}{0.48\textwidth}
        \centering
        \includegraphics[height=4.5cm]{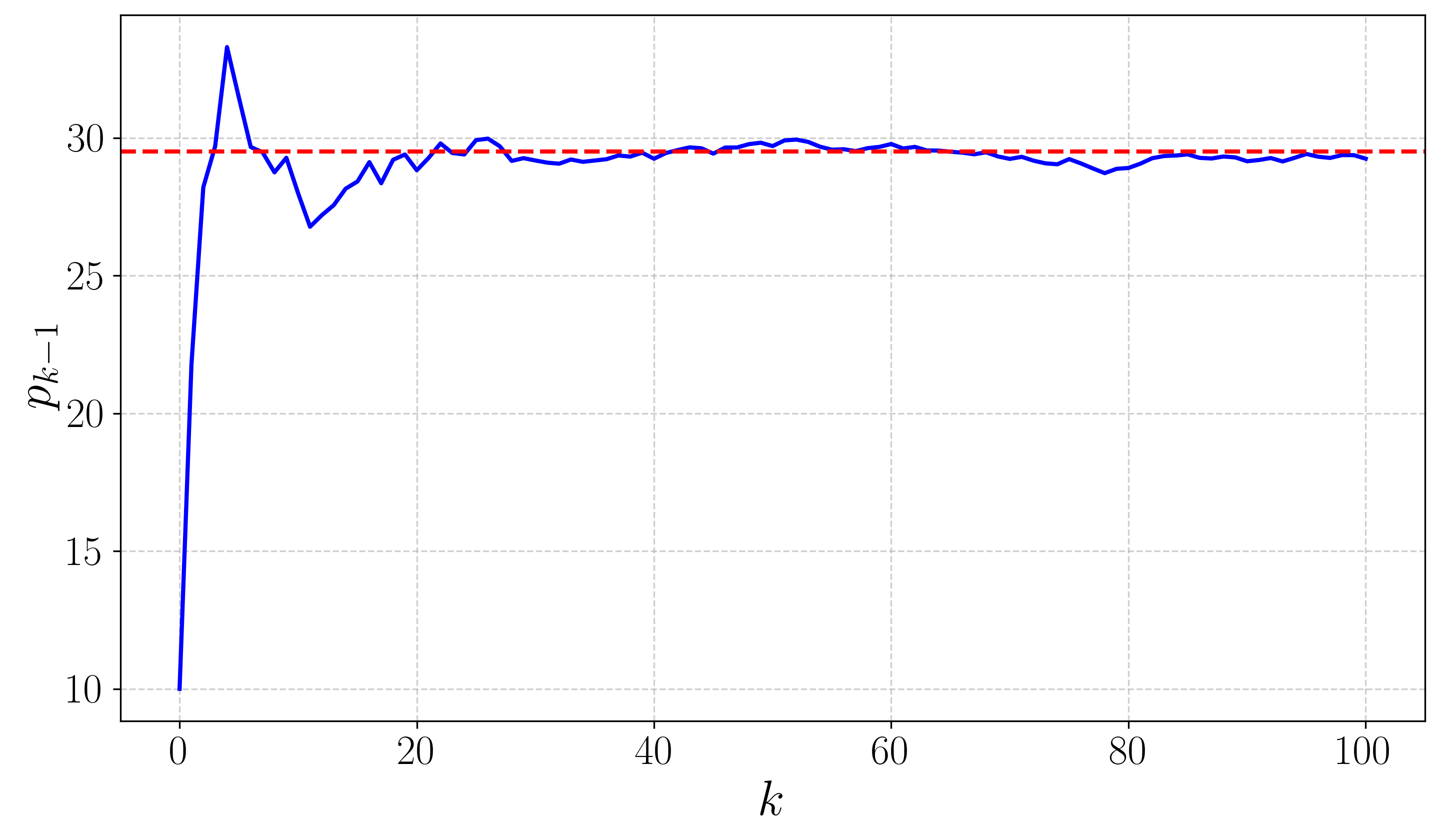}
        \caption{$p_{k-1}$ vs.\ $k$ for $L = 100$ iterations, $p_0 = 10$.}
        \label{fig:example_1:price_iterates}
    \end{subfigure}
    \caption{SGD-based learning algorithm in action for Example 3.}
    \label{fig:SGD:example_1:known_parameters}
\end{figure}

\begin{figure}[t!]
    \centering
    \begin{subfigure}{0.48\textwidth}
        \centering
        \includegraphics[height=4.5cm]{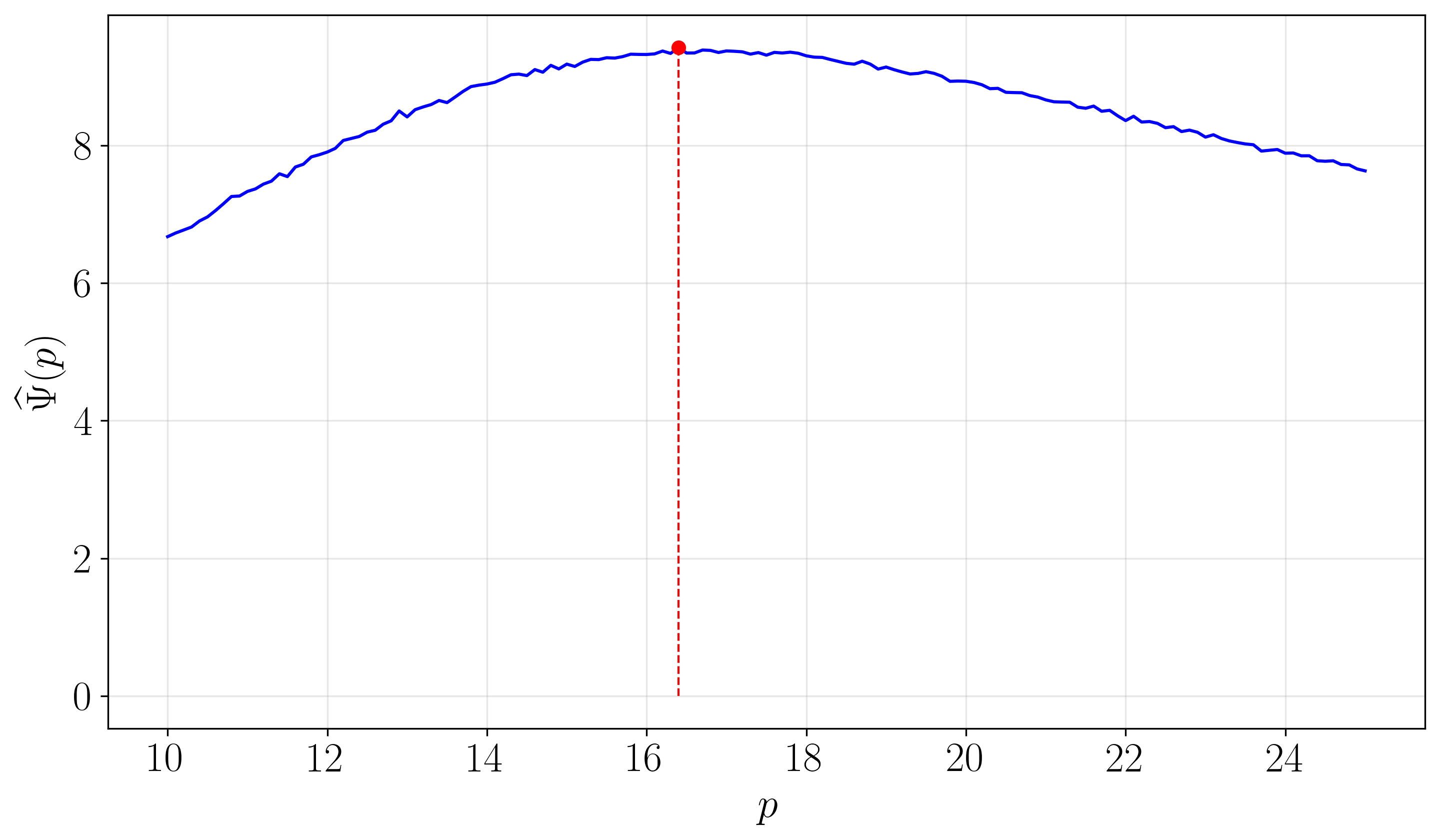}
        \caption{$\widehat{\Psi}(p)$ versus $p$.}
        \label{fig:example_4:psi_plot}
    \end{subfigure}
    \hspace{0.01\textwidth}
    \begin{subfigure}{0.48\textwidth}
        \centering
        \includegraphics[height=4.5cm]{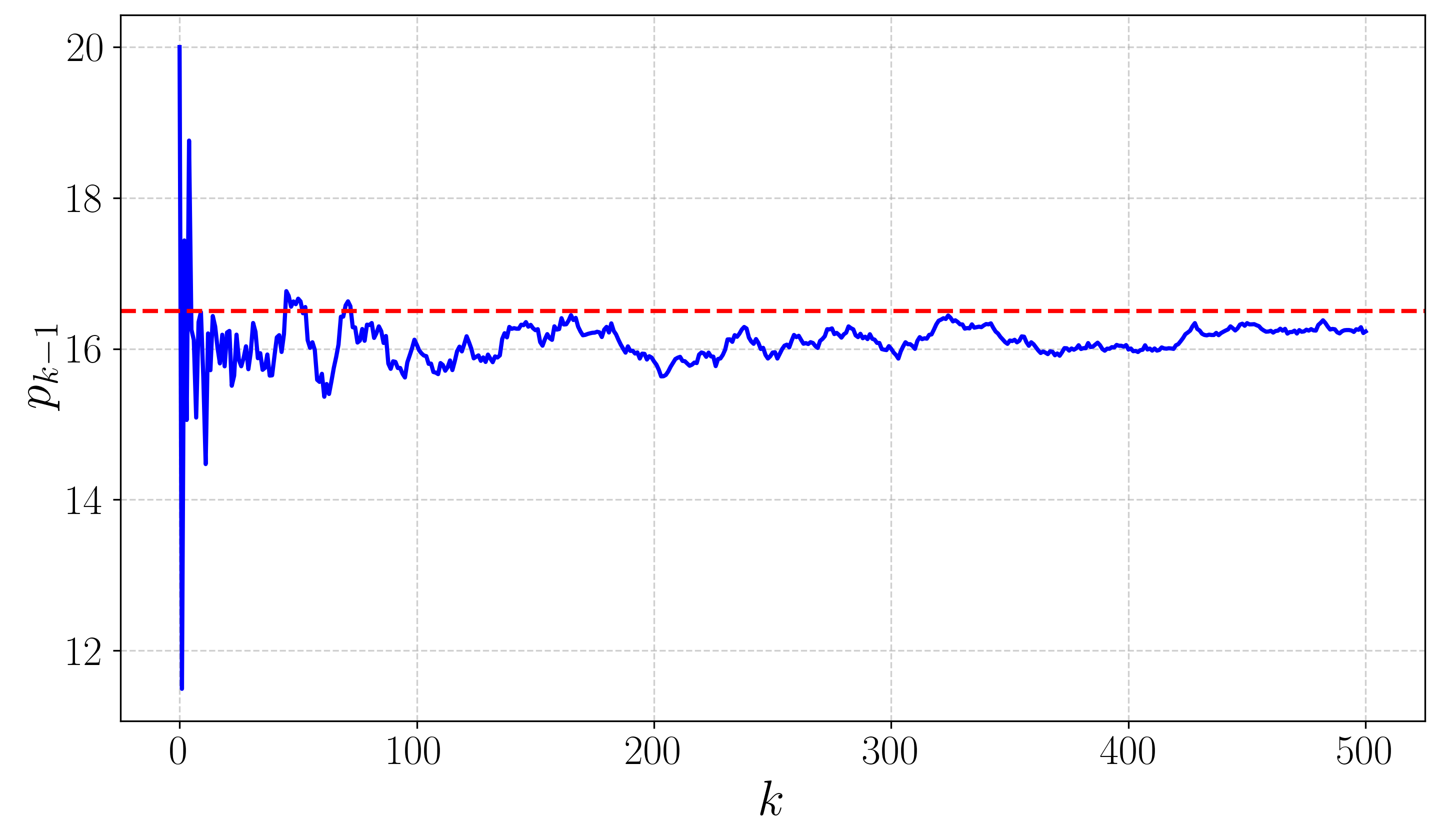}
        \caption{$p_{k-1}$ vs.\ $k$ for $L = 500$ iterations, $p_0 = 20$.}
        \label{fig:example_4:price_iterates}
    \end{subfigure}

    \vspace{0.35cm}
    
    \begin{subfigure}{0.48\textwidth}
        \centering
        \includegraphics[height=5cm]{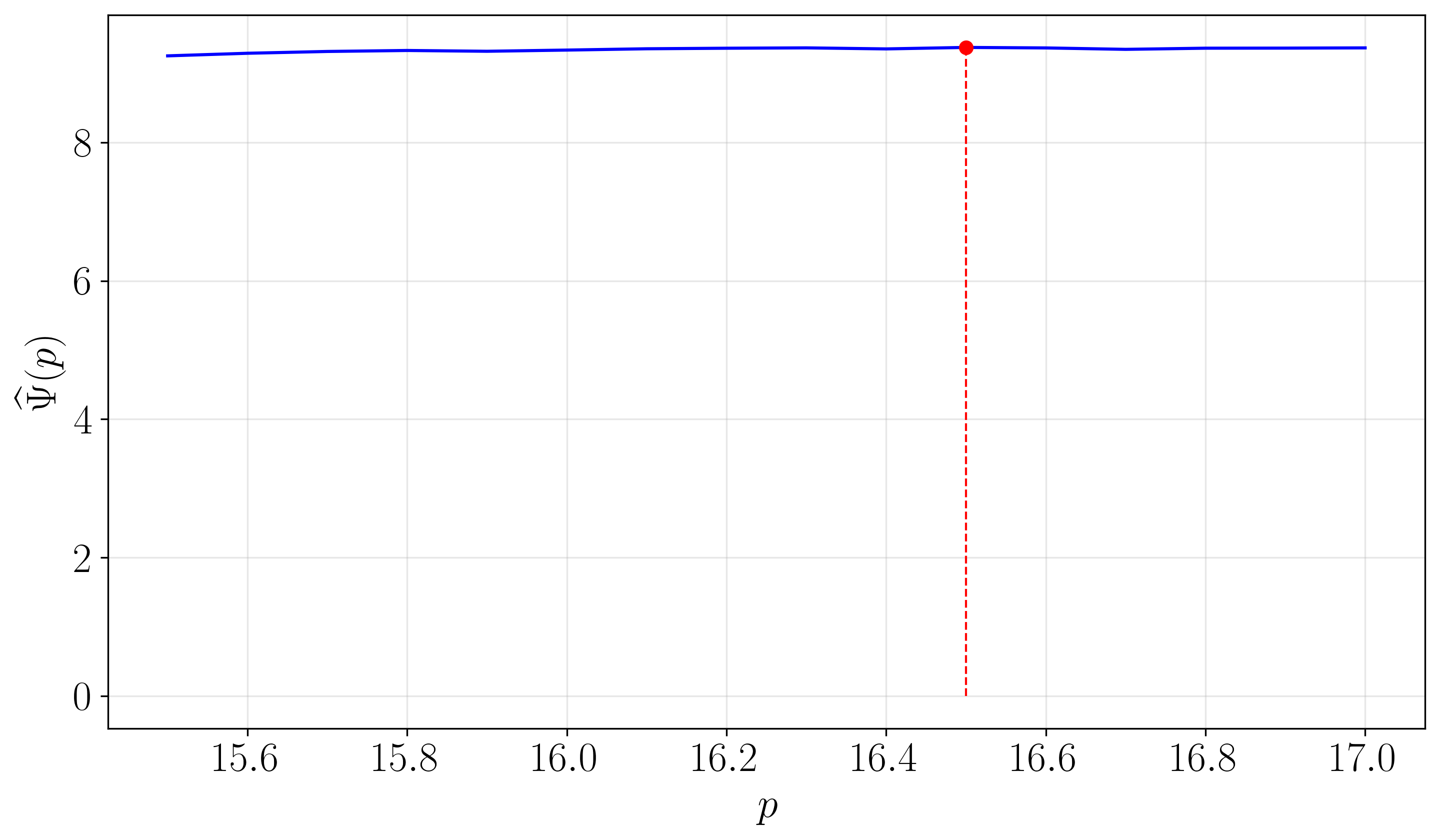}
        \caption{$\widehat{\Psi}(p)$ versus $p$ in the neighbourhood of $p^*$.}
        \label{fig:example_4:psi_plot_zoomed}
    \end{subfigure}

    \caption{SGD-based learning algorithm in action for Example 4.}
    \label{fig:SGD:example_4:known_parameters}
\end{figure}

\subsection{Convergence time comparison}\label{subsection:numerical_experiments:window_size}

In the previous subsection, we verified that the stochastic gradient descent algorithm converges in practice. {We in particular assessed the impact of the distribution of the service times and the joining probability, for fixed window sizes $T_k^*$ and learning rate $\eta_k$.}
A natural next question is: how does the choice of window size affect the convergence rate of the price iterates? 
Specifically, given a fixed time budget $T_{\max}$ for learning the optimal price, what window sizes should be chosen to most effectively and accurately identify it?
This question involves a trade-off.
Smaller windows allow for more iterations of gradient descent within the allotted time, potentially accelerating convergence. 
However, larger windows produce more accurate gradient estimates, which can improve stability and precision. 
Understanding this trade-off is key to optimizing the algorithm’s performance.
In this subsection, we investigate this question through numerical experiments in which we vary $T_k^*$, keeping the service-requirement distribution and the joining probability fixed as in Example 1 from Section~\ref{subsection:numerical_experiments:convergence_rate}.

\begin{itemize}
    \item[$\circ$]
    {\it  Varying window-sizes up to constants.} 
    We begin by varying the time–window sizes up to constant factors. Specifically, we examine the evolution of the price iterates under two growth rates, namely $T^{*}_k = C\sqrt{k}$ and $T^{*}_k = C\log(k+1)$, across a range of values for the constant $C$. The results for both sets of experiments are presented in Figure~\ref{fig:example_5:changing_time_window_const_rate}.

    We first consider in Figure~\ref{fig:example_5:sqrt:price_convergence} the choice $T^{*}_k = C\sqrt{k}$. In Section~\ref{subsection:numerical_experiments:convergence_rate}, our attention was on plots of the price iterates as functions of the iteration number. To better understand how quickly the iterates converge to $p^*$ in \emph{time}, Figure~\ref{fig:example_5:sqrt:price_convergence} presents the price iterates against the actual time. The {top panel of the} figure displays the trajectories produced by the window choices $T^{*}_k = 10\sqrt{k}$, $T^{*}_k = 50\sqrt{k}$, and $T^{*}_k = 100\sqrt{k}$, with the colored bands indicating the dispersion of the corresponding iterates. A wider band indicates more fluctuation, while a narrower band indicates stability.
    {The bottom panel of Figure~\ref{fig:example_5:sqrt:price_convergence} shows} the evolution of the error $\lvert p_{k-1} - p^* \rvert$ as a function of time. From these results, it is apparent that $C=10$ (blue curve) yields the fastest convergence, followed by the $C=50$ (red curve) and $C=100$ (green curve) converging the slowest. 
    
    {Figure~\ref{fig:example_5:log:price_convergence}  reveals that similar conclusions on the speed of convergence can be drawn for the choice $T^{*}_k = C \log(k+1)$ with $C\in\{10,50,100\}$.}

\begin{figure}[t!]
    \centering
    \begin{subfigure}{0.45\textwidth}
        \centering
        \includegraphics[height=6.8cm]{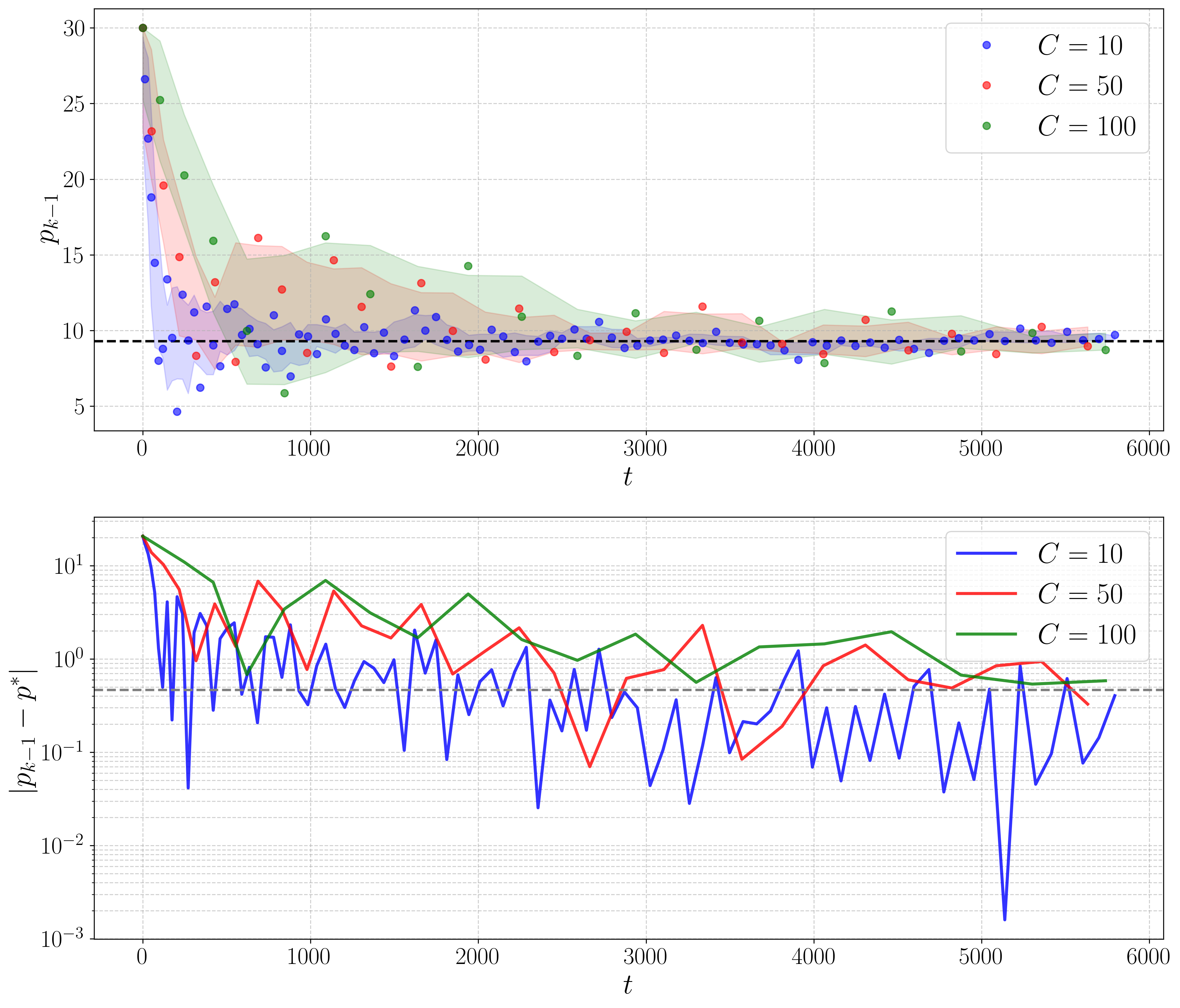}
        \caption{Price iterates versus simulation time for $T^*_k = C\protect\sqrt{k}$, $p_0 = 30$}
        \label{fig:example_5:sqrt:price_convergence}
    \end{subfigure}
    \hspace{0.01\textwidth}
    \begin{subfigure}{0.45\textwidth}
        \centering
        \includegraphics[height=6.8cm]{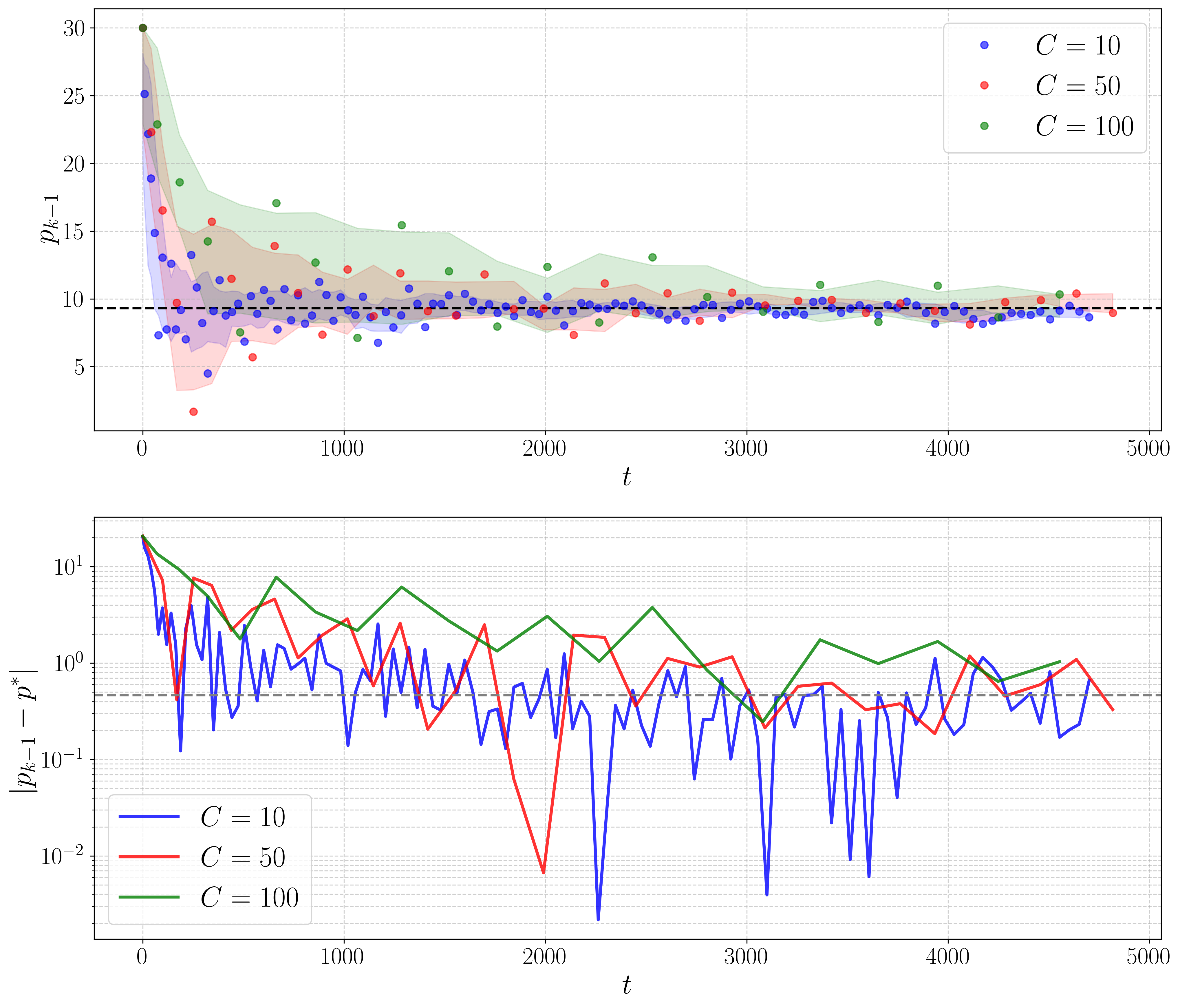}
        \caption{Price iterates versus simulation time for $T^*_k = C\protect\log(k+1)$, $p_0 = 30$}
        \label{fig:example_5:log:price_convergence}
    \end{subfigure}

    \caption{Impact of window sizes on price convergence rates in Example 1.}
    \label{fig:example_5:changing_time_window_const_rate}
\end{figure}

\item[$\circ$]{\it Varying window size rates.} 
Next, we vary the {\it rate} at which the window sizes grow by studying the behavior of the price iterates over time for window functions of the form $T^{*}_k = C\,k^\alpha$, considering multiple values of the exponent $\alpha$. The corresponding plots are shown in Figure~\ref{fig:example_5:changing time window:variable rate}.
Once again, similar to the observation in Figure \ref{fig:example_5:changing time window:variable rate}, we observe that $\alpha = 0.1$ (blue curve) yields the fastest convergence followed by $\alpha = 0.3$ (red curve) and $\alpha = 0.5$ (green curve) being the slowest.
\end{itemize}

The above two experiments highlight an inherent trade-off: increasing $C$ and $\alpha $ leads to longer time windows, which provide more data per iteration and thus yield more accurate gradient estimates. However, this comes at the cost of fewer updates within a fixed time budget. From Figures~\ref{fig:example_5:changing_time_window_const_rate}--\ref{fig:example_5:changing time window:variable rate}, it becomes evident that  smaller values of $C$ and $\alpha$ — i.e., slower growth of the window size — lead to faster convergence of the price iterates to the optimal price.

\begin{figure}
    \centering
    \includegraphics[width=0.6\linewidth]{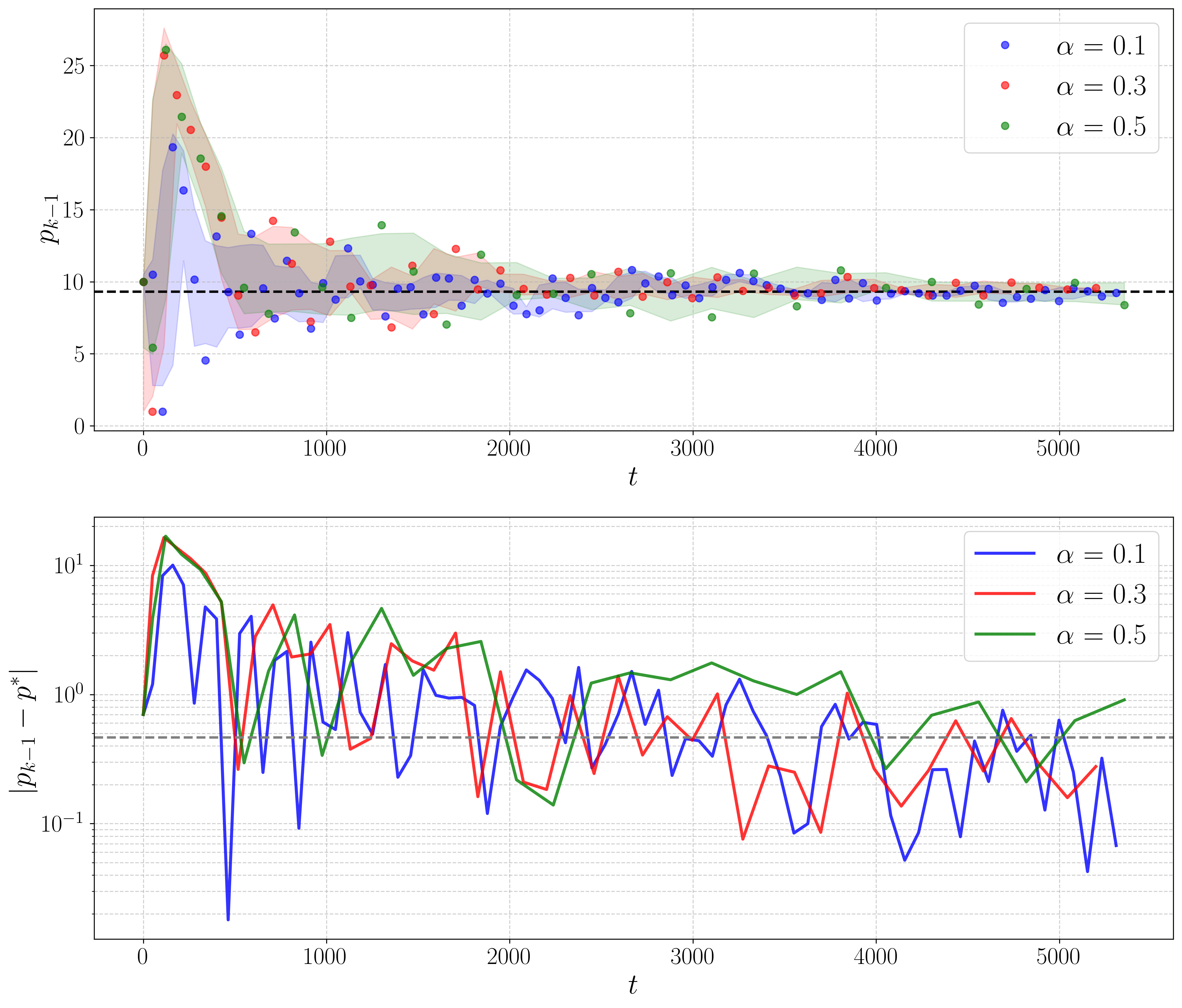}
    \caption{Price iterates versus simulation time for $T^*_k = 50 \,k^{\alpha}$, $p_0 = 10$.}
    \label{fig:example_5:changing time window:variable rate}
\end{figure}

\section{Concluding remarks}

In this paper, we analyzed a dynamic pricing problem in a single-server queueing system, where the objective was to maximize the expected revenue per unit time. Customers were allowed to balk in response to high congestion, which introduced a non-trivial interaction between the pricing policy and the effective arrival process. Importantly, the service provider had access only to data from effective arrivals — those customers who chose to join the queue — making the system partially observable and inherently state-dependent.

To address this challenge, we developed a stochastic gradient descent algorithm that dynamically adjusted the price based solely on observable information. A key contribution of our work was the design of a novel IPA-based procedure, which consistently estimated the stationary effective arrival rate. This estimator was then used to construct an iterative learning algorithm that, under mild regularity conditions, was shown to converge asymptotically to the optimal price. Numerical experiments further confirmed the effectiveness and robustness of the proposed algorithm.

\medskip

While our work makes several significant contributions, there remain multiple avenues for future research. We highlight a few key directions here:
\begin{itemize}
    \item[$\circ$]
    A natural next step is to extend the analysis to multi-server queueing systems. This poses significant challenges, as the coupling arguments developed in Section~\ref{sec:queue_coupling} for queues with different initial workloads do not straightforwardly generalize to multi-server settings (cf. the observations in \cite{BodasJacobovic2024}). Substantial adaptations of the theoretical framework would therefore be required.
 \item[$\circ$]Another important direction is to consider scenarios where both the arrival rate and the customers’ disutility threshold distribution are unknown. If the disutility threshold distribution is parametric with unknown parameters, existing estimation procedures, such as those developed in \cite{ BodasMandjesRavner2024,InoueRavnerMandjes2023}, could be incorporated. However, this would introduce an additional layer of complexity, as the algorithm would need to update not only the price but also estimates of the arrival rate and patience distribution at each iteration to compute the gradient estimator accurately. In cases where the disutility threshold distribution is non-parametric, this problem naturally invites approaches from reinforcement learning.
 \item[$\circ$] Assumption \ref{assumption:Psi} requires that $\Psi(\cdot)$ is a strongly concave function. An open problem of theoretical interest is whether Theorem \ref{thm:price_convergence_rate} and thereby Theorem \ref{thm:regret:upper_bound}, or some weaker versions of them, can also be established when $\Psi(\cdot)$ is not strongly concave.
\end{itemize}


\bibliographystyle{plain}

\newpage

\appendix

\begin{center}
    \Large{{\textbf{--- APPENDICES ---}}}
\end{center}

\section{Proofs} \label{appendix:proofs}

\subsection{Proof of Proposition \ref{propn:formula:Psi}} \label{appendix:proof:propn:formula:Psi}

\begin{proof}[Proof of Proposition~\ref{propn:formula:Psi}]
     Let $N(\cdot)$ be the counting process corresponding to effective arrivals. Equation \eqref{eqn:formula:Psi:1} follows because
    \begin{align*}
        \Psi(p) = \lim_{T \rightarrow \infty} \frac{p \ N(T)}{T} = p \lim_{T \rightarrow \infty} \frac{N(T)}{T} = p \ \Lambda \exptn\big[H\left(p, W_{\infty}(p)\right)\big].
    \end{align*}
    Recall that for $i \geq 0$, the queue dynamics are dictated by the following recursions:
    \begin{align*}
        \overline{W}_{i} = \underline{W}_{i} + S_i, \ A_{i+1} = F^{-1}_p\left(\zeta_{i+1}; \overline{W}_i\right), \ \underline{W}_{i+1} = \left(\overline{W}_i - A_{i+1}\right)^{+}.
    \end{align*}
     {Since the Markov chain $\big\{\overline{W}_i\big\}_{i \geq 1}$ is positive recurrent, and $A_{i+1}$ is a measurable function of $\big(\overline{W}_i, \zeta_{i+1}\big)$, it follows that $\big(\overline{W}_i, A_{i+1}\big)$ is an ergodic sequence.}
    Hence, the ergodic theorem entails that
    \begin{align*}
        \lim_{T \rightarrow \infty} \frac{\sum_{i=1}^{N(T)} A_i}{N(T)} = \exptn\left[A_{\infty}(p)\right].
    \end{align*}
    Then Equation \eqref{eqn:formula:Psi:2} follows due to
    \begin{align*}
        \Psi(p) = p \lim_{T \rightarrow \infty} \frac{N(T)}{T} = p \ \left({\displaystyle \lim_{T \rightarrow \infty} \frac{T}{N(T)}}\right)^{-1} = p \ \left({\displaystyle \lim_{T \rightarrow \infty} \frac{\sum_{i=1}^{N(T)} A_i}{N(T)}}\right)^{-1} = \frac{p}{\exptn\left[A_{\infty}(p)\right]}.
    \end{align*}This proves the stated. 
\end{proof}

\subsection{Proof corresponding to Example \ref{example: Psi and H}} \label{appendix:proof:example:workload der:bounds}

For ease of writing, write $\Lambda_p := ({\Lambda}/{\theta_2})\, e^{-\theta_1 p}$. 
Due to $ 1 \geq  F_p(W;W) > 1-e^{-\Lambda_p}$ we have that
\[1 > \prob\left(\zeta < F_p(W;W)\right) > 1-e^{-\Lambda_p}.\]
Then, as per the definition of $\Xi[\zeta]$ in Example \ref{example: Psi and H},
\begin{align*}
    \exptn\left[\big(\Xi[\zeta]\big)^m\right] 
    &= \exptn\left[\big(\Xi[\zeta]\big)^m \,\big|\, \zeta < F_p(W;W) \right] \cdot \prob\left(\zeta < F_p(W;W)\right) \\
    &\quad + \exptn\left[\big(\Xi[\zeta]\big)^m \,\big|\, \zeta \geq F_p(W;W) \right] \cdot \prob\left(\zeta \geq F_p(W;W)\right) \\
    &\leq \left( \int_{0}^{1 - e^{-\Lambda_p}} 
        \left(1 - \frac{1}{1 - {\Lambda_p}/{\log(1 - x)}}\right)^m \mathrm{d}x \right)
        + e^{-\Lambda_p} \\
    &= \int_{0}^{1 - e^{-\Lambda_p}} 
        \left( \frac{\Lambda_p}{\Lambda_p - \log(1 - x)} \right)^m \mathrm{d}x
        + e^{-\Lambda_p}.
\end{align*}

Applying the substitution $1-x = e^{\Lambda_p (1-z)}$, and correspondingly, $-\mathrm{d}x = -\Lambda_p e^{\Lambda_p (1-z)} \ \mathrm{d}z$, the integral in the preceding display can be rewritten as 
\begin{align*}
    \int_{1}^{2} \Bigg(\frac{\Lambda_p}{\Lambda_p - \Lambda_p(1-z)}\Bigg)^m \Lambda_p e^{\Lambda_p(1-z)} \ \mathrm{d}z = \Lambda_p e^{\Lambda_p} \int_{1}^{2} \frac{e^{-\Lambda_p z}}{z^m} \ \mathrm{d}z = \Lambda_p e^{\Lambda_p} \,E_m(\Lambda_p).
\end{align*}
where we have defined $E_m(t): = \int_{1}^{2} ({e^{-tx}}/{x^m}) \, \mathrm{d}x.$
We thus conclude that
\begin{align*}
    \exptn\big[\left(\Xi[\zeta]\right)^m\big] \leq \Lambda_p e^{\Lambda_p} E_m(\Lambda_p) + e^{-\Lambda_p}.
\end{align*}
With $
    f_m(t) := t e^t E_m(t) + e^{-t}$, observe that for any fixed $t > 0$, the sequence $\left\{f_m(t)\right\}_{m \geq 1}$ is strictly decreasing in $m$. Therefore, it suffices to prove that $f_1(t) \leq 1$ for all $t > 0$, which follows due to
\begin{align*}
    f_1(t) = t e^t \int_{1}^{2} \frac{e^{-tx}}{x} \ \mathrm{d}{x} + e^{-t} < t e^t \int_{1}^{2} e^{-tx} \ \mathrm{d}{x} + e^{-t} = {t} e^t \frac{\left(e^{-t} - e^{-2t}\right)}{{t}} + e^{-t} = 1.
\end{align*}

\subsection{Proof of Lemma~\ref{lemma:example:gradient_bounds}} \label{appendix:proof:lemma:example:gradient_bounds}

\begin{proof}[Proof of Lemma~\ref{lemma:example:gradient_bounds}]
    First, we show via induction that for all $n \geq 0$,
    \[ \nabla_p \overline{W}_n \in \bigg[-\frac{\theta_1}{\theta_2}, 0\bigg].\]
    The base case $\nabla_p \overline{W}_0 = 0$ trivially satisfies the hypothesis. Let us assume that for $1 \leq k \leq n$, the induction hypothesis holds true. 
    Let $\zeta_{n+1} \sim \text{Uniform}[0,1]$, the seed for the interarrival time between the $n$-th and $(n+1)$-st customer, be sampled independently of $\overline{W}_n$. If $\zeta_{n+1} \geq F_p\big(\overline{W}_n;\overline{W}_n\big)$, then the $(n+1)$-st customer starts a new regeneration cycle, and $\nabla_p \overline{W}_{n+1} = 0$, and we are done. On the other hand, if $\zeta_{n+1} < F_p\big(\overline{W}_n;\overline{W}_n\big)$, then, by Equation \eqref{eqn:price der_inverse_F_p(zeta;w)}, 
    \begin{align*}
        \nabla_p \overline{W}_{n+1} &= \nabla_p \overline{W}_n - \Big(\theta_1 + \theta_2 \ \nabla_p \overline{W}_n\Big) \bigg(\theta_2 - e^{-\theta_1 p-\theta_2 \overline{W}_n} \frac{\Lambda}{\log(1-\zeta_{n+1})}\bigg)^{-1}.
    \end{align*}
    Denoting
    \begin{align*}
        U \equiv U\big(\zeta_{n+1}, \overline{W}_n\big) :=  -e^{-\theta_1 p-\theta_2 \overline{W}_n} \frac{\Lambda}{\log(1-\zeta_{n+1})} > 0,
    \end{align*}
    we can rewrite 
    \begin{align*}
        \nabla_p\overline{W}_{n+1} = \nabla_p\overline{W}_n \bigg(\frac{U}{\theta_2+U}\bigg) - \frac{\theta_1}{\theta_2+U}.
    \end{align*}
    which is an increasing linear map. As a result,
    \begin{align*}
        \nabla_p \overline{W}_{n+1} \in \bigg[-\frac{\theta_1}{\theta_2} \bigg(\frac{U}{\theta_2+U}\bigg) - \frac{\theta_1}{\theta_2+U}, \ - \frac{\theta_1}{\theta_2+U}\bigg] = \bigg[-\frac{\theta_1}{\theta_2}, -\frac{\theta_1}{\theta_2+U}\bigg] \subset \bigg[-\frac{\theta_1}{\theta_2}, 0 \bigg].
    \end{align*}
    Our induction proof is complete. Again referring to Equation \eqref{eqn:price der_inverse_F_p(zeta;w)} and the result we just proved, it now immediately follows that
    \begin{align*}
        0 \leq \nabla_p F^{-1}_p\big(\zeta_{n+1};\overline{W}_n\big) &= \Big(\theta_1 + \theta_2 \ \nabla_p\overline{W}_n\Big) \bigg(\theta_2 - e^{-\theta_1 p-\theta_2 \overline{W}_n} \frac{\Lambda}{\log(1-\zeta_{n+1})}\bigg)^{-1} 
        \\
        &\leq \frac{\Big(\theta_1 + \theta_2 \ \nabla_p\overline{W}_n\Big)}{\theta_2} \leq \frac{\theta_1}{\theta_2}.
    \end{align*}
    This completes the proof. 
\end{proof}

\subsection{Proof of Proposition~\ref{propn:model:theoretical_results}} \label{appendix:proof:propn:model:theoretical_results}

\begin{proof}[Proof of Proposition~\ref{propn:model:theoretical_results}]
    \begin{enumerate}
        \item[(a)] For now, let us restrict our analysis to the set \[{\mathscr E} := \big\{\omega: \overline{W}^{1}_0(\omega) \geq \overline{W}^{2}_0(\omega)\big\} \subseteq \Omega;\] the analysis for the other case ${\mathscr E}^c = \big\{\omega: \overline{W}^{1}_0(\omega) < \overline{W}^{2}_0(\omega) \big\}$ follows analogously.
        Let $g(x) = x - F_p^{-1}(\zeta;x)$. Then, by Assumption \ref{assumption: workload der: bounds}, we have that $g(\cdot)$ is an increasing function, as 
        $g'(x) = 1 - \nabla_w F_p^{-1}(\zeta;x) > 0$. 
        By the Lindley recursion it follows that $\overline{W}^{1}_n(\omega) \geq \overline{W}^{2}_n(\omega)$ for all $n \geq 1$, and $\omega \in {\mathscr E}$. In order to avoid cumbersome notation, for the rest of the proof, we will suppress $\omega$ in our random variables.
        
        \textbf{Case I: } Let $0 < \zeta_{k+1} \leq F_p\big(\overline{W}^{2}_k; \overline{W}^{2}_k\big)$, i.e., $\overline{W}^{1}_{k+1} \geq \overline{W}^{2}_{k+1} \geq 0$. Then, by the mean-value theorem, there exists a random variable $\widetilde{W}_k \in \big(\overline{W}^{2}_k, \overline{W}^{1}_k\big)$ such that
       \begin{align*}
            \overline{W}^{1}_{k+1} - \overline{W}^{2}_{k+1} 
            &= \overline{W}^{1}_k - F^{-1}_p\left(\zeta_{k+1}; \overline{W}^{1}_k\right) 
            - \overline{W}^{2}_k + F^{-1}_p\left(\zeta_{k+1}; \overline{W}^{2}_k\right) \\
            &= \left( \overline{W}^{1}_k - \overline{W}^{2}_k \right) 
            - \left( F^{-1}_p(\zeta_{k+1}; \overline{W}^{1}_k) - F^{-1}_p(\zeta_{k+1}; \overline{W}^{2}_k) \right) \\
            &= \left( \overline{W}^{1}_k - \overline{W}^{2}_k \right) 
            \left( 1 - \nabla_{w} F^{-1}_p(\zeta_{k+1}; \widetilde{W}_k) \right) \\
            &\leq \left( \overline{W}^{1}_k - \overline{W}^{2}_k \right) \,
            \Xi[\zeta_{k+1}],
        \end{align*}
        where the last inequality follows by Assumption \ref{assumption: workload der: bounds}.
        
        \textbf{Case II: } Now consider $F_p(\overline{W}^{1}_k; \overline{W}^{1}_k) > \zeta_{k+1} > F_p(\overline{W}^{2}_k; \overline{W}^{2}_k)$, i.e., $\overline{W}^{1}_{k+1} > \overline{W}^{2}_{k+1} = 0$. Then
       \begin{align*}
            \overline{W}^{1}_{k+1} - \overline{W}^{2}_{k+1} 
            &= \overline{W}^{1}_k - F^{-1}_p(\zeta_{k+1}; \overline{W}^{1}_k) \\
            &\leq \overline{W}^{1}_k - F^{-1}_p(\zeta_{k+1}; \overline{W}^{1}_k) 
            \underbrace{- \overline{W}^{2}_k + F^{-1}_p(\zeta_{k+1}; \overline{W}^{2}_k)}_{> 0} \\
            &\leq \left( \overline{W}^{1}_k - \overline{W}^{2}_k \right) 
            \,\Xi[\zeta_{k+1}].
        \end{align*}
        
        \textbf{Case III: } Finally, if $1 > \zeta_{k+1} > F_p(\overline{W}^{1}_k; \overline{W}^{1}_k)$, then $\overline{W}^{1}_{k+1} = \overline{W}^{2}_{k+1} = 0$, and clearly,
        \begin{align*}
            0 = \overline{W}^{1}_{k+1} - \overline{W}^{2}_{k+1} \leq (\overline{W}^{k}_1 - \overline{W}^{2}_k) \,\Xi[\zeta_{k+1}].
        \end{align*}
        We conclude that in all three cases, which encompass all possibilities $\zeta_{k+1} \in(0,1)$, and for any $m \geq 0$, we have that
        \begin{align*}
            \left(\overline{W}^{1}_{k+1} - \overline{W}^{2}_{k+1}\right)^m \leq \left(\overline{W}^{1}_k - \overline{W}^{2}_k\right)^m \,\big(\Xi[\zeta_{k+1}]\big)^m.
        \end{align*}
        Applying this inequality recursively from $k = 0$ to $k = n-1$, we conclude that for all $\omega \in {\mathscr E}$,
        \begin{align*}
            \left(\overline{W}^{1}_{n} - \overline{W}^{2}_{n}\right)^m \leq \left(\overline{W}^{1}_{0} - \overline{W}^{2}_{0}\right)^m \ \prod_{i=1}^{n} \big(\Xi[\zeta_i]\big)^m.
        \end{align*}
        Similarly, conditioning on the complementary event $\left\{\overline{W}^{1}_0 < \overline{W}^{2}_0 \right\}$, we obtain
        \begin{align*}
            \left(\overline{W}^{2}_{n} - \overline{W}^{1}_{n}\right)^m \leq \left(\overline{W}^{2}_{0} - \overline{W}^{1}_{0}\right)^m \ \prod_{i=1}^{n} \big(\Xi[\zeta_i]\big)^m.
        \end{align*}
        Recall that $(\Omega, \sigmaf, \prob)$ is the probability space on which the two queues have been coupled. Therefore, adding up both contributions,
        \begin{align*}
            \exptn\left\vert \overline{W}^{1}_n - \overline{W}^{2}_n \right\vert^m &= \int_{\left\{\omega: \ \overline{W}^{1}_0 \geq \overline{W}^{2}_0\right\}} \left(\overline{W}^{1}_n - \overline{W}^{2}_n\right)^m \ \mathrm{d}\prob(\omega) \ + \int_{\left\{\omega: \ \overline{W}^{1}_0 < \overline{W}^{2}_0\right\}} \left(\overline{W}^{2}_n - \overline{W}^{1}_n\right)^m \ \mathrm{d}\prob(\omega)
            \\
            &\leq \int_{\left\{\omega: \ \overline{W}^{1}_0(\omega) \geq \overline{W}^{2}_0(\omega)\right\}} \left(\overline{W}^{1}_{0} - \overline{W}^{2}_{0}\right)^m \ \prod_{i=1}^{n} \big(\Xi[\zeta_i]\big)^m \ \mathrm{d}\prob(\omega)
            \\
            &\hspace{1cm} +\int_{\big\{\overline{W}^{1}_0 < \overline{W}^{2}_0\big\}} \left(\overline{W}^{2}_{0} - \overline{W}^{1}_{0}\right)^m \ \prod_{i=1}^{n} \big(\Xi[\zeta_i]\big)^m \mathrm{d}\prob
            \\
            &= \int_{\Omega} \left\vert \overline{W}^{1}_0 - \overline{W}^{2}_0 \right\vert^m \ \prod_{i=1}^{n}  \big(\Xi[\zeta_i]\big)^m\ \mathrm{d}\prob.
        \end{align*}
        Using the fact that $\big\{ \zeta_i \big\}_{i \geq 1}$ are i.i.d., and independent of $\overline{W}^{1}_0$ and $\overline{W}^{2}_0$, we conclude
        \begin{align*}
            \exptn \left\vert \overline{W}^{1}_{n} - \overline{W}^{2}_{n} \right\vert^m \leq \exptn \left\vert \overline{W}^{1}_{0} - \overline{W}^{2}_{0} \right\vert^m  \prod_{i=1}^{n}  \exptn \big[ \big(\Xi[\zeta_i]\big)^m\big] = \gamma_m^n \exptn \left\vert \overline{W}^{1}_{0} - \overline{W}^{2}_{0} \right\vert^m.
        \end{align*}

        \item[(b)] Once again, we restrict our analysis to the event $\big\{\overline{W}^{1}_0 \geq \overline{W}^{2}_0 \big\}$.
        By the same reasoning as before, this implies that $\overline{W}^{1}_i \geq \overline{W}^{2}_i$ for all $i \geq 1$.
        Recall that for $n \geq 1$, 
        \begin{align*}
            A^{i}_n = F^{-1}_p\left(\zeta_{n}; \overline{W}^{i}_{n-1}\right).
        \end{align*}
        By the mean-value theorem in combination with Assumption \ref{assumption: workload der: bounds}, we have that for a random variable $\widetilde{W}_j \in \big(\overline{W}^{2}_j, \overline{W}^{1}_j\big)$,
        \begin{align}\label{eqn:A_1-A_2:upper_bound}
            A^{1}_j - A^{2}_j = F^{-1}_p\left(\zeta_{j+1}; \overline{W}^{1}_j\right) - F^{-1}_p\left(\zeta_{j+1}; \overline{W}^{2}_j\right) &= \left(\overline{W}^{1}_j - \overline{W}^{2}_j\right) \ \nabla_{w} F^{-1}_p\left(\zeta_{j+1}; \widetilde{W}_j \right) \notag
            \\
            &< \left(\overline{W}^{1}_j - \overline{W}^{2}_j\right). 
        \end{align}
       As a result, and by applying the conclusion of part (a),
        \begin{align*}
            \exptn\left\vert A^{1}_j - A^{2}_j \right\vert < \exptn\left\vert \overline{W}^{1}_j - \overline{W}^{2}_j \right\vert \leq \gamma_1^j \exptn\left\vert \overline{W}^{1}_0 - \overline{W}^{2}_0 \right\vert.
        \end{align*}
        The $n$-th effective arrival time in system $i=1,2$ is given by $\widetilde{A}^{i}_n = \sum_{j=1}^{n} A^{i}_j$. By Equation \eqref{eqn:A_1-A_2:upper_bound},
        \begin{align*}
            0 \leq \widetilde{A}^{1}_n - \widetilde{A}^{2}_n = \sum_{j=1}^{n} \left(A^{1}_j - A^{2}_j \right)<\sum_{j=0}^{n-1} \left(\overline{W}^{1}_j - \overline{W}^{2}_j\right) .
        \end{align*}
        Again using the result from part (a),
        \begin{align*}
            \exptn\big\vert \widetilde{A}^{1}_n - \widetilde{A}^{2}_n \big\vert &= \int_{\big\{\overline{W}^{1}_0 > \overline{W}^{2}_0 \big\}} \left(\widetilde{A}^{1}_n - \widetilde{A}^{2}_n\right) \ \mathrm{d}\prob + \int_{\big\{\overline{W}^{2}_0 \geq \overline{W}^{1}_0 \big\}} \left(\widetilde{A}^{2}_n - \widetilde{A}^{1}_n\right) \ \mathrm{d}\prob
            \\
            &\leq \int_{\big\{\overline{W}^{1}_0 > \overline{W}^{2}_0 \big\}} \sum_{j=0}^{n-1} \left(\overline{W}^{1}_j - \overline{W}^{2}_j\right) \ \mathrm{d}\prob + \int_{\big\{ \overline{W}^{2}_0 \geq \overline{W}^{1}_0 \big\}} \sum_{j=0}^{n-1} \left(\overline{W}^{2}_j - \overline{W}^{1}_j\right) \ \mathrm{d}\prob
            \\
            &= \int_{\Omega} \sum_{j=0}^{n-1} \left\vert \overline{W}^{1}_j - \overline{W}^{2}_j \right\vert \ \mathrm{d}\prob = \sum_{j=0}^{n-1} \exptn \left\vert \overline{W}^{1}_j - \overline{W}^{2}_j \right\vert \leq \sum_{j=0}^{n-1} \gamma_1^j \exptn\left\vert \overline{W}^{1}_0 - \overline{W}^{2}_0 \right\vert
            \\
            &= \bigg( \frac{1 - \gamma_1^n}{1-\gamma_1} \bigg) \exptn\left\vert \overline{W}^{1}_0 - \overline{W}^{2}_0 \right\vert.
        \end{align*}

        \item[(c)] We again restrict our analysis to the event $\big\{\overline{W}^{1}_0 \geq \overline{W}^{2}_0 \big\}$.
        Define\begin{align*}
         K := \inf\left\{n \geq 1 :  \underline{W}^{1}_n = \underline{W}^{2}_n \right\},   
        \end{align*} i.e., $K$ is the effective arrival number starting when the two queues are \textit{coupled with a time-lag}, where this time-lag is given by $\Delta := \big(\widetilde{A}^{1}_{K} - \widetilde{A}^{2}_{K} \big)$. Note that $N_1\left(0,t\right] - N_2\left(0,t\right]$ is increasing for $t \leq \widetilde{A}^{1}_{K}$, and is constant for $t > \widetilde{A}^{1}_{K}$. Then, for $t > \widetilde{A}^{1}_{K}$,
        \begin{align*}
            N_{1}\big(0,t \big] &= N_{1}\big(0, \widetilde{A}^{1}_{K}\big] + N_{1}\big(\widetilde{A}^{1}_{K},t \big] = K + N_{1}\big(\widetilde{A}^{1}_{K},t \big],
            \\
            N_2\big(0, t\big] &= N_2\big(0, \widetilde{A}^{2}_{K}\big] + N_{2}\big(\widetilde{A}^{2}_{K},t \big] = K + N_{2}\big(\widetilde{A}^{2}_{K},t \big] = K + N_{2}\big(\widetilde{A}^{2}_{K},t - \Delta \big] + N_2\big(t-\Delta, t\big].
        \end{align*}
        Observe that, since the two queues are coupled with a time-lag,
        \begin{align*}
            N_{1}\big(\widetilde{A}^{1}_{K},t \big] = N_{2}\big(\widetilde{A}^{2}_{K},t - \Delta \big].
        \end{align*}
        We then have that, for any $t > 0$,
        \begin{align}
            N_2\left(0,t\right] - N_1\left(0,t\right] = N_2\left(t-\Delta, t\right] \leq_{\text{st}} \overline{N}\left(\Delta\right),\label{sd}
        \end{align}
        where the random variable $\overline{N}(\Delta)$ records the number of arrivals in an interval of length $\Delta$, under a Poisson stream of arrivals with rate $\Lambda H\big(\underline{p}, 0\big)$. The stochastic dominance follows from an evident coupling argument: each interarrival time in the second queue stochastically dominates an $\text{Exp}\big(\Lambda\, H\big(\underline{p}, 0\big)\big)$ random variable. 
        Similarly, under the complementary event $\big\{\overline{W}^{1}_0 < \overline{W}^{2}_0 \big\}$ we have
        \eqref{sd} as well.
        In conclusion, 
         $\big\vert N_1\left(0,T^*\right] - N_2\left(0,T^*\right]\big\vert \leq_{\text{st}} \overline{N}\left(\Delta\right).$ 
         
        We have that $\overline{N}(\Delta)$ conditional on $\Delta=\delta$ has a Poisson distribution with parameter $\Lambda \,H\big(\underline{p}, 0\big) \, \delta$.
        Thus, by appealing to part~(b) of the present proposition,
        \begin{align*}
            \exptn \left\vert N_1\left(0, T^*\right] - N_2\left(0, T^*\right] \right\vert &\leq
            \exptn\Big[\overline{N}\big(\Delta\big)\Big] = \int_{0}^{\infty} \exptn\Big[\overline{N}\big(\Delta\big) \,\big\vert\, \Delta=\delta\Big] \, \prob\big(\Delta\in \mathrm{d} \delta\big)\\& = \int_{0}^{\infty} \Lambda H\big(\underline{p}, 0\big) \, \delta \, \prob\big(\Delta\in \mathrm{d} \delta\big)
            \\
            &= \Lambda \,H\big(\underline{p}, 0\big) \, \exptn\big[\Delta\big] \leq \frac{\Lambda \,H\big(\underline{p}, 0\big)}{1-\gamma_1} \exptn\left\vert \overline{W}^{1}_0 - \overline{W}^{2}_0 \right\vert.
        \end{align*}

        \item[(d)] Suppose there are $\sigma_T-2$ regeneration cycles excluding the partial regeneration cycles at the start and end of the interval $(0,T]$. Let $C_1, \ldots, C_{\sigma_T}$ be the corresponding number of arrivals in each of these time intervals, with $C_2, \ldots, C_{\sigma_T}-1$ being i.i.d. Let $C$, $\tau$ be steady-state number of effective arrivals and steady-state regeneration cycle length respectively.
        Then,
        \begin{align*}
            \frac{N_1\left(0,T\right]}{T} = \frac{C_1 + \cdots + C_{\sigma_T}}{\sigma_T} \times \frac{\sigma_T}{T}\asarrow \frac{\exptn\left[C\right]}{\exptn[\tau]} = \frac{1}{\exptn[A_{\infty}(p)]}.
        \end{align*}
        As a consequence, recalling that $g(\cdot) = x^m + a_1 x^{m-1} + \cdots + a_m$ is a degree-$m$ polynomial that $T$ has support $[T^*,\infty)$,
        \begin{align*}
            \exptn\left[\frac{1}{g\left(N_1(0,T]\right)}\right] &= \mathcal{O}_{\prob}\left(\exptn\left[\frac{1}{N_1\left(0, T\right]^m}\right]\right) = \mathcal{O}_{\prob}\left( \exptn\left[\frac{1}{T^m}\frac{1}{\left(N_1\left(0, T\right]/T\right)^m}\right]\right)
            \\
            &\leq \mathcal{O}_{\prob}\left(\sqrt{\exptn\left[\frac{1}{T^{2m}}\right]\exptn\left[\frac{1}{\left(N_1\left(0,T\right]/T\right)^{2m}}\right]}\right) 
            \\
            &\leq \mathcal{O}_{\prob}\left(\sqrt{\frac{1}{{T^*}^{2m}}\exptn\left[\frac{1}{\left(N_1\left(0, T\right]/T\right)^{2m}}\right]}\right)
            = \mathcal{O}\left(\frac{1}{{T^*}^m}\right). 
        \end{align*}
    \end{enumerate}This completes the proofs of the four parts. 
\end{proof}

\subsection{Proof of Proposition \ref{propn:estimator_is_consistent}} \label{appendix:proof:propn:estimator_is_consistent}

Recall that
\begin{align*}
    \widehat{\nabla \Psi}(p) := \frac{1}{\widehat{A_{\infty}}(p)} - p\frac{\widehat{\nabla A_{\infty}}(p)}{\widehat{A_{\infty}}(p)^2}, \ \text{and} \ \nabla \Psi(p) = \frac{1}{\exptn[A_{\infty}(p)]} - p\frac{\exptn[\nabla A_{\infty}(p)]}{\exptn[A_{\infty}(p)]^2}.  
\end{align*}
Our strategy is to first prove the statements
\begin{enumerate}
    \item[(a)] $\widehat{A_{\infty}}(p)\asarrow \exptn[A_{\infty}(p)]$,
    \item[(b)] $\widehat{\nabla A_{\infty}}(p)\asarrow \exptn[\nabla A_{\infty}(p)]$,
\end{enumerate}
so that, by the continuous mapping theorem, it follows that $\widehat{\nabla \Psi}(p)\asarrow \nabla\Psi(p)$.
We use the setup as before: we start at time $0$, and enforce an admission price $p$ until time $T$. Customers are indexed from 1 to $N$, where $N$ is a random variable, such that $\widetilde{A}_N = T$.

\medskip

\textit{Proof of (a):}
It is clear that $\big(A_i, \overline{W}_i\big)$ is an ergodic sequence. Therefore, by the ergodic theorem,
\begin{align*}
    \widehat{A_{\infty}}(p) = \frac{A_1 + \cdots + A_{N}}{N}\asarrow \exptn[A_{\infty}(p)].
\end{align*}
as $N$ goes to $\infty$, which happens as $T$ goes to $\infty$.

\medskip

\textit{Proof of (b):}
Suppose the sample path of the queue in our time period of observation $[0,T]$ can be divided into $M$ regeneration cycles, allowing the first and last cycle being partial cycles with $c_1, \ldots, c_M$ customers in each of those cycles, where $c_1 + \ldots + c_M = N$. 
Let \[\big\{\zeta_{1,1}, \ldots, \zeta_{1, c_1}, \zeta_{2,1}, \ldots, \zeta_{2, c_2}, \ldots, \zeta_{M,1}, \ldots, \zeta_{M,c_M} \big\}\] be seeds for the interarrival times (being independent samples from the \text{Uniform}[0,\,1] distribution), where
\begin{itemize}
    \item[$\circ$] $\zeta_{1,1}$ = $\zeta_1$, generating the arrival time of the first customer;
    \item[$\circ$]for $j \in \{1, \ldots, M\}$ and $i \in \{2, \ldots, c_j\}$, $\zeta_{j,i}$ is the seed which generates the interarrival time between the $(i-1)$-th and $i$-th customers in the $j$-th regeneration cycle
    \item[$\circ$] for $j \in \{2, \ldots, M\}$, $\zeta_{j,i}$ is the seed which generates the interarrival times between customer $c_{j-1}$ of regeneration cycle $j-1$, and the first customer of cycle $j$;
\end{itemize}
observe that this set of random quantities is the same as the set $\{\zeta_1, \zeta_2, \cdots, \zeta_{N}\}$ but enumerated differently. Then,
\begin{align*}
   \widehat{\nabla A_{\infty}}(p) =  \frac{1}{N-1} \times \Big(&\nabla F^{-1}_p\big(\zeta_{1,2};\overline{W}_{1,1}\big) + \cdots + \nabla F^{-1}_p\big(\zeta_{1,c_1};\overline{W}_{1,c_1-1}\big) + 
    \nabla F^{-1}_p\big(\zeta_{2,1}; \overline{W}_{1,c_1}\big)
    \\
    &+ \nabla F^{-1}_p\big(\zeta_{2,2};\overline{W}_{2,1}\big) + \cdots + \nabla F^{-1}_p\big(\zeta_{2,c_2};\overline{W}_{2,c_2-1}\big) + \nabla F^{-1}_p\big(\zeta_{3,1}; \overline{W}_{2,c_2}\big)
    +
    \\
    &+ \ \cdots
    \\
    &+ \nabla F^{-1}_p\big(\zeta_{M,2}; \overline{W}_{M,1}\big) + \cdots + \nabla F^{-1}_p\big(\zeta_{M,c_M}; \overline{W}_{M,c_M-1} \big) \Big).
\end{align*}

The next step is to define a collection of random vectors $\big\{\bs{Z}_1, \ldots, \bs{Z}_M\big\}$ as follows. Each vector $\bs{Z}_i$ consists of a set of independent \text{Uniform}[0,\,1] random variables along with the relevant queueing data from regeneration cycle $i$. Consequently, (i) the vectors $\big\{\bs{Z}_1, \ldots, \bs{Z}_M\big\}$ are mutually independent, and (ii) with the exception of $\bs{Z}_1$ and $\bs{Z}_M$, they are identically distributed. These vectors are defined through
\begin{align*}
    \bs{Z}_1 &= \big(\zeta_{1,2}, \ldots, \zeta_{1,c_1}, \overline{W}_{1,1}, \ldots, \overline{W}_{1,c_1} \big),
    \\
    \bs{Z}_2 &= \big(\zeta_{2,1}, \zeta_{2,2}, \ldots, \zeta_{2,c_2}, \overline{W}_{2,1}, \ldots, \overline{W}_{2,c_2} \big),
    \\
    &\vdots
    \\
    \bs{Z}_{M-1} &= \big(\zeta_{M-1, 1}, \zeta_{M-1,2}, \ldots, \zeta_{M-1,c_{M-1}}, \overline{W}_{M-1,1}, \ldots, \overline{W}_{M-1,c_{M-1}}\big),
    \\
    \bs{Z}_M &= \big(\zeta_{M, 1}, \zeta_{M,2}, \ldots, \zeta_{M,c_M}, \overline{W}_{M,1}, \ldots, \overline{W}_{M,c_M-1}\big).
\end{align*}
We use Equation (A.4) to split $\nabla F^{-1}_p\big(\zeta_{j,1}; \overline{W}_{j-1, c_{j-1}} \big)$ into a part depending on regeneration cycle $j-1$, and a part depending on regeneration cycle $j$.
Define a collection of functions $\big\{q^p\big(\bs{Z}_1\big), \cdots, q^p\big(\bs{Z}_M\big) \big\}$ via
\begin{align*}
    q^p\big(\bs{Z}_1\big) &= \nabla F^{-1}_p\big(\zeta_{1,2};\overline{W}_{1,1}\big) + \cdots + \nabla F^{-1}_p\big(\zeta_{1,c_1};\overline{W}_{1,c_1-1}\big) + 
   \Big(1-e^{-\theta_2 \overline{W}_{1,c_1}}\Big) \nabla \overline{W}_{1,c_1},
   \\
   q^p\big(\bs{Z}_2\big) &= -\frac{\theta_1}{\Lambda} e^{\theta_1 p} \log(1-\zeta_{2,1}) + \nabla F^{-1}_p\big(\zeta_{2,2};\overline{W}_{2,1}\big) + \cdots + \nabla F^{-1}_p\big(\zeta_{2,c_2};\overline{W}_{2,c_2-1}\big)
   \\
   & \ \ \ +\Big(1-e^{-\theta_2 \overline{W}_{2,c_2}}\Big) \nabla \overline{W}_{2,c_2},
   \\
   &\vdots
   \\
   q^p\big(\bs{Z}_M\big) &= -\frac{\theta_1}{\Lambda} e^{\theta_1 p} \log(1-\zeta_{M,1}) + \nabla F^{-1}_p\big(\zeta_{M,2};\overline{W}_{M,1}\big) + \cdots + \nabla F^{-1}_p\big(\zeta_{M,c_M};\overline{W}_{M,c_M-1}\big).
\end{align*}
It is clear that all vectors are independent, and that all but $q^p\big(\bs{Z}_1\big)$ and $q^p\big(\bs{Z}_M\big)$ are identically distributed. By the strong law of large numbers and the continuous mapping theorem, we therefore conclude that
\begin{align*}
    \widehat{\nabla A_{\infty}}(p) = \frac{1}{N-1} \sum_{j=1}^{M} q^{p}\big(\bs{Z}_j\big) = \bigg(\frac{M}{N-1}\bigg) \Bigg(\frac{\sum_{j=1}^{M} q^p\big(\bs{Z}_j\big)}{M} \Bigg)\asarrow \frac{\exptn\Big[q^p\big(\bs{Z}\big)\Big]}{\exptn\big[C \vert p\big]} = \nabla \exptn\big[A_{\infty}(p)].
\end{align*}

\subsection{Proofs for Section~\ref{subsection:bias analysis} - Bias Analysis} \label{appendix:proof:bias}

Recall that the upper bound on the bias could be decomposed into the following terms,
\begin{align*}
    \omega_k^{\rm (I)}&:={\Bigg\vert \exptn\left[\frac{1}{\widehat{A_{\infty}}(p_{k-1})} \ \Big\vert \mathcal{F}_{k-1} \right] - \frac{1}{\exptn\big[A_{\infty}(p_{k-1})\big]} \Bigg\vert},\\ \omega_k^{\rm (II)}&:=  {\Bigg\vert \exptn\left[\frac{\widehat{\nabla A_{\infty}}(p_{k-1})}{\widehat{A_{\infty}}(p_{k-1})^2} \ \Big\vert \mathcal{F}_{k-1} \right] - \frac{\nabla \exptn\big[A_{\infty}(p_{k-1})\big]}{\exptn\big[A_{\infty}(p_{k-1})\big]^2}\Bigg\vert}.
\end{align*}

Before establishing the convergence rates of these terms, as stated in Lemmas~\ref{L1} and~\ref{L2} we define
\begin{align*}
    \omega_k^{\rm (III)}:=\left\vert \exptn\left[\widehat{A_{\infty}}(p_{k-1}) \ \big\vert \mathcal{F}_{k-1} \right] - \exptn\big[A_{\infty}(p_{k-1})\big] \right\vert,
\end{align*}
which we analyze fist as an intermediate step.

\begin{lemma} \label{lemma:omega3} As $k\to\infty$,
    \begin{align} \label{eqn:term III:final_upper_bound}
    \omega_k^{\rm (III)}= \mathcal{O}\left(\eta_{k-1} \frac{1}{{T^*_{k-1}}^2}\right).
\end{align}
\end{lemma}
\begin{proof}
As a first step, observe that $\omega_1^{\rm (III)}$ can be majorized via
\begin{align}\label{eqn:term_III:temp_1}
   \omega_k^{\rm (III)} =\ &\left| \mathbb{E}\left[ \frac{1}{N_k} \sum_{j=\widetilde{N}_{k-1}+1}^{\widetilde{N}_k} A_j \,\bigg|\, \mathcal{F}_{k-1} \right] 
    - \mathbb{E}\left[A_{\infty}(p_{k-1})\right] \right| \notag \\
    =\ &\left| \mathbb{E}\left[ \frac{1}{N_k} \sum_{j=\widetilde{N}_{k-1}+1}^{\widetilde{N}_k} 
    \left( A_j - A_{\infty}(p_{k-1}) \right) \,\bigg|\, \mathcal{F}_{k-1} \right] \right| \notag \\
    =\ &\left| \mathbb{E}\left[ \mathbb{E}\left[ \frac{1}{N_k} \sum_{j=\widetilde{N}_{k-1}+1}^{\widetilde{N}_k} 
    \left( A_j - A_{\infty}(p_{k-1}) \right) \,\bigg|\, N_k, \mathcal{F}_{k-1} \right] 
    \,\bigg|\, \mathcal{F}_{k-1} \right] \right| \notag \\
    \leq\ &\mathbb{E}\left[ \frac{1}{N_k} \left| \sum_{j=\widetilde{N}_{k-1}+1}^{\widetilde{N}_k} 
    \mathbb{E}\left[ A_j - A_{\infty}(p_{k-1}) \,\big|\, N_k, \mathcal{F}_{k-1} \right] \right| 
    \,\bigg|\, \mathcal{F}_{k-1} \right] \notag \\
    \leq\ &\mathbb{E}\left[ \frac{1}{N_k} \sum_{j=\widetilde{N}_{k-1}+1}^{\widetilde{N}_k} 
    \mathbb{E}\left[ \left| A_j - A_{\infty}(p_{k-1}) \right| \,\big|\, N_k, \mathcal{F}_{k-1} \right] 
    \,\bigg|\, \mathcal{F}_{k-1} \right],
\end{align}
where in the first inequality we have used Jensen's inequality, and in the second inequality we have used the triangle inequality. 
In order to further dominate the expectation in \eqref{eqn:term_III:temp_1}, we construct a probability space $\chi$, and on this space we couple two M/G/1\,+\,$H(p, V)$ queues with the same admission price $p_{k-1}$, interarrival-time seeds, and service times but with different initial workloads, as follows: 
\begin{enumerate}
    \item[$\circ$] In the first queue we have an initial workload $W\big(\overline{T}_{k-1}\big)$ which is the workload in the queue at time $\overline{T}_{k-1}$, being the result of following the learning algorithm that we defined.
    \item[$\circ$]In the second queue we have an initial workload $\overline{W}_{\infty}(p_{k-1})$, i.e., the stationary workload in the queue with price $p_{k-1}$, sampled independently from the variables of the learning algorithm.
\end{enumerate}
We now have a setup where Proposition \ref{propn:model:theoretical_results} can be used. Specifically, 
\begin{enumerate}
    \item[$\circ$] In the first queue, $\overline{W}^{1}_0 \sim W\big(\overline{T}_{k-1}\big)$. Therefore, for $n \geq 1$, $\overline{W}^{1}_n \sim \overline{W}_{\widetilde{N}_{k-1}+n}$, and $A^{1}_n \sim A_{\widetilde{N}_{k-1}+n}$.
    \item[$\circ$] In the second queue, we start in stationarity, i.e., $\overline{W}^{2}_0 \sim \overline{W}_{\infty}(p_{k-1})$. It thus follows that for all $n \geq 1$, $\overline{W}^{2}_{n} \sim \overline{W}_{\infty}(p_{k-1})$ and $A^{2}_n \sim A_{\infty}(p_{k-1})$.
\end{enumerate}
Then, for any $j \in \big\{ \widetilde{N}_{k-1}+1, \cdots, \widetilde{N}_k\big\}$, by Proposition \ref{propn:model:theoretical_results}(c),
\begin{align*}
    \exptn\big[ \big\vert A_j - A_{\infty}(p_{k-1}) \big\vert \ \big\vert N_k, \mathcal{F}_{k-1} \big] \leq \gamma_1^{j-\widetilde{N}_{k-1}} \exptn\left[\left\vert W\left(\overline{T}_{k-1}\right) - \overline{W}_{\infty}\left(p_{k-1}\right) \right\vert \ \big\vert \,\mathcal{F}_{k-1}\right].
\end{align*}
Substituting this into inequality \eqref{eqn:term_III:temp_1},  recalling that $\gamma_1<1$, 
\begin{align*}
   \omega_k^{\rm (III)}  &\leq \exptn\left[\frac{1}{N_k} \exptn\left[\left\vert W\left(\overline{T}_{k-1}\right) - \overline{W}_{\infty}\left(p_{k-1}\right) \right\vert \ \big\vert \mathcal{F}_{k-1}\right] \sum_{j=1}^{N_k} \gamma_1^j \ \Big\vert \mathcal{F}_{k-1}\right]
    \\
    &<_{\text{a.s.}} \exptn\left[\left\vert W\left(\overline{T}_{k-1}\right) - \overline{W}_{\infty}(p_{k-1}) \right\vert \ \big\vert \mathcal{F}_{k-1} \right]\\
    &\le \ \exptn\big[\left\vert W\left(\overline{T}_{k-1}\right) - \overline{W}_{\infty}(p_{k-2}) \right\vert \ \big\vert \mathcal{F}_{k-1} \big] + \exptn\big[\left\vert \overline{W}_{\infty}(p_{k-2}) - \overline{W}_{\infty}(p_{k-1}) \right\vert \ \big\vert \mathcal{F}_{k-1} \big] \notag,
\end{align*}
where in the last step we once more used the triangle inequality. 

Now appealing to Proposition \ref{propn:model:theoretical_results}(b),
\begin{align*}
    \exptn\left[\left\vert W\left(\overline{T}_{k-1}\right) - \overline{W}_{\infty}(p_{k-2}) \right\vert \ \big\vert \mathcal{F}_{k-1} \right] \leq \gamma_1^{N_{k-1}} \exptn\left[\left\vert W\left(\overline{T}_{k-2}\right) - \overline{W}_{\infty}(p_{k-2}) \right\vert \ \big\vert \mathcal{F}_{k-1} \right],
\end{align*}
so that we arrive at the upper bound
\begin{align}\label{eqn:term III:temp_2}
    \omega_k^{\rm (III)} \leq \ \delta_{k-1} \leq \gamma_1^{N_{k-1}} \delta_{k-2} + \rho_{k-2},
\end{align}
with, for any $i\in{\mathbb N}$, for ease suppressing the dependence on $k$,
\begin{align}
    \delta_i &:= \exptn\left[\left\vert W\left(\overline{T}_{i}\right) - \overline{W}_{\infty}(p_{i}) \right\vert \ \big\vert \mathcal{F}_{k-1} \right], \quad 
    \rho_i := \exptn\left[\big\vert \overline{W}_{\infty}(p_i) - \overline{W}_{\infty}(p_{i+1})\big\vert \ \big\vert \mathcal{F}_{k-1}\right]. \label{eqn:rho:defn}  
\end{align}
By iteration we obtain that, the empty product being defined as $1$, 
\begin{align} \label{eqn:delta:stochastic upper bound}
    \delta_{k-1} &\leq \gamma_1^{N_1 + \cdots + N_{k-1}} \delta_0 + 
    \sum_{\ell=0}^{k-2} \rho_{\ell}\left(\prod_{m=\ell+2}^{k-1} N_m\right)
    = \mathcal{O}_{\prob} \left(\rho_{k-2}\right);
\end{align} 
here it has been used (i)~that, as $k \rightarrow \infty$, $N_{k-1} \rightarrow \infty$, (ii)~that, by Proposition \ref{propn:model:theoretical_results}(a), $\gamma_1 < 1$ (also noting that the $\rho_i$ are bounded quantities).
Applying Taylor’s theorem, and invoking Corollary~\ref{corr:bounds_on_gradients}, 
\begin{align*} 
    \rho_{k-2} &= \exptn\left[\left\vert W_{\infty}(p_{k-2}) - W_{\infty}(p_{k-1}) \right\vert \ \big\vert \mathcal{F}_{k-1} \right] \notag
    \\
    &\leq G_{\max} \exptn\left[\left\vert p_{k-2} - p_{k-1}\right\vert \big\vert \mathcal{F}_{k-1}\right] \leq G_{\max} \exptn\left[\left\vert \eta_{k-1} \widehat{\nabla\Psi}(p_{k-2})\right\vert \ \big\vert \mathcal{F}_{k-1}\right] \notag
    \\
    &= G_{\max} \eta_{k-1} \exptn\left[\left\vert \frac{1}{\widehat{A_{\infty}}(p_{k-2})} - p_m\frac{\widehat{\nabla A_{\infty}}(p_{k-2})} {\widehat{A_{\infty}}(p_{k-2})^2} \right\vert \ \Big\vert \mathcal{F}_{k-1} \right]\notag  \\&\overset{(*)}{=} \mathcal{O}_{\prob} \left(\eta_{k-1} \exptn\left[\frac{1}{\Gamma^2} \ \Big\vert \mathcal{F}_{k-1}\right]\right) = \mathcal{O}_{\prob} \left(\eta_{k-1} \frac{1}{(N_{k-1}-2)(N_{k-1}-3)}\right). \notag
\end{align*}
where the equality marked $(*)$ has been established, for
$\left[\,\Gamma\, \vert \,\mathcal{F}_{k-1}\,\right] \sim \text{Gamma}\big(N_{k-1}-1, \Lambda e^{-\theta_1 \underline{p}}\big)$,  in Section \ref{subsection:variability analysis}.
Then by Proposition \ref{propn:model:theoretical_results}(d),
\begin{align}\label{eqn:rho:stochastic upper bound}
    \rho_{k-2} \le \mathcal{O}\left(\eta_{k-1} \frac{1}{{T^*_{k-1}}^2}\right).
\end{align}
The proof is concluded by combining \eqref{eqn:term III:temp_2}, \eqref{eqn:delta:stochastic upper bound} and \eqref{eqn:rho:stochastic upper bound}.
\end{proof}

\begin{proof}[Proof of Lemma~\ref{L1}]
Let $\epsilon$ be such that $0 < \epsilon <  \exptn\big[A_{\infty}(\underline{p})\big] < \infty$. Then, by Jensen's inequality,
\begin{align*}
    &\hspace{-0.2cm}\omega_k^{\rm (I)}\leq \exptn\left[ \left\vert \frac{1}{\widehat{A_{\infty}}(p_{k-1})} - \frac{1}{\exptn[A_{\infty}(p_{k-1})]} \right\vert \Big\vert \mathcal{F}_{k-1} \right] \leq \omega_k^{\rm (I.a)}+\omega_k^{\rm (I.b)},
\end{align*}
 where, with ${\boldsymbol 1}_k= \mathbbm{1}\big\{\widehat{A_{\infty}}(p_{k-1}) \geq \epsilon\big\}$,
   \begin{align*}\omega_k^{\rm (I.a)}
   &:={\exptn\left[ \left\vert \frac{1}{\widehat{A_{\infty}}(p_{k-1})} - \frac{1}{\exptn[A_{\infty}(p_{k-1})]} \right\vert  {\boldsymbol 1}_k\,\Big\vert \mathcal{F}_{k-1} \right]},\\\omega_k^{\rm (I.b)}&:={\exptn\left[ \left\vert \frac{1}{\widehat{A_{\infty}}(p_{k-1})} - \frac{1}{\exptn[A_{\infty}(p_{k-1})]} \right\vert (1-{\bs 1}_k) \, \Big\vert \mathcal{F}_{k-1} \right]},
\end{align*}

We now proceed to analyze both terms separately.

\smallskip 

We start by exploring the decay of $\omega_k^{\rm (I.a)}$. By Taylor's theorem, there exists a random variable $\xi$ between $\widehat{A_{\infty}}(p_{k-1})$ 
and $\exptn\left[A_{\infty}(p_{k-1})\right]$ 
such that
\begin{align*}
    \left\vert\frac{1}{\widehat{A_{\infty}}(p_{k-1})} - \frac{1}{\exptn[A_{\infty}(p_{k-1})]} \right\vert=  \frac{1}{\xi^2} \left\vert\widehat{A_{\infty}}(p_{k-1}) - \exptn[A_{\infty}(p_{k-1})]\right\vert.
\end{align*}
Then,
\begin{align*}
    \omega_k^{\rm (I.a)}&= \exptn\left[ \frac{1}{\xi^2} \left\vert \widehat{A_{\infty}}(p_{k-1}) - \exptn[A_{\infty}(p_{k-1})]\right\vert\, {\bs 1}_k \Big\vert \mathcal{F}_{k-1} \right]
    \\
    & \leq_{\text{a.s.}} \frac{1}{\epsilon^2} \exptn\left[\left\vert \widehat{A_{\infty}}(p_{k-1}) - \exptn[A_{\infty}(p_{k-1})]\right\vert {\bs 1}_k \ \big\vert \mathcal{F}_{k-1}\right]
    \\
    & \leq_{\text{a.s.}} \frac{1}{\epsilon^2} \exptn\Big[\left\vert \widehat{A_{\infty}}(p_{k-1}) - \exptn[A_{\infty}(p_{k-1})]\right\vert {\bs 1}_k 
    \\
    &\hspace{2cm} + \left\vert \widehat{A_{\infty}}(p_{k-1}) - \exptn[A_{\infty}(p_{k-1})]\right\vert (1 - {\bs 1}_k) \ \big\vert \mathcal{F}_{k-1}\Big]
    \\
    &= \frac{1}{\epsilon^2} \exptn\left[\Big\vert \widehat{A_{\infty}}(p_{k-1}) - \exptn[A_{\infty}(p_{k-1})] \Big\vert \ \big\vert \mathcal{F}_{k-1} \right]
    = \mathcal{O}\left(\eta_{k-1} \frac{1}{{T^*_{k-1}}^2}\right),
\end{align*}
where we have used $\xi > \epsilon$ in the first inequality, and Lemma \ref{lemma:omega3} in the final step.

\smallskip

We continue by showing that $\omega_k^{\rm (I.b)}$, as $k \rightarrow \infty$, decays faster than $\omega_k^{\rm (I.a)}$, so that $\omega_k^{\rm (I.b)}$ becomes negligible. By `splitting the event of interest into slices', we can bound the expectations corresponding to the resulting subevents. As a first step, with 
\begin{align*}
    {\boldsymbol i}_{k,j}:= \mathbbm{1}\big\{\epsilon/2^{j+1}<\widehat{A_{\infty}}(p_{k-1}) \leq \epsilon/2^j\big\},    
\end{align*}
we write
\begin{align*}
   \omega_k^{\rm (I.b)} = \sum_{j=0}^{\infty} \ \exptn\Bigg[ \left\vert \frac{1}{\widehat{A_{\infty}}(p_{k-1})} - \frac{1}{\exptn[A_{\infty}(p_{k-1})]} \right\vert \,{\boldsymbol i}_{k,j}\, \Big\vert \mathcal{F}_{k-1} \Bigg].
\end{align*}
Now note that in case ${\bs i}_{k,j}$ equals $1$ for a given $j \geq 0$, we have, recalling that $0 < \epsilon <  \exptn\big[A_{\infty}(\underline{p})\big]$,
\begin{align*}
    \frac{2^{j+1}}{\epsilon} > \frac{1}{\widehat{A_{\infty}}(p_{k-1})} - \frac{1}{\exptn[A_{\infty}(p_{k-1})]} \geq \frac{1}{\widehat{A_{\infty}}(p_{k-1})} - \frac{1}{\exptn\left[A_{\infty}(\underline{p})\right]} > \frac{2^j}{\epsilon} - \frac{1}{\epsilon} \geq 0,
\end{align*}
thus leading to the inequality
\begin{align*}
    \exptn\Bigg[ \left\vert \frac{1}{\widehat{A_{\infty}}(p_{k-1})} - \frac{1}{\exptn[A_{\infty}(p_{k-1})]} \right\vert \,{\boldsymbol i}_{k,j} \Big\vert \mathcal{F}_{k-1} \Bigg]&< \frac{2^{j+1}}{\epsilon} \,\prob\{{\boldsymbol i}_{k,j}\,\vert\, \mathcal{F}_{k-1}\},
\end{align*}
so that
\begin{align*}
    \omega_k^{\rm (I.b)}\leq \sum_{j=0}^{\infty} \ \frac{2^{j+1}}{\epsilon} \ \prob\{{\boldsymbol i}_{k,j}\,\vert\, \mathcal{F}_{k-1}\}.
\end{align*}
We continue by further bounding $\prob\{{\boldsymbol i}_{k,j}\,\vert\, \mathcal{F}_{k-1}\}$. 
\begin{align*}
    \prob\{{\boldsymbol i}_{k,j}\,\vert\, \mathcal{F}_{k-1}\} &\leq   \ \prob \left\{\frac{\epsilon}{2^{j+1}} < \widehat{A_{\infty}}(p_{k-1}) \leq \frac{\epsilon}{2^j} \ \Big\vert \mathcal{F}_{k-1}\right\}= \prob\left\{ \frac{\epsilon}{2^{j+1}} < \frac{T_k}{N_k} \leq \frac{\epsilon}{2^j} \ \Big\vert \mathcal{F}_{k-1} \right\}
    \\
    & \leq \prob\left\{ \frac{2^j}{\epsilon} T_k \leq N_k \ \Big\vert \mathcal{F}_{k-1} \right\}
    \\
    & \leq  \prob\left\{ \frac{2^j}{\epsilon} T^{*}_k \leq N\left(\overline{T}_{k-1}, \overline{T}_{k-1}+T^{*}_k\right] + 1 \ \Big\vert \mathcal{F}_{k-1} \right\}.
\end{align*}
Let $\overline{N}(\cdot)$ be the counting process of arrivals in a Poisson process with arrival rate $\Lambda^\circ  = \Lambda H\big(\underline{p}, 0\big)$. 
From an easy coupling argument, it follows that $\overline{N}\big(T^*_k\big) \equiv \overline{N}\big(\overline{T}_{k-1}, \overline{T}_{k-1}+T^*_k\big] \geq_{\text{st}} N_k$, so that
\begin{align*}
    \prob\{{\boldsymbol i}_{k,j}\,\vert\, \mathcal{F}_{k-1}\} \leq  \prob\left\{ \frac{2^j}{\epsilon} T^{*}_k-1 \leq \overline{N}(T^{*}_k)\right\}.
\end{align*}
Now pick $a_j = {2^j T^{*}_k}/{\epsilon}$, $\mu = \Lambda^\circ  T^*_k$, and $f_j = a_j/\mu$. Then, using the Chernoff bound for $\overline{N}\left(T^{*}_k\right)$, we find
\begin{align*}
    \prob\left\{\overline{N}\left(T^{*}_k\right) \geq a_j\right\} &\leq \exp\left(-\mu\left(f_j\log(f_j) - f_j + 1\right)\right)
    \\
    &= \exp\left(-\Lambda^\circ  T^{*}_k\right) \exp\left(-\frac{2^j T^*_k}{\epsilon} \big(j\log 2 - \log\left(\Lambda^\circ \epsilon\right) - 1\big)\right). 
\end{align*}
Now put $C_2 := \log 2/\epsilon$ and $C_3 := - \log\left(\Lambda^\circ \epsilon\right) - 1$. Then, for any $\epsilon < \Lambda^\circ /e$, $C_2, C_3 > 0$, we obtain
\begin{align*}
    \prob\left\{\overline{N}\left(T^{*}_k\right) \geq a_j\right\} \leq \exp\left(\Lambda^\circ  T^{*}_k\right) \exp\left(-C_2 T^{*}_k \ j 2^j - C_3 T^{*}_k \ 2^j\right) \leq \exp\left(\Lambda^\circ  T^{*}_k\right) \exp\left(-C_2 T^*_k \ j 2^j\right).
\end{align*}
Upon combining the above, we have arrived at the upper bound
\begin{align*}
    \omega_k^{\rm (I.b)}&\leq \exp\left(-\Lambda^\circ  T^{*}_k\right) \sum_{j=0}^{\infty} \frac{2^{j+1}}{\epsilon} \exp\left(-C_2 T^*_k \ j 2^j\right)
    = \exp\left(-\Lambda^\circ  T^{*}_k\right) \frac{2}{\epsilon} \sum_{j=0}^{\infty} \exp\left(j\log 2 -C_2 T^*_k \ j 2^j\right).
\end{align*}
We are left with verifying that the expression on the right-hand side is  negligible compared to $\omega_k^{\rm (I.a)}$ as $k\to\infty$.
Since $j \log 2 \ll j2^j$, there exists $C_2' < C_2$ such that $j \log 2 - C_2 T^{*}_k \ j 2^j < -C_2' T^{*}_k \ j 2^j$.
As a consequence, $\exp\left(j\log 2 -C_2 T^*_k \ j 2^j\right) < \exp\left(-C_2' T^{*}_k \ j2^j\right)$, which becomes exponentially small as $j \rightarrow \infty$. As we can find a constant $C_2''$ such that
\begin{align*}
    \sum_{j=0}^{\infty} \exp\left(-C_2' T^{*}_k \ j2^j\right) \leq C_2'',
\end{align*}
we thus  conclude that
$
   \omega_k^{\rm (I.b)}\leq \mathcal{O}\big(e^{-\Lambda^\circ  T^*_k}\big),$
   which is asymptotically negligible relative to the contribution of $\omega_k^{\rm (I.a)}$.
Combining our findings regarding the decay of both terms, we obtain an upper bound to the the decay rate of $\omega_k^{\rm (I)}$. 
\end{proof}

\begin{proof}[Proof of Lemma~\ref{L2}]
In the analysis of the decay of $\omega_k^{\rm (II)}$ we use a reasoning that is similar to the one we used for analyzing $\omega_k^{\rm (I)}$ in Lemma \ref{L1}, in that we split the quantity of interest into two contributions, as follows:
\begin{align*}
    \omega_k^{\rm (II)}
    & \leq \exptn\Bigg[\left\vert \frac{\widehat{\nabla A_{\infty}}(p_{k-1})}{\widehat{A_{\infty}}(p_{k-1})^2} - \frac{\nabla \exptn\big[A_{\infty}(p_{k-1})\big]}{\exptn\big[A_{\infty}(p_{k-1})\big]^2} \right\vert  \Big\vert \mathcal{F}_{k-1}\Bigg] =\omega_k^{\rm (II.a)}+\omega_k^{\rm (II.b)},
\end{align*}
where, with ${\bs 1}_k$ as defined in the proof of Lemma \ref{L1},
\begin{align*}
    &  \omega_k^{\rm (II.a)}:={\exptn\Bigg[\left\vert \frac{\widehat{\nabla A_{\infty}}(p_{k-1})}{\widehat{A_{\infty}}(p_{k-1})^2} - \frac{\nabla \exptn\big[A_{\infty}(p_{k-1})\big]}{\exptn\big[A_{\infty}(p_{k-1})\big]^2} \right\vert \,{\bs 1}_k\,\Big\vert \mathcal{F}_{k-1}\Bigg]},
    \\
    & \omega_k^{\rm (II.b)}:= {\exptn\Bigg[\left\vert \frac{\widehat{\nabla A_{\infty}}(p_{k-1})}{\widehat{A_{\infty}}(p_{k-1})^2} - \frac{\nabla \exptn\big[A_{\infty}(p_{k-1})\big]}{\exptn\big[A_{\infty}(p_{k-1})\big]^2} \right\vert (1-{\bs 1}_k)\Big\vert \mathcal{F}_{k-1}\Bigg]}.
\end{align*}

\smallskip

We start by analyzing the decay of $\omega_k^{\rm (II.a)}$. Let $\xi_1$ and $\xi_2$ lie between $\widehat{\nabla A_{\infty}}(p_{k-1})$ and $\nabla\exptn[A_{\infty}(p_{k-1})]$. By the multivariate version of Taylor's theorem, applied to the function $f(x,y) ={x}/{y^2}$ at the point $\left(\nabla\exptn[A_{\infty}(p_{k-1})], \exptn[A_{\infty}(p_{k-1})]\right)$, we obtain
\begin{align*}
    \frac{\widehat{\nabla A_{\infty}}(p_{k-1})}{\widehat{A_{\infty}}(p_{k-1})^2} - \frac{\nabla \exptn\big[A_{\infty}(p_{k-1})\big]}{\exptn\big[A_{\infty}(p_{k-1})\big]^2} &= \left(\widehat{\nabla A_{\infty}}(p_{k-1}) - \nabla\exptn[A_{\infty}(p_{k-1})]\right)\cdot \frac{1}{\xi_2^2} \:-\\&\quad \quad \left(\widehat{A_{\infty}}(p_{k-1}) - \exptn[A_{\infty}(p_{k-1})]\right) \cdot 2\,\frac{\xi_1}{\xi_2^3}
   \end{align*}
   so that
    \begin{align*}&\left\vert \frac{\widehat{\nabla A_{\infty}}(p_{k-1})}{\widehat{A_{\infty}}(p_{k-1})^2} - \frac{\nabla \exptn\big[A_{\infty}(p_{k-1})\big]}{\exptn\big[A_{\infty}(p_{k-1})\big]^2} \right\vert \leq \left\vert\widehat{\nabla A_{\infty}}(p_{k-1}) - \nabla\exptn[A_{\infty}(p_{k-1})]\right\vert \left(\frac{1}{\xi_2^2} + 2\,\frac{\xi_1}{\xi_2^3} \right).
\end{align*}
Recalling that $\xi_2 > \epsilon$, and $\xi_1 \leq G_{\max}$, we find
\begin{align*}    
    \omega_k^{\rm (II.a)} &  \leq_{\text{a.s.}} \exptn\Bigg[\left\vert\widehat{\nabla A_{\infty}}(p_{k-1}) - \nabla\exptn[A_{\infty}(p_{k-1})]\right\vert \left(\frac{1}{\xi_2^2} + 2\frac{\xi_1}{\xi_2^3} \right) \,{\bs 1}_k\, \Big\vert \mathcal{F}_{k-1}\Bigg]
    \\
    &  \leq_{\text{a.s.}} \left(\frac{1}{\epsilon^2} + G_{\max} \frac{2}{\epsilon^3}\right) \exptn\left[\left\vert\widehat{\nabla A_{\infty}}(p_{k-1}) - \nabla\exptn[A_{\infty}(p_{k-1})]\right\vert \,{\bs 1}_k\,\Big\vert \mathcal{F}_{k-1}\right]
    \\
    & \leq_{\text{a.s.}} \left(\frac{1}{\epsilon^2} + G_{\max} \frac{2}{\epsilon^3}\right) \exptn\Big[\left\vert \widehat{A_{\infty}}(p_{k-1}) - \exptn[A_{\infty}(p_{k-1})]\right\vert\,{\bs 1}_k
    \\
    &\hspace{3.2cm}+ \left\vert \widehat{A_{\infty}}(p_{k-1}) - \exptn[A_{\infty}(p_{k-1})]\right\vert (1-{\bs 1}_k)\,\big\vert \mathcal{F}_{k-1}\Big]
    \\
    &  = \left(\frac{1}{\epsilon^2} + G_{\max}\frac{2}{\epsilon^3}\right) \exptn\left[ \left\vert\widehat{\nabla A_{\infty}}(p_{k-1}) - \nabla\exptn[A_{\infty}(p_{k-1})]\right\vert \big\vert \mathcal{F}_{k-1}\right]  = \mathcal{O}\left(\eta_{k-1} \frac{1}{{T^*_{k-1}}^2}\right).
\end{align*}

\smallskip

We finally aim to show that $\omega_k^{\rm (II.b)}$ is asymptotically negligible relative to $\omega_k^{\rm (II.a)}$. To this end, observe, with ${\bs i}_{k,j}$ as defined in the proof of Lemma \ref{L1}, 
\begin{align*}
    &\omega_k^{\rm (II.b)}= \sum_{j=0}^{\infty} \exptn\Bigg[\left\vert \frac{\widehat{\nabla A_{\infty}}(p_{k-1})}{\widehat{A_{\infty}}(p_{k-1})^2} - \frac{\nabla \exptn\big[A_{\infty}(p_{k-1})\big]}{\exptn\big[A_{\infty}(p_{k-1})\big]^2} \right\vert{\bs i}_{k,j}\, \Big\vert \mathcal{F}_{k-1}\Bigg].
\end{align*}
On the event $\{ {\epsilon}/{2^{j+1}} < \widehat{A_{\infty}}(p_{k-1}) \leq {\epsilon}/{2^j} \}$, again applying Taylor's theorem, and with reference to Corollary \ref{corr:bounds_on_gradients}, $\left\vert \widehat{\nabla A_{\infty}}(p_{k-1})\right\vert \leq G_{\max}$ and $\nabla \exptn[A_{\infty}(p_{k-1})] < G_{\max}$, we have
\begin{align*}
    \left\vert \frac{\widehat{\nabla A_{\infty}}(p_{k-1})}{\widehat{A_{\infty}}(p_{k-1})^2} - \frac{\nabla \exptn\big[A_{\infty}(p_{k-1})\big]}{\exptn\big[A_{\infty}(p_{k-1})\big]^2} \right\vert &\leq \left\vert \frac{\widehat{\nabla A_{\infty}}(p_{k-1})}{\widehat{A_{\infty}}(p_{k-1})^2} \right\vert + \left\vert \frac{\nabla \exptn\big[A_{\infty}(p_{k-1})\big]}{\exptn\big[A_{\infty}(p_{k-1})\big]^2} \right\vert \leq G_{\max} \frac{2^{2j+3}}{\epsilon^2}.
\end{align*}
Therefore,
\begin{align*}
    \omega_k^{\rm (II.b)} \leq \sum_{j=0}^{\infty} G_{\max} \frac{2^{2j+3}}{\epsilon^2}\,\prob\{{\boldsymbol i}_{k,j}\,\vert\, \mathcal{F}_{k-1}\}.
\end{align*}
From this point on, we essentially follow the argumentation developed in the analysis of $\omega_k^{\rm (I.b)}$. In particular, we have that
\begin{align*}
     \sum_{j=0}^{\infty} G_{\max} \frac{2^{2j+3}}{\epsilon^2}\prob\left\{ \frac{2^j T^*_k}{\epsilon}-1 \leq \overline{N}(T^*_k) \right\}= \mathcal{O}\big(e^{-\Lambda^\circ  T^*_k}\big).
\end{align*}
Combining the results we found for $\omega_k^{\rm (II.a)}$ and $\omega_k^{\rm (II.b)}$, we have proven the stated.
\end{proof}

\subsection{Proof of Theorem \ref{thm:price_convergence_rate}} \label{appendix:proof:th:price_convergence_rate}

First, we show that there exists $\eta < \infty$ such that Assumption~\ref{assk}(a) is satisfied. Consider the function
\begin{align*}
    v(x; \eta,\alpha) := 1 + \frac{K_0\eta}{2} x^{-\alpha} - \left(1 + \frac{1}{x}\right)^{\alpha}.
\end{align*}
We prove that there exists $\eta < \infty$ such that $v(x; \eta,\alpha) > 0$ for all $x \geq 1$, $\alpha \in \big(\frac{1}{2}, 1\big]$, i.e.,
\begin{align*}
    \frac{2}{K_0} \sup_{\substack{x \geq 1 \\\alpha \in \big(\frac{1}{2}, 1\big]}} x^\alpha \left(\left(1 + \frac{1}{x}\right)^\alpha - 1\right) < \infty.
\end{align*}
Towards that,
\begin{align*}
    \frac{2}{K_0} x^\alpha \left(\left(1 + \frac{1}{x}\right)^\alpha - 1\right) &= \frac{2}{K_0} x^\alpha \left(\frac{\alpha}{x} + \frac{\alpha(\alpha-1)}{2! \ x^2} + \cdots\right)
    \\
    &\leq \frac{2}{K_0} x^{\alpha} \left(\frac{\alpha}{x} + \frac{\alpha^2}{2 \ x^2} + \frac{\alpha^3}{2^2 \ x^3} + \cdots \right)
    \\
    &\leq \frac{4}{K_0} x^{\alpha} \left(\frac{\alpha}{2x} + \left(\frac{\alpha}{2x}\right)^2 + \left(\frac{\alpha}{2x}\right)^3 + \cdots \right)
    \\
    &= \frac{4}{K_0} x^{\alpha} \frac{\alpha}{2x - \alpha} 
    \leq \frac{4}{K_0} x^{\alpha} \frac{\alpha}{x} \leq \frac{4}{K_0}.
\end{align*}
Clearly, we can choose $\eta \geq \frac{4}{K_0}$, and hypothesis (a) is satisfied. Now suppose $T^*_k = \tau g(k)$. 
Due to \eqref{eqn:bias:final_rate}, $\tau$ can be chosen sufficiently large so that hypothesis (b) is satisfied.
Note that \cite[Thm.\ 2]{ChenLiuHong2024} assumes condition (c) about the variability $\nu_k$, which due to {the fact that $\nu_k = \mathcal{O}(1)$ as $k \rightarrow \infty$ (see Section \ref{subsection:variability analysis})}
holds trivially in our case when choosing $K_3 = \eta C_1$.
We now reproduce the first part of the proof of \cite{ChenLiuHong2024} within the context of our model.
Define for $j \geq 1$, $e_{j} := \exptn\left[(p_{j-1}-p^*)^2\right]$.

Similar to \cite{ChenLiuHong2024}, we  derive a recursive relationship between $e_{k}$ and $e_{k+1}$:
\begin{align}\label{eqn:price_convergence:temp_1}
    \exptn\left[\left(p_{k}-p^*\right)^2\right] &\leq \exptn\left[\left(p_{k-1}+\eta_k \widehat{\nabla\Psi}(p_{k-1}) - p^*\right)^2 \right] \notag
    \\
    &= \exptn\left[\left(p_{k-1}-p^*\right)^2 + 2\eta_k\widehat{\nabla\Psi}(p_{k-1})\left(p_{k-1}-p^*\right) + \eta_k^2 \widehat{\nabla\Psi}(p_{k-1})^2 \right] \notag
    \\
    &= \exptn\left[\left(p_{k-1}-p^*\right)^2 + 2\eta_k \nabla\Psi(p_{k-1})(p_{k-1}-p^*)\right] \notag
    \\
    &\hspace{0.5cm} + \exptn\left[2\eta_k\left(\widehat{\nabla\Psi}(p_{k-1})-\nabla\Psi(p_{k-1})\right)(p_{k-1}-p^*)\right] + \exptn\left[\eta_k^2 \ \widehat{\nabla\Psi}(p_{k-1})^2\right].
\end{align}
We begin by analyzing the first term on the right-hand side of Equation~\eqref{eqn:price_convergence:temp_1}. By Assumption \ref{assumption:Psi},
\begin{align*}
    \nabla\Psi(p_{k-1})(p_{k-1}-p^*) \leq -K_0(p_{k-1}-p^*)^2.
\end{align*}
Therefore,
\begin{align}\label{eqn:price_convergence:temp_1:term_1}
    \exptn\left[\left(p_{k-1}-p^*\right)^2 + 2\eta_k \nabla\Psi(p_{k-1})(p_{k-1}-p^*)\right] \leq \left(1 - 2\eta_k K_0\right) \exptn\left[(p_{k-1}-p^*)^2\right].
\end{align}
We then analyze the second term on the right-hand side of Equation \eqref{eqn:price_convergence:temp_1}:
\begin{align}\label{eqn:price_convergence:temp_1:term_2}
    &\exptn\left[2\eta_k\left(\widehat{\nabla\Psi}(p_{k-1}) - \nabla\Psi(p_{k-1})\right)(p_{k-1}-p^*)\right] \notag
    \\
    &\hspace{2cm} = \exptn\left[\exptn\left[2\eta_k\left(\widehat{\nabla\Psi}(p_{k-1}) - \nabla\Psi(p_{k-1})\right)(p_{k-1}-p^*) \ \big\vert \mathcal{F}_{k-1}\right]\right] \notag
    \\
    &\hspace{2cm} = 2\eta_k\exptn\left[\exptn\left[\widehat{\nabla\Psi}(p_{k-1}) - \nabla\Psi(p_{k-1}) \big\vert \mathcal{F}_{k-1}\right](p_{k-1}-p^*)\right] \notag
    \\
    &\hspace{2cm} \leq 2\eta_k \exptn\left[\exptn\left[\widehat{\nabla\Psi}(p_{k-1}) - \nabla\Psi(p_{k-1}) \big\vert \mathcal{F}_{k-1}\right]^2\right]^{\frac{1}{2}} \exptn\left[(p_{k-1}-p^*)^2\right]^{\frac{1}{2}} \notag
    \\
    &\hspace{2cm} \leq \eta_k \exptn\left[\exptn\left[\widehat{\nabla\Psi}(p_{k-1}) - \nabla\Psi(p_{k-1}) \big\vert \mathcal{F}_{k-1}\right]^2\right]^{\frac{1}{2}} \left(1 + \exptn\left[(p_{k-1}-p^*)^2\right]\right),
\end{align}
where the second last inequality follows by the Cauchy-Schwartz inequality while the last inequality holds because $2a \leq 1+a^2$.
Substituting \eqref{eqn:price_convergence:temp_1:term_1} and \eqref{eqn:price_convergence:temp_1:term_2} into \eqref{eqn:price_convergence:temp_1} yields
\begin{align*}
    \exptn\left[\left(p_{k}-p^*\right)^2\right] \leq &\left(1-2\eta_k K_0\right)\exptn\left[(p_{k-1}-p^*)^2\right]  + \eta_k^2 \exptn\left[\widehat{\nabla\Psi}(p_{k-1})^2\right] 
    \\
    &\hspace{0.5 cm} + \eta_k \exptn\left[\exptn\left[\widehat{\nabla\Psi}(p_{k-1}) - \nabla\Psi(p_{k-1}) \big\vert \mathcal{F}_{k-1}\right]^2\right]^{\frac{1}{2}} \left(1 + \exptn\left[(p_{k-1}-p^*)^2\right]\right) 
    \\
    = &\left(1 - 2\eta_k K_0 + \eta_k \exptn\left[\exptn\left[\widehat{\nabla\Psi}(p_{k-1}) - \nabla\Psi(p_{k-1}) \big\vert \mathcal{F}_{k-1}\right]^2\right]^{\frac{1}{2}}\right)\exptn\left[(p_{k-1}-p^*)^2\right]
    \\
    &\hspace{0.5 cm} + \eta_k \exptn\left[\exptn\left[\widehat{\nabla\Psi}(p_{k-1}) - \nabla\Psi(p_{k-1}) \big\vert \mathcal{F}_{k-1}\right]^2\right]^{\frac{1}{2}} + \eta_k^2 \exptn\left[\widehat{\nabla\Psi}(p_{k-1})^2\right],
\end{align*}
i.e.,
\begin{align*}
    e_{k+1} \leq \left(1 - 2\eta_k K_0 + \eta_k B_k \right)e_k + \eta_k B_k + \eta_k^2 \nu_k.
\end{align*}
Once we have established this recursion, the second part of the proof relies on assumptions (a) and (b) of Theorem \ref{thm:price_convergence_rate}. 
We refer the reader to the rest of the proof in \cite{ChenLiuHong2024}.

\section{Formulae for Example~\ref{example: poly_H}} \label{appendix:formulae:alternative_joining_probability}

In this appendix, we provide expressions analogous to Equations \eqref{eqn:F_p(l;w)}, \eqref{eqn:inverse:F_p(zeta;w)} and \eqref{eqn:price der_inverse_F_p(zeta;w)}, but now for the joining probability function
$H(p,V) = \bigl(1 + \theta_1 p^2 + \theta_2 V^2\bigr)^{-1}$:
\begin{align*} 
    F_p(\ell;w) &= 
    \begin{cases}
        - \exp\biggl(-\frac{\Lambda}{\sqrt{\theta_2} \sqrt{1 + \theta_1 p^2}} \biggl\{ \arctan\biggl(\frac{\sqrt{\theta_2}w}{\sqrt{1+\theta_1 p^2}}\biggr) - \arctan\biggl(\frac{\sqrt{\theta_2}(w-\ell)}{\sqrt{1+\theta_1 p^2}}\biggr) \biggr\}  \biggr), 
        & \ell < w, \\[1.2ex]
        1 - \exp\biggl(-\Lambda\biggl\{\frac{1}{\sqrt{\theta_2} \sqrt{1 + \theta_1 p^2}} \arctan\biggl(\frac{\sqrt{\theta_2}w}{\sqrt{1+\theta_1 p^2}}\biggr) + \frac{\ell-w}{1+\theta_1 p^2} \biggr\}  \biggr), 
        & \ell \geq w ,
    \end{cases}\\
    F^{-1}_p(\zeta;w) &=
    \begin{cases}
        w - \frac{1+\theta_1 p^2}{\theta_2} \tan\Biggl(\arctan\biggl(\frac{\sqrt{\theta_2}w}{\sqrt{1 + \theta_1 p^2}} + \frac{\sqrt{\theta_2}\sqrt{1+\theta_1 p^2}}{\Lambda} \log(1-\zeta) \biggr) \Biggr), 
        & \zeta < F_p(w;w), \\[1.2ex]
        w - \frac{1 + \theta_1 p^2}{\Lambda} \log(1-\zeta) - \frac{\sqrt{1+\theta_1p^2}}{\sqrt{\theta_2}} \arctan\biggl(\frac{\sqrt{\theta_2}w}{\sqrt{1+\theta_1 p^2}}\biggr), 
        & \zeta \geq F_p(w;w) ,
    \end{cases}\\
    \nabla_{p} F^{-1}_p(\zeta;w) &=
    \begin{cases}
        \begin{alignedat}{2}
        &\nabla_p w 
        - \frac{\theta_1 p}{\sqrt{\theta_2}\sqrt{1+\theta_1 p^2}}\tan(\vartheta) - \sec^2(\vartheta) \times  \\
        &\quad \Biggl(
        \frac{1+\theta_1 p^2}{1+\theta_1p^2 +\theta_2w^2} \nabla_p w
        - \frac{\theta_1 p w}{1 + \theta_1p^2 + \theta_2 w^2}
        + \frac{\theta_1 p}{\Lambda} \log(1-\zeta) \Biggr),
        &\qquad & \zeta < F_p(w;w),
        \\[1.2ex]
        &\frac{\theta_2 w^2}{1+\theta_1 p^2+\theta_2 w^2} \nabla_p w
        - \frac{2\theta_1 p}{\Lambda}\log(1-\zeta) + \frac{\theta_1 p w}{1 + \theta_1 p^2 + \theta_2 w^2}
        \\
        &\quad
        - \frac{\theta_1 p}{\sqrt{\theta_2}\sqrt{1+\theta_1 p^2}}
          \arctan\Biggl(\frac{\sqrt{\theta_2}w}{\sqrt{1+\theta_1p^2}}\Biggr),
        &\qquad & \zeta \ge F_p(w;w),
        \end{alignedat}
    \end{cases}
\end{align*}
where
\begin{align*}
    \vartheta := \arctan\Biggl(\frac{\sqrt{\theta_2}w}{\sqrt{1+\theta_1 p^2}}\Biggr) + \frac{\sqrt{\theta_2} \sqrt{1+\theta_1 p^2}}{\Lambda} \log(1-\zeta).
\end{align*}

\end{document}